%% file: dynamic_ustar.tex
\documentclass[reqno]{amsart}
\usepackage[url=false,giveninits=true,maxnames=5,backend=biber]{biblatex}
\makeatletter
\bibliography{sources}
\renewcommand\part{%
	\if@noskipsec \leavevmode \fi
	\par
	\addvspace{4ex}%
	\@afterindentfalse
	\secdef\@part\@spart}

\def\@part[#1]#2{%
	\ifnum \c@secnumdepth >\m@ne
	\refstepcounter{part}%
	\addcontentsline{toc}{part}{\thepart\hspace{1em}#1}%
	\else
	\addcontentsline{toc}{part}{#1}%
	\fi
	{\parindent \z@ \raggedright
		\interlinepenalty \@M
		\normalfont
		\ifnum \c@secnumdepth >\m@ne
		\Large\bfseries \partname\nobreakspace\thepart
		\par\nobreak
		\fi
		\huge \bfseries #2%
		\par}%
	\nobreak
	\vskip 3ex
	\@afterheading}
\def\@spart#1{%
	{\parindent \z@ \raggedright
		\interlinepenalty \@M
		\normalfont
		\huge \bfseries #1\par}%
	\nobreak
	\vskip 3ex
	\@afterheading}
\makeatother
 
\usepackage{color}
\usepackage[dvipsnames]{xcolor}

\usepackage{ifpdf}
\ifpdf 
    \usepackage{graphicx}   
    \usepackage[pdftex,     
            plainpages=false,   
            breaklinks=true,    
            colorlinks=true,
            linkcolor=red,
            citecolor=green,
           ]{hyperref} 
\else 
    \usepackage{graphicx}       
    \usepackage{hyperref}       
\fi 

\usepackage{graphicx}
\usepackage{aurical}
\usepackage{amsfonts,amsmath}
\usepackage{amssymb}
\usepackage{verbatim}
\usepackage{amsopn}
\usepackage[english]{babel}
\usepackage{amsthm}
\usepackage{enumerate}
\usepackage{mathrsfs}
\usepackage{enumitem}
\usepackage{mathtools}
\usepackage{esint}
\usepackage{bbm}
\usepackage{marginnote}
\usepackage[marginparwidth=2cm]{geometry}
\usepackage[normalem]{ulem}
\usepackage{glossaries}
\usepackage{glossary-mcols}
\usepackage{subcaption}
\usepackage{csquotes}

\date{\today}

\theoremstyle{definition} \newtheorem{definition}{Definition}[section]
\theoremstyle{definition} \newtheorem{remark}[definition]{Remark}
\theoremstyle{plain} \newtheorem{lemma}[definition]{Lemma}
\theoremstyle{plain} \newtheorem{proposition}[definition]{Proposition}
\theoremstyle{plain} \newtheorem{theorem}[definition]{Theorem}
\theoremstyle{plain} \newtheorem{corollary}[definition]{Corollary}
\theoremstyle{definition} 
\theoremstyle{plain} 
\theoremstyle{definition} 
\theoremstyle{definition}

\DeclareMathOperator{\diag}{diag}

\DeclareMathOperator{\inter}{int}

\newcommand{\R}{\mathbb{R}}
\newcommand{\Q}{\mathbb{Q}}
\newcommand{\N}{\mathbb{N}}

\newcommand*{\comme}[1]{\textcolor{cyan}{#1}}

\newcommand{\ind}{1\!\!\mathrm{I}}

\numberwithin{equation}{section} 

\theoremstyle{plain} \newtheorem*{theorem*}{Theorem}
\theoremstyle{plain} 
\theoremstyle{plain} \newtheorem*{mthm*}{Main Theorem}
\theoremstyle{plain} \newtheorem*{conjecture*}{Conjecture}
\theoremstyle{plain} 
\theoremstyle{plain} \newtheorem*{problem*}{Problem}

\makeglossaries

\title[Invariant manifolds and transversality for a traveling-wave ODE]{Invariant manifolds and transversality for a traveling-wave ODE in a controlled reaction-diffusion model}

\subjclass[2020]{35K57,49K15,92-10}

\author{Stefano Bianchini} 
\address{S. Bianchini: S.I.S.S.A., via Bonomea 265, 34136 Trieste, Italy}
\email{bianchin@sissa.it}

\author{Chiara Trifone}
\address{Chiara Trifone: S.I.S.S.A., via Bonomea 265, 34136 Trieste, Italy}
\email{ctrifone@sissa.it}

\thanks{Funding: this work has been partially supported by the GNAMPA-INDAM Grant "Problemi di controllo e regolarit\'a per evoluzione di insiemi"}

\begin{document}

\begin{abstract}
We study the parameter dependence and transversality properties of trajectories of the controlled ODE
$$
U_x = P, \qquad P_x = - f(U) - \beta P + \tilde \alpha(x)U,
$$
arising from traveling-wave solutions of a bistable reaction–diffusion equation with control. Optimal profiles minimizing the $L^1$-norm of the control $\tilde \alpha$ are determined by the unstable and stable manifolds $\Gamma_u,\Gamma_s$, the critical curve
$$
P^*(U) = \sqrt{Uf(U)},
$$
and the associated function
$$
F(U) = \frac{dP^*}{dU} + \frac{f(U)}{P^*(U)}.
$$
We establish generic transversality properties for these objects with respect to perturbations of the speed $\beta$ and the nonlinear source $f$, together with quantitative estimates on the first- and second-order dependence of the invariant manifolds on these parameters. We also analyze the loss of regularity at the critical speed $\beta^{**}$, where the unstable manifold approaches the strongly stable manifold of the equilibrium $(U^*,0)$. These results provide the ODE ingredients used in the analysis of the structure and regularity of optimal profiles in paper \cite{bianchini2026travelingprofilescontrolcost}.
%
%
%
\end{abstract}

\keywords{Reaction-diffusion PDE, traveling profiles, perturbations of ODEs, invariant manifolds for ODEs}

\maketitle


\begin{center}
Preprint SISSA 20/2026/MATE
\end{center}

\tableofcontents
%


\section{Introduction}
\label{S:intro}

In this paper we study the properties of the ODE associated to optimal traveling profiles for the controlled reaction-diffusion equation
\begin{equation}
\label{Equa:parab_contr_1}
U_t = \Delta U + f(U) - \tilde \alpha(t,x) U, \quad U \in [0,1], \ \tilde \alpha \geq 0.
\end{equation}
By optimal profile we mean a traveling profile $U(x - \beta t)$ that minimizes the \emph{Effort Function $E(\beta,f) = \|\tilde \alpha\|_{L^1(\R)}$}. 
Our main results concern the transversality and smooth dependence of trajectories and invariant manifolds and some critical curves which are relevant for the optimization problem.

The assumption on the source term \newglossaryentry{fU}{name=\ensuremath{f(U)},description={nonlinear source function}} $\gls{fU}$ is that it is \emph{bistable}, see \eqref{H1}: $f \in C^3([0,1])$ and satisfies (see Fig. \ref{Fig:fU_Gammasu})
\begin{equation}
\label{Equa:H1_intro}
f(0)=f(1)=0, \ f'(0)<0, \ f'(1)<0,
\end{equation}
and there exists a unique point \newglossaryentry{Ustar}{name=\ensuremath{U^*},description={unique $0$ of the source $f$ in $(0,1)$}} $\gls{Ustar} \in (0,1)$ such that $f(U^*)=0$ and $f'(U^*)>0$. 

The PDE with control \eqref{Equa:parab_contr_1} describes a model of pest eradication, which fits into the large subject of optimal harvesting problems and control of parabolic equations \cite{anitacapasso2019,Ania2019GlobalEF,bressancetraro,BressanChiri22,BressanChiri23,doi:10.1137/16M1061886,coclitegaraspino2018,coron2007,seirin2013,faucris.242204001,Lenhart2001OPTIMALCO}.

We do not discuss the modeling aspects in detail here. Instead, we recall the structural properties of optimal profiles that reduce the problem to an ODE with control and motivate the dynamical-systems analysis carried out below.

\subsection{The ODE for optimal traveling profiles}
\label{Ss:opt_trav_intro}

Before stating the main results, we introduce the ODE for traveling profiles, and the properties of optimal profiles. Four objects naturally arise: two invariant manifolds, a critical curve and a function describing the directions in which the free flow crosses this curve.

The ODE associated to a traveling profile with speed \newglossaryentry{beta}{name=\ensuremath{\beta},description={speed of the traveling profile}} \gls{beta}
\begin{equation*}
-\beta U_x = U_{xx} + f(U) - \tilde \alpha(x) U, \quad U(-\infty) = 0, \ U(+\infty) = 1,
\end{equation*}
can be written in the phase space \newglossaryentry{Uvaria}{name=\ensuremath{U},description={variable $U$ in the system of ODEs \eqref{sys}}} \newglossaryentry{Pvaria}{name=\ensuremath{P},description={variable $P$ in the system of ODEs \eqref{sys}}} $(\gls{Uvaria},\gls{Pvaria})$ as
\begin{equation*}
\begin{cases}
U_x = P, \\
P_x = - f(U) - \beta P + \tilde \alpha(x) U,
\end{cases} \quad (U,P)(-\infty) = (0,0), \ (U,P)(+\infty) = (1,0).
\end{equation*}
For $\tilde\alpha = 0$, the two equilibria $(0,0)$ and $(1,0)$ are easily seen to be saddle points, while $(U^*,0)$ is unstable for $\beta < 0$ and stable for $\beta > 0$. We denote by $\lambda_\pm(\beta,U)$, $U = 0,U^*,1$, the eigenvalues of the three equilibria, and let \gls{Gammau} be the \emph{unstable manifold of $(0,0)$} and \gls{Gammas} be the \emph{stable manifold of $(1,0)$}. See Fig. \ref{Fig:fU_Gammasu}.

Using that the optimal profile \gls{Pbeta} has $x \mapsto U(x)$ monotone increasing (see Section \ref{Ss:optim_Pbeta}), we can rewrite the above system as
\begin{equation*}
\frac{dP}{dU} = - \frac{f(U)}{P} - \beta + \alpha(U) U, \quad \alpha(U(x)) = \frac{\tilde \alpha(x)}{P}.
\end{equation*}
Moreover, it is proved in \cite{bianchini2026travelingprofilescontrolcost} that the optimal control acts in the interior \gls{OcalPopt} of the set \newglossaryentry{KcalPopt}{name=\ensuremath{\mathcal K(\beta)},description={critical set, closed set where the optimal profile $P_\beta$ is equal to $P^*$}} \gls{KcalPopt} where $P_\beta$ lies on the \emph{critical curve} 
\begin{equation*}
\gls{Pstar} = \sqrt{U f(U)}, \quad U \in [U^*,1].
\end{equation*}
More precisely \newglossaryentry{OcalPopt}{name=\ensuremath{\mathcal O(\beta)},description={control set, open set where the control $\alpha$ is acting for the optimal profile $P_\beta$, i.e. $\alpha > 0$}}
\begin{equation}
\label{Equa:opt_control_intro}
\alpha(U) > 0 \quad \Longrightarrow \quad U \in \gls{OcalPopt} = \inter \Big\{ U \in (U^*,1): P_\beta(U) = \sqrt{U f(U)} \Big\}.
\end{equation}

With these conditions, the \emph{Effort Function $\gls{Ebeta}$} can be rewritten as \newglossaryentry{Ebeta}{name=\ensuremath{E(\beta)},description={effort function for a traveling wave speed $\gls{beta} \geq \beta^*$}}
\begin{equation*}
\begin{split}
E(\beta) &= \int_\R \tilde \alpha(x) dx 
= \int_{\gls{OcalPopt}} \frac{\frac{dP^*}{dU} + \frac{f(U)}{P^*} + \beta}{U} dU = \int_{\gls{OcalPopt}} \frac{F(U) + \beta}{U} dU,
\end{split}
\end{equation*}
where \newglossaryentry{Fdef}{name=\ensuremath{F(U)},description={function describing the crossing of the flow $P'$ across $P^*$}}
\begin{equation*}
\gls{Fdef} = \frac{dP^*}{dU} + \frac{f(U)}{P^*(U)} = \frac{3 f(U) + U f'(U)}{2 \sqrt{U f(U)}}, \quad U \in [U^*,1].
\end{equation*}
The fact that \gls{Pbeta} can stay on \gls{Pstar} implies that $F(U) + \beta \geq 0$, and then the optimal trajectory must avoid the intervals where
$$
\big\{ U \in [U^*,1], F(U) + \beta < 0 \big\}.
$$
The structure of the above set will be important in what follows. We note here that
\begin{equation}
\label{Equa:limit_intro_F}
\lim_{U \searrow U^*} F(U) = + \infty, \quad \lim_{U \nearrow 1} F(U) = -\infty,
\end{equation}

Observe that from the condition \eqref{Equa:opt_control_intro} implies that the optimal profile \gls{Pbeta} has the following structure (see page \pageref{Point1:why} of Section \ref{Ss:optim_Pbeta} and Fig. \ref{Fig:KOIdef_intro}).

\begin{figure}
\resizebox{.5\textwidth}{!}{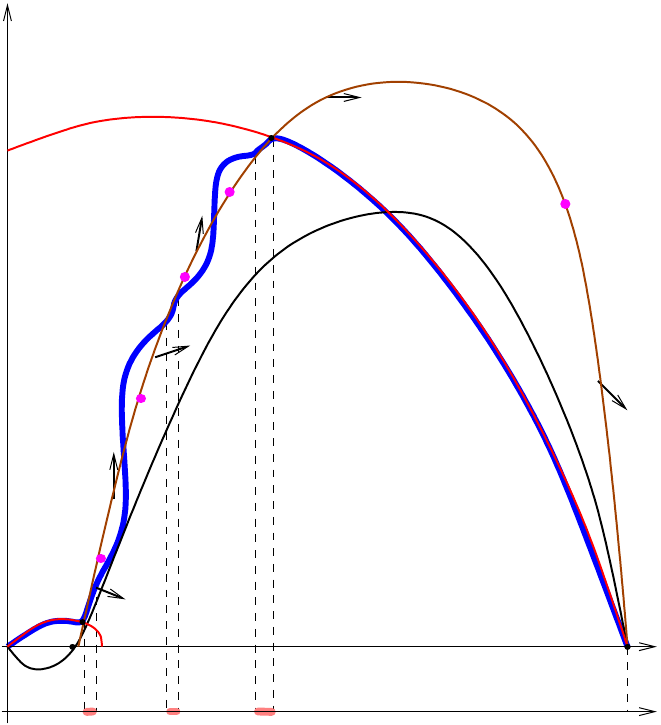}
\caption{The structure of the optimal profile $P_\beta(U)$ (blue) as described in Section \ref{Ss:opt_trav_intro}. In the picture both the nonlinear source $f(U)$ (black) and the critical curve $P^*(U)$ (brown) are depicted. The profile $P_\beta$ cannot remain on the critical curve $P^*$ in the regions where $F(U) + \beta < 0$, i.e. the flux is exiting the region $0 \leq P \leq P^*(U)$: the arrow and the magenta dots denote the regions where $F(U) + \beta$ is positive or negative. The set $\mathcal O(P,\beta)$ is interior of the regions where $P_\beta = P^*$: it is worth noticing that in general this is made of countably many intervals, here just 3 (see Remark \ref{Rem:existence_multiple_inter}).}
\label{Fig:KOIdef_intro}
\end{figure}

\newglossaryentry{Gammau}{name=\ensuremath{\Gamma_u(\beta)},description={unstable manifold of the equilibrium $(U,P) = (0,0)$}} \newglossaryentry{Gammas}{name=\ensuremath{\Gamma_s(\beta)},description={stable manifold of the equilibrium $(U,P) = (1,0)$}}

\begin{itemize}
\item Since it must start at the saddle point $(0,0)$, and the control is not acting if $U \leq U^*$, then the optimal profile \gls{Pbeta} is an arc of $\Gamma_u(U,\beta)$ for $0 \leq U \leq U_u(\beta)$, with \newglossaryentry{Uubeta}{name=\ensuremath{U_u(\beta)},description={(first) intersection of $\Gamma_u(\beta)$ with $P^*$}} $\gls{Uubeta} \geq U^*$ and
$$
P^*(U_u(\beta)) = \Gamma_u(U_u(\beta),\beta).
$$
It corresponds to the first intersection between $\Gamma_u$ and $P^*$.

\item A trajectory solving
\begin{equation}
\label{Equa:ODE_contr_PU_intro}
\frac{dP_\beta}{dU} = - \frac{f(U)}{P_\beta} - \beta + \alpha(U) U,
\end{equation}
for some control $\alpha(U,\beta) > 0$ acting only on $\gls{OcalPopt} \cap \{F(U) + \beta > 0\}$, and $U^* \leq U_u(\beta) < U < U_s(\beta) < 1$ with \newglossaryentry{Usbeta}{name=\ensuremath{U_s(\beta)},description={(last) intersection of $\Gamma_s(\beta)$ with $P^*$}}
$$
P_\beta(U_s(\beta)) = \sqrt{U_s(\beta) f(U_s(\beta))}.
$$
The point $U_s(\beta)$ is the last point less than $U=1$ such that $\Gamma_s(U,\beta) = P^*(U)$, i.e.
$$
U_s(\beta) = \max \, \gls{Cs}= \max \big\{ U < 1: \Gamma_s(U,\beta) = P^*(U) \big\}.
$$
The fact that this point exists is a consequence of \eqref{Equa:limit_intro_F}. Hence in this interval $P_\beta$ is made of countably many open intervals $I_n(\beta) = (U^-_n(\beta),U^+_n(\beta))$ where the control is $0$ and thus $P_\beta$ solves the free ODE
\begin{equation*}
\frac{dP}{dU} = - \frac{f(U)}{P} - \beta, \quad P_\beta(U^\pm_n(\beta)) = P^*(U^\pm_n(\beta)),
\end{equation*}
and a compact set where $P_\beta(U) = P^*(U)$ and the control is positive and equal to $\frac{F(U) + \beta}{U}$.

\item In the interval $[U_s(\beta),1]$ we have $P_\beta(U) = \Gamma_s(U,\beta)$, the stable manifold of $(1,0)$.
\end{itemize}

We conclude that the objects determining the structure and regularity of \gls{Pbeta} and hence the functional form of \gls{Ebeta} are the following:

\begin{enumerate}
\item the unstable manifold \gls{Gammau} of $(0,0)$;
\item the critical function \gls{Pstar};
\item the admissibility function \gls{Fdef};
\item the stable manifold \gls{Gammas} of $(1,0)$.
\end{enumerate}

In the next section of this introduction we state our main results concerning the properties of the above sets: one of the main issues is the transversality of these objects, which yields smooth dependence of the boundary points of the control set $\mathcal O(\beta)$, and hence of the Effort Function $E(\beta)$. 
%

\subsection{Main results}
\label{Ss:main_resu_intro}

The main statements can be divided into results on the \emph{transverse crossing} of $\Gamma_u$, $\Gamma_s$, $F(U)$ and $P^*(U)$, and \emph{smooth dependence} of these crossings with respect to the parameters $\beta,f$.

\subsubsection{Transversal speeds}
\label{Sss:transv_cross_intro}

In the set \gls{Dcal} of admissible parameters $(\beta,f)$,
\begin{equation*}
\gls{Dcal} = \Big\{ (\beta,f) \in \R \times C^3([0,1]), \beta > \beta^*(f), f \ \text{satisfying Assumption \eqref{H1}} \Big\},
\end{equation*}
we are interested in the points ($\beta,f)$ which satisfy the following properties.

\begin{definition}
\label{Def:trans_calT_intro}
Let $f$ be an admissible source satisfying \eqref{Equa:H1_intro}. We say that $\beta > \beta^*$ is \emph{transversal}, and write \newglossaryentry{Tcal}{name=\ensuremath{\mathcal T},description={set of speeds $\beta$ such that the roots of $F(U) + \beta$, $P^*(U) - \Gamma_u(U)$ are finite and not degenerate}} $\beta \in \gls{Tcal} = \mathcal T_f$, if
\begin{enumerate}
\item there exist finitely many $U_0=U^*<U_1<...<U_{2N}=1$ such that
\begin{equation*}
F(U)+\beta \begin{cases} =0 & \text{for }  U=U_i, \ i=1,...,2N-1, \\
>0  & \text{for } U\in P_i:=(U_{2i},U_{2i+1}), \ i=0,...,N-1, \\
<0  & \text{for } U\in N_i:=(U_{2i+1},U_{2i+2}), \ i=0,...,N-1;
\end{cases}
\end{equation*}
\item the tangency points are not degenerate:
\begin{equation*}
(-1)^i F'(U_i) > 0 \quad \forall i=1,..., 2N-1;
\end{equation*}
\item the sets
\begin{equation*}
\gls{Cs} := \big\{ U \in (U^*, 1): \Gamma_s(U,\beta)=P^*(U) \big\} \quad \text{and} \quad \gls{Cu} := \big\{ U\in (U^*, 1): \Gamma_u(U,\beta)=P^*(U) \big\}
\end{equation*}
are finite;
\item the points in $C_s \cup C_u$ correspond to transverse crossings:
\begin{equation*}
F(U) + \beta \neq 0 \quad \forall U \in C_s \cup C_u.
\end{equation*}
\end{enumerate}
\end{definition}

Our main results is the following density property.

\begin{theorem}
\label{Theo:transv_generic_intro}
There is a dense open set \newglossaryentry{Tfrak}{name=\ensuremath{\mathfrak T},description={set of functions for which $\R \setminus \mathcal T$ is finite}} \gls{Tfrak} in the space of admissible sources $f$ w.r.t. the norm $C^3([0,1])$, such that: \newglossaryentry{Tcalf}{name=\ensuremath{\mathcal T_f},description={set of speeds $\beta$ such that the roots of $F(U) + \beta$, $P^*(U) - \Gamma_u(U)$ are finite and not degenerate}}
\begin{enumerate}
\item it contains a dense set of polynomials;
\item for all $f \in \mathfrak T$ there exists a finite sets \gls{Ncal0} (concerning $F(U)$) and \gls{Ncal1} (concerning $\Gamma_u,\Gamma_s$) such that the transversality set \gls{Tcalf} of $f$ is
$$
\gls{Tcalf} = [\beta^*,\infty) \setminus (\mathcal N_0 \cup \mathcal N_1);
$$
\item the set \newglossaryentry{Dcal}{name=\ensuremath{\mathcal D},description={set of admissible sources and admissible speed speeds $\beta$}}
\begin{equation*}
\bigcup_{f \in \mathfrak T} \mathcal T_f \times \{f\} \subset \gls{Dcal} = \Big\{ (\beta,f) \in \R \times C^3([0,1]), \beta > \beta^*(f), f \ \text{satisfying Assumption \eqref{H1}} \Big\}
\end{equation*}
is open and dense in the product topology, and of full measure in the subset \newglossaryentry{Dcalk}{name=\ensuremath{\mathcal D(k)},description={domain of admissible polynomials $f = \sum_{i=1}^k a_i U^i$ and admissible speeds $\beta$}}
\begin{equation*}
\gls{Dcalk} = \bigg\{ (\beta,a_1,\dots,a_k) \in \R^{k+1}, f = \sum_{i=1}^K a_i U^i,\ \beta \geq \beta^*(f), \ f \ \text{satisfies Assumption (\ref{H1})} \bigg\}.
\end{equation*}
\end{enumerate}
\end{theorem}

It is worth noticing that the example of Remark \ref{Rem:existence_multiple_inter} yields a source function $f$ and speed $\beta$ such that $F(U) + \beta > 0$ is made by countably disjoint intervals, in particular the above statement is meaning.

\subsubsection{Dependence of $\Gamma_u$, $\Gamma_s$ w.r.t $\beta,f$}
\label{Sss:dep_betaf_intro}

The next set of results concerns the dependence of the invariant manifolds $\Gamma_u$, $\Gamma_s$ w.r.t. $\beta,f$: as a consequence of the structure of the optimal profile $P_\beta$, it is particularly interesting the behavior of the intersection points \gls{Uubeta}, \gls{Usbeta}. The critical value of the speed is \newglossaryentry{betastarstar}{name=\ensuremath{\beta^{**}},description={first value of $\beta$ for which $U_{uf}(\beta) = U^*$}} $\beta = \gls{betastarstar}$, corresponding to a heteroclinic connection between $\Gamma_u$ and the strongly stable manifold of $(U^*,0)$ (Remark \ref{Rem:ether_1}): the change in the dynamics leads to a discontinuity in the second derivative of \gls{Ebeta}, as shown in \cite{bianchini2026travelingprofilescontrolcost}.

Let \gls{Uufbeta} be such that $\Gamma_u(U) > 0$ for $U \in (0,U_{uf}(\beta))$.
 
\begin{theorem}
\label{manifoldbound_intro}
Let $\beta \ge \beta^*$ and $f \in C^2([0,1])$ satisfying hypothesis $(\ref{H1})$.

\begin{description}

\item[Derivatives of $\Gamma_u$ w.r.t. $\beta$] we have
\begin{equation*}
- 1 - \bigg[ \ln \bigg(\frac{U}{U^*} \bigg) \bigg]^+ < \frac{\partial_\beta \Gamma_u(U,\beta)}{U} < 0, \quad U \in (0,U_u(\beta)).
\end{equation*}
For all $\beta \leq \underline{\beta} < \beta^{**}$ there is a constant $C(f,\underline{\beta})$ depending only on
\begin{equation*}
\begin{split}
&\|f\|_{C_2}, \quad \max \Big\{ \Gamma_u(U)^{-1}, \beta \in [\beta^*,\underline{\beta}], U \in [\bar r,U_u(\beta)] \Big\} 
\end{split}
\end{equation*}
such that
\begin{equation*}
\big| \partial^2_\beta \Gamma_u(U,\beta) \big| \leq C(f,\underline{\beta}), \quad U \in [0,U_u(\beta)].
\end{equation*}

\item[Derivatives of $\Gamma_s$ w.r.t. $\beta$] we have
\begin{equation*}
0 < \frac{\partial_\beta \Gamma_s(U,\beta)}{1 - U} < 1, \quad U \in [U^*,1).
\end{equation*}
Moreover, there is a constant depending only on $\|f\|_{C^2}$ such that
\begin{equation*}
0 \leq \partial^2_\beta \Gamma_s(U) \leq C(\|f\|_{C^2}), \quad U \in [U^*,1].
\end{equation*}

\item[Derivatives of $\Gamma_u,\Gamma_s$ w.r.t. $f$] there are two constants \newglossaryentry{rbar}{name=\ensuremath{\underline{r}},description={radius of the ball around the equilibria $(0,0)$, $U^*,0)$, $(1,0)$ where the perturbative analysis of the ODE is performed}} \newglossaryentry{betaunder}{name=\ensuremath{\underline{\beta}},description={speed where $U_u$ is sufficiently close to $U^*$}}
\begin{equation*}
\gls{rbar} = \underline{r}(\|f\|_{C^2}), \quad \gls{betaunder} = \underline{\beta}(\|f\|_{C^2}) < \beta^{**}(f)
\end{equation*}
and a constant \newglossaryentry{Cfbeta}{name=\ensuremath{C(f)},description={Lipschitz constant for the dependence of $\Gamma_u,\Gamma_s$ w.r.t. $f \in C^1([0,1])$}} $\gls{Cfbeta} > 0$ depending only on
\begin{equation*}
\begin{split}
&\|f\|_{C_2}, \quad \max \Big\{ \Gamma_u(U)^{-1}, \beta \in [\beta^*,\underline{\beta}], U \in [\bar r,U_u] \Big\}, \\
&\max \bigg\{ \Gamma_u(U)^{-1}, \beta \in \Big[ \beta^{*},2 \sqrt{\max_{[0,U^*]} f'(U)} \Big], U \in [\underline{r},U_u - \underline{r}] \bigg\},
\end{split}
\end{equation*}
such that if $\Gamma_u(U,f + \epsilon \delta f)$, $\Gamma_s(U,f+\epsilon \delta f)$ are the unstable, stable manifold for the source $f + \epsilon \delta f$ satisfying \eqref{H1} with $\delta f \geq 0$, then 
\begin{equation*}
0 \le - \lim_{\epsilon \to 0} \frac{\Gamma_u(U,f+\epsilon \delta f) - \Gamma_u(U,f)}{U} \le C(f) \|\delta f\|_{C^1},
\end{equation*}
for all $U \leq U_u(f)$, and
\begin{equation*}
0 \le \lim_{\epsilon \to 0} \frac{\Gamma_s(U,f+\epsilon \delta f) - \Gamma_s(U,f)}{1 - U} \le C(f) \|\delta f\|_{C^1},
\end{equation*}
for all $U \in (U^*,1)$.
\end{description}
\end{theorem}

The restriction to speeds bounded away from $\beta^{**}$ is necessary: although first-order continuity is preserved across the critical speed, the second derivative may become unbounded.

\begin{theorem}
\label{Theo:disco_cont_intro}
The map
\begin{equation*}
f \mapsto \partial_\beta \Gamma_u(U_u(\beta,f),\beta,f)
\end{equation*}
is continuous in the $C^2$-topology of $f$. The map
\begin{equation*}
f \mapsto \Gamma_u(U_u(\beta,f),\beta,f)
\end{equation*}
is locally Lipschitz but in general not $C^1$.

If the ratio between the two eigenvalues $\lambda_\pm(\beta^{**},U^*)$ of $(U^,0)*$ satisfies 
\begin{equation*}
\frac{\lambda_+(\beta^{**},U^*)}{\lambda_-(\beta^{**},U^*)} < \frac{1}{2},
\end{equation*}
then
\begin{equation*}
\lim_{\beta \nearrow \beta^{**}} \partial^2_\beta \Gamma_u(\beta,U_u(\beta)) = +\infty.
\end{equation*}
\end{theorem}

The quantity $\partial_\beta \Gamma_u(U_u(\beta,f),\beta,f)$ enters in the computation of the first derivative $\partial_\beta E(\beta)$. Similarly the quantity $\partial^2_\beta \Gamma_u(U_u(\beta,f),\beta,f)$ enters in the computation of $\partial^2 E(\beta)$. Hence the above statement will imply that the Effort Function \gls{Ebeta} is only $C^1$ across the critical speed \gls{betastarstar}.

The Lipschitz dependence w.r.t. $f$ is key to study the problem in the simplified setting of $(\beta,f) \in \gls{Tfrak}$. In this settings the derivatives of $E$ can be computed explicitly, and in particular the convexity of the Effort Function can be analyzed. This convexity property is essential when studying the l.s.c. envelope of the cost functional
\begin{equation*}
\int_0^T \phi \bigg( \int_{\R^d} \tilde \alpha(t,x) dx \bigg) dt. 
\end{equation*}
For a justification of this fact, we refer to the results of \cite{BressanChiri23,bianchini2026travelingprofilescontrolcost}.

\subsection{Outline of the paper}
\label{Ss:outline_intro}

The paper is organized as follows.

Section \ref{S:struct_opt} serves as an introduction to the main results for optimal traveling profiles, which in turn justifies the analysis of the next two sections: it can be thought as an extended introduction to the structure of optimal profiles, presented in Section \ref{Ss:opt_trav_intro}. 

Section \ref{S:transver_flux} addresses the transversality properties of $F$ and the manifolds $\Gamma_u$ and $\Gamma_s$, as defined in Definition \ref{Def:trans_calT_intro}. The last result of the section is the proof of Theorem \ref{Theo:transv_generic_intro}. \\
First, in Proposition \ref{zeros}, via an almost immediate application of Sard's Lemma, we show  that the set of \emph{critical velocities}, that is those values of $\beta$ that do not satisfy the transversality condition 
\begin{equation*}
\gls{Fdef} + \beta = 0 \quad \Rightarrow \quad F'(U) \not= 0, 
\end{equation*}
is compact with zero Lebesgue measure for every smooth source $f$. Subsequently, with Proposition \ref{Prop:F_transv_poly}, we make a substantial improvement of the above result by restricting ourselves to the family of polynomial sources. 
 \newglossaryentry{NTFk}{name=\ensuremath{\mathcal{NT}_k},description={non-transversality set of polynomials of degree $k$}} 
\\
Next we address the problem of transversality of \newglossaryentry{Cs}{name=\ensuremath{C_s},description={intersection points of $\Gamma_s$ with $P^*$}} \newglossaryentry{Cu}{name=\ensuremath{C_u},description={intersection points of $\Gamma_u$ with $P^*$}} \gls{Gammau}, \gls{Gammas} and \gls{Pstar}. 
In the proof of Proposition \ref{Prop:transv_Gamma_gen}, we exploit the Implicit Function Theorem to show that for $\beta$ outside of a compact set of zero Lebesgue measure, \gls{Cu} and \gls{Cs} are finite and the intersections are always \emph{transverse}, that is
\begin{equation*}
U \in C_u\cup C_s \quad \Rightarrow \quad F(U)+\beta \neq 0.
\end{equation*} 
As in the above, the result can be strengthened considering polynomial sources. In Proposition \ref{Prop:Gamma_trans_poly} we prove that if $f$ is a polynomial, the cardinality of $C_u(\beta)$ and $C_s(\beta)$ are uniformly bounded in $\beta$, and, in this case, the family of velocities for which a tangential crossing with $P^*$ happens either for $\Gamma_u$ or for $\Gamma_s$ is finite.\\
In Section \ref{Ss:transc_calT} we introduce the family of \emph{transversal velocities} for a source $f$, that is  \gls{Tcal}$=\mathcal{T}_f$, as in Definition \ref{Def:trans_calT_intro} (see also Definition \ref{Def:trans_calT}). 
Collecting the results previously obtained throughout this section, we are able to prove Theorem \ref{Theo:transv_generic_intro}. 
In Remark \ref{Rem:existence_multiple_inter} we provide the explicit example of a source for which the set $\{F(U) + \beta > 0\}$ consists of countably many connected components, thus also implying that $\mathfrak T$ is strictly contained in the family of all admissible sources.

Section \ref{Ss:stab_sti_Gamma_us} is devoted to the proof of Theorem \ref{manifoldbound_intro} and Theorem \ref{Theo:disco_cont_intro}, i.e. the analysis of the stability of $\Gamma_u$ and $\Gamma_s$ both with respect to the velocity $\beta$ and the source $f$. \\
In Theorem \ref{manifoldbound}, we provide explicit bounds for the derivative of the manifolds, both with respect to $\beta$ (Estimates \eqref{Equa:Gammau_beta}, \eqref{Equa:Gammas_beta}) and with respect to $f$ (Estimates \eqref{Equa:estimat_f_uGG}, \eqref{Equa:estimat_f_sGG}).
In Corollary \ref{Cor:Frechet_Gamma}, we sum up what we have proved above to show Fr\'echet  differentiability of the maps $(U,\beta,f) \mapsto \Gamma_u(U,\beta,f)$ and $(U,\beta,f) \mapsto \Gamma_s(U,\beta,f)$ outside the critical value $\beta^{**}$. This concludes the proof of Theorem \ref{manifoldbound_intro} concerning the first order derivatives. \\
Next, in Theorem \ref{Theo:second_der_Gammaus}, we prove that the second derivative $\partial_\beta^2 \Gamma_u(U,\beta)$ exists in $[0,\gls{Uubeta}]$ when $\beta \not= \beta^{**}$, while $\partial_\beta^2 \Gamma_s(U,\beta)$ exists in $[0,1]$. Moreover, defining \newglossaryentry{Ubarubeta}{name=\ensuremath{\overline{U}_u(\beta)},description={last point of intersection of $\gamma_u(\beta)$ with $P^*$}} \gls{Ubarubeta} as the last intersection of \gls{Gammau} and \gls{Pstar}, in Corollary \ref{Cor:bound_hatu} we extend the estimate for $\partial_\beta^2\Gamma_u$ to the whole interval $[0,\overline U_u(\beta)]$. This concludes the part of Theorem \ref{manifoldbound_intro} concerning the second derivatives away from $\beta^{**}$. \\
We remark that in order to bound the second derivative of $\Gamma_u$ with respect to $\beta$, we must make the further assumption that $\beta$ is uniformly far from the critical speed $\beta^{**}$ Proposition \ref{Prop:blowupsecond} and Proposition \ref{Prop:Lipechit_U_u} show that these regularity estimates are sharp at the critical speed $\beta$.
%

\begin{remark}
\label{Rem:notation}
A glossary of notations is at end of the paper: we have omitted standard notations for natural, real numbers, etc. Since the functions considered may depend on several variables, in order to avoid formulas too cumbersome we omit the dependence from the variable which do not play a role in the computation: this should not create confusion.
\end{remark}

\section{Existence and structure of optimal profiles}
\label{S:struct_opt}

Aim of this section is to recall some basic facts about traveling profiles for parabolic reaction-diffusion equations. For all classical results regarding traveling wave solutions for parabolic reaction-diffusion equations we refer to \cite[Chapter 4]{Fife}. 

Two critical speeds play a distinguished role. The first is the control-free speed $\beta^*$, corresponding to the heteroclinic connection between $(0,0)$ and $(1,0)$. The second is the critical speed $\beta^{**}$, at which the unstable manifold of $(0,0)$ connects to the strongly stable manifold of $(U^*,0)$. For $\beta>\beta^{**}$, no optimal traveling profile exists, although the optimal effort can be approached arbitrarily well by traveling profiles.

\subsection{Traveling waves with no control}
\label{Ss:travelin_no_alpha}

Consider the PDE \eqref{Equa:parab_contr_1} without control
\begin{equation}
\label{eq}
U_t = U_{xx} + f(U).
\end{equation}
As in \cite{BressanChiri22} and \cite{BressanChiri23} we make the following assumptions on the \emph{source function}
$\gls{fU}$:
\begin{enumerate}[label=\textbf{(H\arabic*)}, ref=H\arabic*]
\item\label{H1} The function $f \in C^3([0,1])$ satisfies (see Fig. \ref{Fig:fU_Gammasu})
\begin{equation*}
f(0)=f(1)=0, \ f'(0)<0, \ f'(1)<0,
\end{equation*}
and there exists a unique point
$\gls{Ustar} \in (0,1)$ such that $f(U^*)=0$ and $f'(U^*)>0$.
\end{enumerate}
It is also natural to require that $f$ is \emph{positively predominant}, that is
\begin{equation*}
\int_0^1 f(U) \, dU > 0.
\end{equation*}
However this conditions does not play any role in the analysis carried out in this paper.

A \emph{traveling wave solution with velocity
$\gls{beta}$} (or \emph{traveling profile}) for \eqref{eq} is a solution of the form
\begin{equation*}
U(t,x) = U(x-\beta t)
\end{equation*}
connecting $U(-\infty) = 0$ to $U(+\infty) = 1$. If we consider the  system in the phase plane
$(\gls{Uvaria},\gls{Pvaria})$
\begin{align}
\label{sys}
\begin{cases}
U' = P, \\
P' = - f(U) -\beta P,
\end{cases}
\end{align}
a traveling profile corresponds to a trajectory in the $(U,P)$-plane connecting the equilibria $(0,0)$ and $(1,0)$ with $P > 0$.

Linearizing the system at a point $(U,P)$ we obtain the Jacobian matrix
\begin{equation*}
D(U,P)=\begin{pmatrix}
0 & 1 \\
- f'(U) & - \beta
\end{pmatrix},
\end{equation*}
whose eigenvalues are \newglossaryentry{lambdapm}{name=\ensuremath{\lambda_{\pm}(U,\beta)},description={eigenvalues of the linearized system \eqref{sys}}}
\begin{align}
\label{Equa:eigenv_equi}
\gls{lambdapm} := - \frac{\beta}{2} \pm \sqrt{\frac{\beta^2}{4}-f'(U)}.
\end{align}
We can therefore distinguish the following cases with respect to the stability of the fixed points $(0,0)$, $(U^*,0)$, and $(1,0)$.

\begin{enumerate}
\item When $f'(U)<0$, the eigenvalues are real and have opposite signs for all $\beta \in \R$: in particular the equilibria $(0,0)$ and $(1,0)$ are saddles.

\item If $f'(U)>0$ and $\beta <- 2 \sqrt{f'(U)}<0 $, the eigenvalues are both real and positive: the equilibrium $(U^*,0)$ is an unstable node.

\item If $f'(U)>0$ and $- 2\sqrt{f'(U)}<\beta < 0$, $\lambda_{\pm}(U,\beta)$ are complex conjugate, both with positive real part: the equilibrium $(U^*,0)$ is an unstable focus.

\item For $\beta = 0$ the system is Hamiltonian, so that $(U^*,0)$ is a center.

\item \label{Point_5:eigen_equi} If $f'(U)>0$ and $0 < \beta < 2\sqrt{f'(U)}$, $\lambda_{\pm}(U,\beta)$ are complex conjugate, both with negative real part: the equilibrium $(U^*,0)$ is a stable focus.

\item If $f'(U)>0$ and $\beta > 2\sqrt{f'(U)}>0 $, the eigenvalues are both real and negative: the equilibrium $(U^*,0)$ is a stable node.
\end{enumerate}

Therefore, for $P>0$ the phase portrait for the system \eqref{sys} contains an unstable manifold $\gls{Gammau} = \Gamma_u(U,\beta,f)$ of the point $(0,0)$ and a stable manifold $\gls{Gammas} = \Gamma(U,\beta,f)$ of the point $(1,0)$, which are characterized by
\begin{equation}
\label{Equa:Gammau_eq0}
\frac{d \Gamma_u}{dU} = - \frac{f(U)}{\Gamma_u} - \beta, \quad \lim_{U \to 0} \frac{\Gamma_u(U)}{U} = \lambda_+(0,\beta),
\end{equation}
\begin{equation}
\label{Equa:Gammas_eq0}
\frac{d \Gamma_s}{dU} = - \frac{f(U)}{\Gamma_s(U)} - \beta, \quad \lim_{U \to 1} \frac{\Gamma_s(U)}{1-U} = - \lambda_-(1,\beta).
\end{equation}
The limits as $U \to 0$, $U \to 1$ are required by the property of being tangent to the unstable, stable eigenspace of the equilibrium $U=0$, $U=1$ respectively.

Moreover (\cite[Chapter 4.4.a]{Fife}) there exists a unique value \newglossaryentry{betastar}{name=\ensuremath{\beta^*},description={speed for traveling profiles without control $\alpha$}} $\gls{betastar} = \beta^*(f)$ (see Remark \ref{Rem:ether_1}) such that $\Gamma_u(\beta^*)=\Gamma_s(\beta^*)$: thus, we can find a unique (up to translation) traveling profile $U$ with speed $\beta^*$. 

We write $\Gamma_u(\beta)$ as a function $\Gamma_u(U,\beta)$ in the region \newglossaryentry{Uufbeta}{name=\ensuremath{U_{uf}(\beta)},description={first point where $\Gamma_u(U,\beta) = 0$}} $[0,\gls{Uufbeta} = U_{uf}(\beta,f)]$ where $\Gamma_u(\beta) \geq 0$, with $\Gamma_u(0,\beta) = \Gamma_u(U_{uf}(\beta),\beta)) = 0$ and strictly positive inside the interval; similarly we write $\Gamma_s(\beta) = \Gamma_s(U,\beta)$, which it is easy to see that it is a function for $\beta \geq \beta^*$.

Point \eqref{Point_5:eigen_equi} above and the fact that for $\beta \to \infty$ we have $\Gamma_u(U,\beta) \to 0$ yields the existence of a value $\gls{betastarstar} = \beta^{**}(f) \geq 2\sqrt{f'(U^*)}$ such that for $U_{uf}(\beta^{**}) = U^*$. From this value on $U_{uf}(\beta) = U^*$.

\begin{remark}
\label{Rem:ether_1}
The value of $\beta^{**}$ corresponds to the "heteroclinic connection" where $\Gamma_u$ coincides with the "strongly" stable manifold defined by the eigenvalue $\lambda_-(U^*,\beta)$.
\end{remark}

See Fig. \ref{Fig:fU_Gammasu} for a graphical representation of the manifolds $\Gamma_u(\beta),\Gamma_s(\beta)$ for different values of $\beta$. 

\begin{figure}
\resizebox{\textwidth}{!}{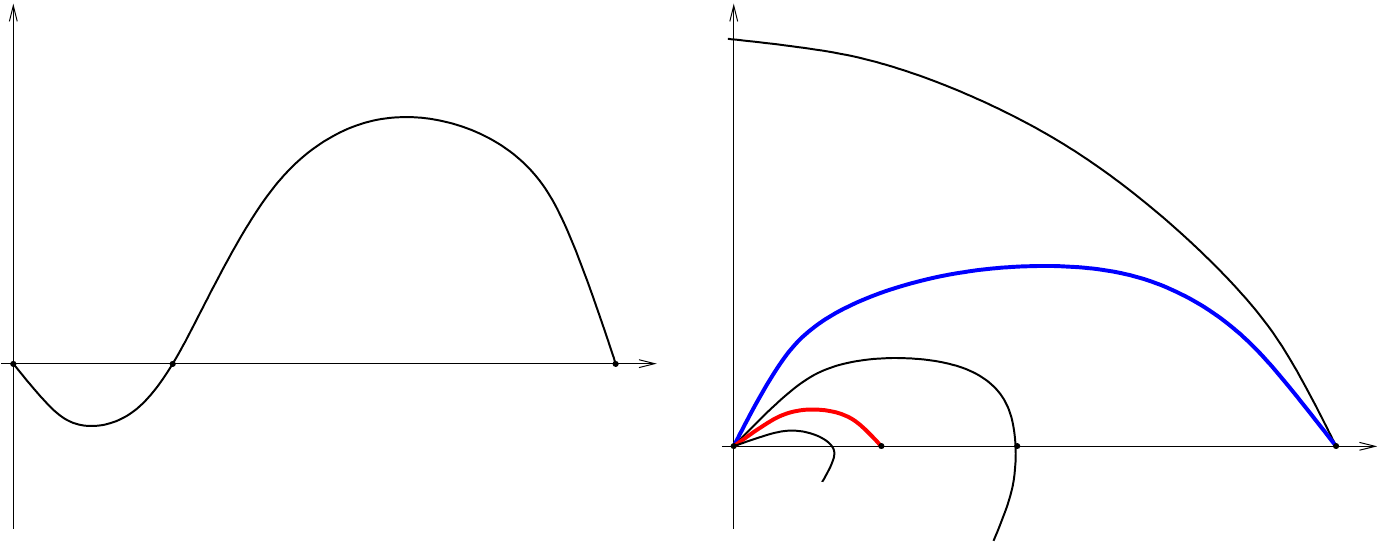}
\caption{The plot of the source $f(U)$ and of the stable and unstable manifolds for different speeds $\beta$, in the case the control $\alpha = 0$. The two critical speed $\beta^*$ (minimal speed to have a traveling profile, blue) and $\beta^{**}$ (heteroclinic connection, red) are highlighted.}
\label{Fig:fU_Gammasu}
\end{figure}

\subsection{Optimal traveling profiles for $\beta > \beta^*$}
\label{Ss:optim_Pbeta}

For $\beta > \beta^*$ we need to introduce a control $\alpha$ in order to find a traveling solution to the parabolic PDE
\begin{equation}
\label{Equa:trav_pro_a>0}
U_t = \Delta U + f(U) - \tilde \alpha U, \quad U(t,x) = U(x - \beta t).
\end{equation}
Rewritten in the phase plane, the above ODE becomes
\begin{equation*}
\begin{cases}
U_x = P, \\
P_x = - f(U) - \beta P + \tilde \alpha(x) U.
\end{cases}
\end{equation*}
If $x \mapsto U(x)$ is monotone increasing, then we can rewrite the above system as
\begin{equation*}
\frac{dP}{dU} = - \frac{f(U)}{P} - \beta + \alpha(U) U, \quad \alpha(U(x)) = \frac{\tilde \alpha(x)}{P}.
\end{equation*}

In \cite{bianchini2026travelingprofilescontrolcost} it is shown that in the phase plane $(U,P)$, for any fixed $\beta \geq \beta^*$ there is a unique optimal profile \newglossaryentry{Pbeta}{name=\ensuremath{P_\beta},description={the unique optimal profile for a given velocity $\beta$}} $\gls{Pbeta} = P_\beta(U)$ have the following structure:
\begin{enumerate}
\item an arc of $P =  \gls{Gammau} = \Gamma_u(U,\beta)$ for $0 \leq U \leq U_u(\beta)$, with $U_u(\beta) \geq U^*$;
\item a trajectory solving
\begin{equation}
\label{Equa:ODE_contr_PU}
\frac{dP_\beta}{dU} = - \frac{f(U)}{P_\beta} - \beta + \alpha(U) U,
\end{equation}
for some control $\alpha(U,\beta) \geq 0$, and $U^* \leq U_u(\beta) < U < U_s(\beta) < 1$ with
$$
P_\beta(U_u(\beta)) = \sqrt{U_u(\beta) f(U_u(\beta))}, \quad P_\beta(U_s(\beta)) = \sqrt{U_s(\beta) f(U_s(\beta))};
$$
\item an arc of $P = \Gamma_s(U)$ for $U_s \leq U \leq 1$.
\end{enumerate}
We observe that that when we write \eqref{Equa:ODE_contr_PU} we assume that optimal trajectories can be represented as functions $P = P(U)$, which is not a priori obvious considering the system \eqref{sys}. This is proved in \cite{bianchini2026travelingprofilescontrolcost}.

In Fig. \ref{Fig:KOIdef} a typical situation is depicted: defining the critical curve \newglossaryentry{Pstar}{name=\ensuremath{P^*(U)},description={curve of minimality}}
\begin{equation*}
\gls{Pstar} = \sqrt{U f(U)}, \quad U \in [U^*,1],
\end{equation*}
the control $\alpha$ is positive only in the (three, in this case) connected components of the interior \gls{OcalPopt} of
$$
\gls{KcalPopt} = \big\{ U : P_\beta(U) = P^*(U) \big\}, \quad \gls{OcalPopt} = \inter \, \gls{KcalPopt},
$$
while in the open region \newglossaryentry{IcalPopt}{name=\ensuremath{\mathcal I(\beta)},description={free set, open set where $\alpha = 0$ for  the optimal profile $P_\beta$}}
$$
\gls{IcalPopt} = \big\{ U : P_\beta \not= P^*(U) \big\}
$$
the trajectories are solution to \eqref{sys}, i.e. $\alpha = 0$: rewritten for $P = P(U)$ we have
\begin{equation}
\label{Equa:ODE_nocontr_PU}
\frac{dP_\beta}{dU} = - \frac{f(U)}{P_\beta} - \beta.
\end{equation}
 The critical curve $P^*(U)$ can be loosely speaking thought as the curve with minimal cost, which we would like to stay on as much as possible.

\begin{figure}
\resizebox{.5\textwidth}{!}{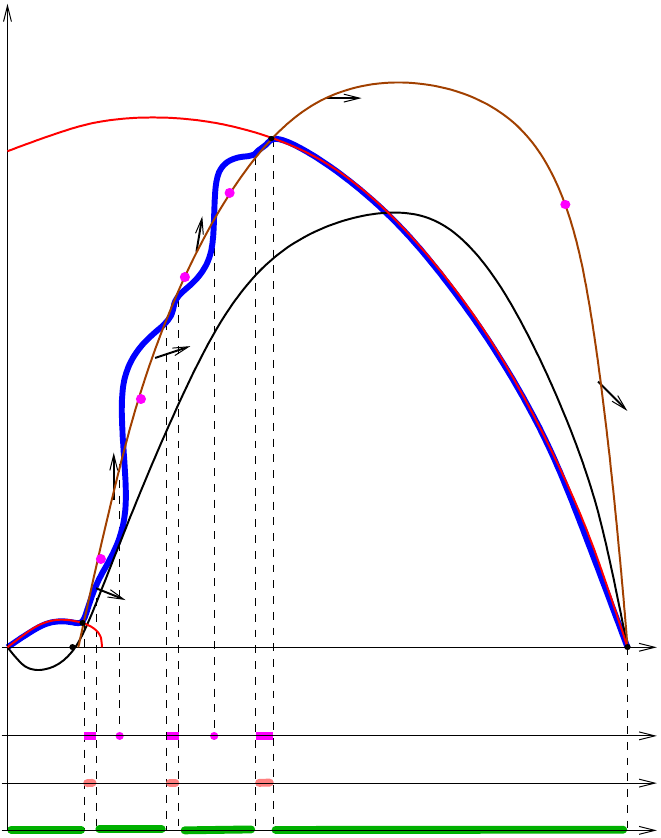}
\caption{A generic optimal profile $P_\beta(U)$ (blue). In the picture both the nonlinear source $f(U)$ (black) and the critical curve $P^*(U)$ (brown) are depicted. The profile $P_\beta$ cannot remain on the critical curve $P^*$ in the regions where $F(U) + \beta < 0$, i.e. the flux is exiting the region $0 \leq P \leq P^*(U)$: the arrow and the magenta dots denote the regions where $F(U) + \beta$ is positive or negative. The sets $\mathcal K(P,\beta)$, $\mathcal O(P,\beta)$, $\mathcal I(P,\beta)$ are the values of $U$ where $P_\beta = P^*$ ($\mathcal K$), its interior ($\mathcal O$) and the free region $\mathcal I$ where no control is acting, see \cite{bianchini2026travelingprofilescontrolcost}.}
\label{Fig:KOIdef}
\end{figure}

The \emph{Effort Function $\gls{Ebeta} = E(\beta,f)$} is defined as the $L^1$-norm of the control $\alpha$:
\begin{equation*}
E(\beta) = \int_0^1 \alpha(U) dU.
\end{equation*}
It is fairly easy to see that the following expression of $E$ is equivalent:
\begin{equation*}
E(\beta) = \int_\R \tilde \alpha(x) dx, \quad \text{because} \quad \alpha(U) dU = \alpha(U) P(U) dU = \tilde \alpha(x) dx.
\end{equation*}
Using the ODE \eqref{Equa:ODE_contr_PU} and the definition of the set \gls{OcalPopt}, we can write
\begin{equation*}
\begin{split}
E(\beta) &= \int_0^1 \bigg[ \frac{\frac{dP_\beta}{dU} + \frac{f(U)}{P_\beta} + \beta}{U} \bigg]^+ dU \\
&= \int_{\gls{OcalPopt}} \frac{\frac{dP_\beta}{dU} + \frac{f(U)}{P_\beta} + \beta}{U} dU \\
&= \int_{\gls{OcalPopt}} \frac{\frac{dP^*}{dU} + \frac{f(U)}{P^*} + \beta}{U} dU = \int_{\gls{OcalPopt}} \frac{F(U) + \beta}{U} dU,
\end{split}
\end{equation*}
where we have defined the function 
\begin{equation}
\label{F}
\gls{Fdef} = \frac{dP^*}{dU} + \frac{f(U)}{P^*(U)} = \frac{3 f(U) + U f'(U)}{2 \sqrt{U f(U)}}.
\end{equation}
The fact that boundary points of \gls{KcalPopt} have $F(U) + \beta = 0$ for $\mathscr L^1$-a.e. $U \in [0,1]$ is proved in \cite{bianchini2026travelingprofilescontrolcost}: it is based on the observation that up to a countable set these points are continuity point of $F(U)$.

Denote the connected components \newglossaryentry{Oomega}{name=\ensuremath{O_\varpi},description={connected components of $\mathcal O(P,\beta)$}} $\gls{Oomega} = (U^-_\varpi,U^+_\varpi)$, $\varpi \in \Q$, as the connected components of \gls{OcalPopt}, ordered so that
\begin{equation*}
\varpi < \varpi' \quad \Longrightarrow \quad U^+_\varpi < U^-_{\varpi'}.
\end{equation*}
Then we obtain
\begin{equation*}
\gls{Ebeta} = \sum_\varpi \int_{U^-_\varpi(\beta)}^{U^+_\varpi(\beta)} \frac{F(U) + \beta}{U} dU.
\end{equation*}
It is thus natural to see that in order to study the structure of $P_\beta$ and $E(\beta)$ and their dependence on the quantities $\beta,f$, one has to study how \gls{Oomega} change as $\beta,f$ vary: more precisely,
\begin{itemize}
\item \label{Point1:why} how the intersection sets
\begin{equation}
\label{Equa:inter_Gammaus}
\big\{ U : \Gamma_u(U,\beta) = P^*(U) \big\}, \quad \big\{ U : \Gamma_s(U,\beta) = P^*(U) \big\}
\end{equation}
behave as the parameter $\beta$ (and also the source $f$) of the PDE \eqref{Equa:ODE_contr_PU} varies, i.e. their transverse intersection with the critical curve \gls{Pstar};

\item \label{Point2:why} the transversality properties of the vector field of the ODE \eqref{Equa:ODE_contr_PU} with zero control $\alpha = 0$, namely
\begin{equation*}
(U,P) \mapsto \bigg( 1, - \frac{f(U)}{P_\beta} - \beta \bigg)
\end{equation*}
w.r.t. the critical curve $P^*(U) = \sqrt{U f(U)}$;

\item \label{Point3:why} the behavior of trajectories near the critical value \gls{betastarstar}, corresponding to the heteroclinic connection of $(0,0)$ to the strongly stable manifold of $(U^*,0)$.
\end{itemize}
These properties are fundamental for the analysis of the differentiability (in the Fr\'echet sense) of the optimal profile \gls{Pbeta} and its effort function \gls{Ebeta}: indeed the first and last points of $\partial \mathcal O(\beta)$ belong to the sets \eqref{Equa:inter_Gammaus}, justifying the interests \eqref{Point1:why}, and Point \eqref{Point2:why} allows to prove that the intervals of \gls{OcalPopt} are finite and depends smoothly on $\beta,f$. This situation is studied in Section \ref{S:transver_flux}.

About Point \eqref{Point3:why}, observe that the optimal profile \gls{Pbeta} does not correspond to any traveling wave $U(x-\beta t)$ for \eqref{Equa:trav_pro_a>0}, since
\begin{equation*}
\int_{\frac{U^*}{2}}^{U^*} \frac{1}{P_\beta(U)} dU = +\infty.
\end{equation*}
See Fig. \ref{Fig:optproof_2} for a picture of how the optimal profile behaves as $\beta$ varies. However it can be approximated with traveling profiles, so that the effort $E(\beta)$ is the infimum of the costs for traveling profiles. Due to this change of structure, it is necessary to study if the intersection of $\Gamma_u(\beta)$ and $P^*(\beta)$ behaves smoothly w.r.t. $\beta,f$. This is studied in Section \ref{Ss:stab_sti_Gamma_us}.

\begin{figure}
\centering
\resizebox{.75\textwidth}{!}{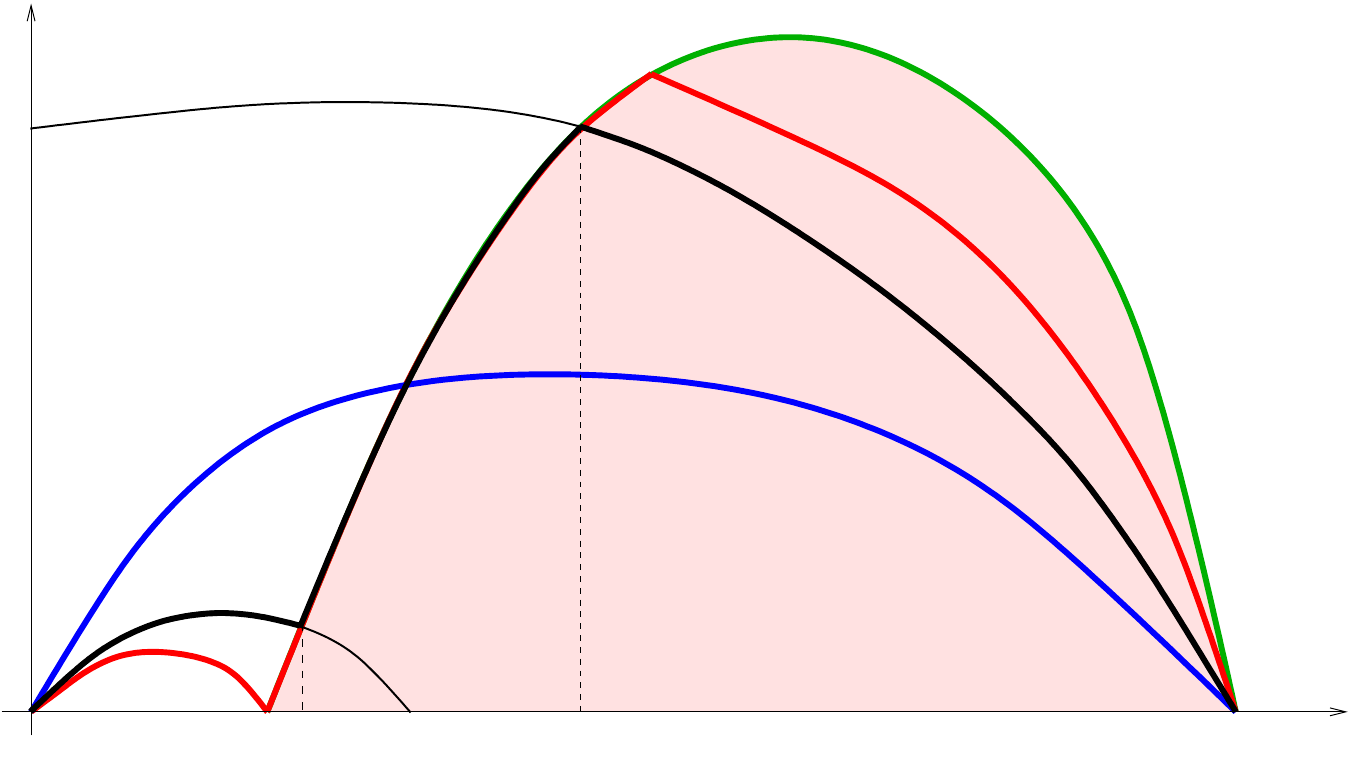} 
\caption{The structure of the optimal solution as a function of $\beta$, under the simplifying assumption that $F(U)$ is monotone (see \cite{BressanChiri22,bianchini2026travelingprofilescontrolcost} for more details about this condition): starting from the natural speed $\beta^*$ (blue), for $\beta^* \leq \beta < \beta^{**}$ the optimal profile $P_\beta(U)$ corresponds to a traveling wave of \eqref{Equa:trav_pro_a>0}. For $\beta \geq \beta^{**}$ (red), the optimal profile minimizing $E(\beta)$ passes through $(U^*,0)$, and then it is not the graph of an actual solution to \eqref{Equa:trav_pro_a>0}.}
\label{Fig:optproof_2}
\end{figure}

The function which determines when $P_\beta$ remains on $P^*(U)$ is \gls{Fdef}. If $F(U) + \beta \geq 0$, then the flow of \eqref{sys} is entering the domain $0 \leq P \leq P^*(U)$, and then by applying the control $\alpha \geq 0$ we can force the trajectory to remain on $P^*$. Viceversa, if $F(U) + \beta < 0$, then there is no positive control allowing to stay on $P^*(U)$, since the flow of \eqref{sys} is exiting. The generic structure of $F(U)$ is then of fundamental importance in the analysis of the optimal profile $P_\beta(U)$.

The next sections are devoted to the analysis of $F(U)$ and $\Gamma_u(\beta,f),\Gamma_s(\beta,f)$, and their intersections with the critical curve $P^*(U)$.

\section{Genericity of transversality properties}
\label{S:transver_flux}

Our aim is to show that there exists a dense family of polynomial sources $f$ satisfying Assumptions \eqref{H1} and enjoying the following additional transversality properties. 

\begin{enumerate}

\item \label{Point1:transv} The function \eqref{F}
\begin{equation*}
(U^*,1) \ni U \mapsto F(U) = \frac{3f(U) + U f'(U)}{2\sqrt{Uf(U)}}
\end{equation*}
has finitely many critical points and the second derivative does not vanish at these points (i.e. they are local maxima/minima).

\item \label{Point2:transv} For all $\beta$ the function \newglossaryentry{Fudef}{name=\ensuremath{F_u(U,\beta)},description={difference $\Gamma_u(U,\beta) - P^*(U)$}}
\begin{equation*}
[U^*,U_u(\beta)] \ni U \mapsto \gls{Fudef} = \Gamma_u(U,\beta) - P^*(U)
\end{equation*}
has finitely many zeros, and either the first or the second derivative does not vanish at these points (i.e. they are local maxima/minima).

\item \label{Point3:transv} Similarly, for all $\beta$ the function \newglossaryentry{Fsdef}{name=\ensuremath{F_s(U,\beta)},description={difference $\Gamma_s(U,\beta) - P^*(U)$}}
\begin{equation*}
[U^*,1] \ni U \mapsto \gls{Fsdef} = \Gamma_s(U,\beta) - P^*(U)
\end{equation*}
has finitely many zeros, and either the first or the second derivative does not vanish at these points (i.e. they are local maxima/minima).

\end{enumerate}

We will also show that the above transversality conditions are generic, in the sense that the set of functions satisfying them is open and dense. Moreover when restricted to polynomials of fixed degree, this transversality has full Lebesgue measure.


\subsection{Transversality properties of $F$}
\label{Ss:transv_F}

We start with the function $F$, whose analysis is easier since it does not depend on $\beta$.

A straightforward application of Sard's Lemma \cite{sard_origi} would let us conclude that the set of speed $\beta$ not satisfying the transversality condition
\begin{equation*}
F'(U) \not= 0 \quad \text{when} \quad F(U) + \beta = 0
\end{equation*}
is compact with zero Lebesgue measure for every smooth source $f$, because of
\begin{equation*}
F(U) \sim \begin{cases}
\frac{\sqrt{f'(U^*)}}{(U - U^*)^{\frac{1}{2}}} & U \searrow U^*, \\
- \frac{\sqrt{-f'(1)}}{(1 - U)^{\frac{1}{2}}} & U \nearrow 1.
\end{cases}
\end{equation*}
In particular
\begin{equation}
\label{Equa:limi_F}
\lim_{U \to U^*} F(U) = + \infty, \quad \lim_{U \to 1} F(U) = - \infty.
\end{equation}

More precisely 

\begin{proposition}
\label{zeros}
Let $F$ be defined by \eqref{F}. There exists a compact set \newglossaryentry{Ncal0}{name=\ensuremath{\mathcal N_0},description={set of critical values for $F$}} $\gls{Ncal0} \subseteq \R$ of Lebesgue measure zero such that for all $\beta \in \R\setminus \mathcal{N}_0$ there is \newglossaryentry{Nbeta}{name=\ensuremath{N(\beta)},description={number of roots of $F(U) + \beta = 0$}} $N = \gls{Nbeta}$ and there are finitely many $U_0=U^*<U_1<...<U_{2N-1}<U_{2N}=1$ satisfying \newglossaryentry{D+i}{name=\ensuremath{D^+_i},description={intervals where the flow is entering the region $\{P \leq P^*\}$}}\newglossaryentry{D-i}{name=\ensuremath{D^-_i},description={intervals where the flow is exiting the region $\{P \leq P^*\}$}}
\begin{equation*}
F(U)+\beta \begin{cases} =0 & \text{ for }  U=U_i, \ i=1,...,2N-1, \\
>0  & \text{ for } U \in \gls{D+i} := (U_{2i},U_{2i+1}), i=0,...,N-1, \\
<0  & \text{ for } U \in \gls{D-i} := (U_{2i+1},U_{2i+2}), i=0,...,N-1,
\end{cases}
\end{equation*}
and
\begin{equation*}
(-1)^i F'(U_i) > 0 \quad \forall i=1,..., 2N-1.
\end{equation*}
\end{proposition}

\begin{proof}
%
First note that, by (\ref{H1}), $F \in C^1((U^*, 1))$ and there holds
\begin{equation*}
F'(U)=\frac{2U f(U) f'(U) + U^2 f(U) f(U)''-\frac{3}{2}f^2(U)-\frac{1}{2}U^2(f'(U))^2}{2(Uf(U))^\frac{3}{2}},
\end{equation*}
which implies in accordance with \eqref{Equa:limi_F}
\begin{equation*}
\lim_{U \to U^*, 1} F'(U)=-\infty.
\end{equation*}
In particular, we can find $I\subseteq(U^*,1)$ compact such that $F'(U)<0 \ \forall U \notin I$.

A direct application of Sard's Lemma \cite{sard_origi} then implies that the compact set
$$
\gls{Ncal0} := - F(\{U\in (U^*,1): \ F'(U)=0\})
$$
is Lebesgue negligible, i.e.
\begin{align*}
\mathcal L^1(\mathcal N_0) &= \mathcal{L}^1 \big( F \big( \big\{ U \in (U^*,1): F'(U) = 0 \big\} \big) \big) = 0 .
\end{align*}

Fix $\beta\in \R\setminus \mathcal{N}_0$. From the coercivity conditions \eqref{Equa:limi_F}
we can infer that the level set \newglossaryentry{Lbeta}{name=\ensuremath{L_\beta},description={level set $\{F(U) + \beta = 0\}$}}
$$
\gls{Lbeta}:= \big\{ U\in(U^*,1): \ F(U)+\beta=0 \big\}
$$
is compact.
Assume by contradiction that there is a sequence $\{U_n\}_{n\in\mathbf{N}}\subseteq L_\beta$ such that $U_n\neq U_k$ for $n\neq k$. By the previous observation we have that, up to subsequences, $U_n\to U \in L_\beta$ and $U_n \not= U$. The derivative of $F$ at $U$ can be computed as
\begin{equation*}
F'(U)=\lim_{n\to+\infty} \frac{F(U_n)-F(U)}{U_n-U}=0,
\end{equation*}
contradicting the assumption $\beta \notin \mathcal{N}_0$. 

We conclude that the cardinality of $L_\beta$ is finite and, once again by \eqref{Equa:limi_F}, we can find $U_0 = U^* < U_1 < ... < U_{2N-1} < U_{2N} = 1$ as in the statement.
\end{proof}

\begin{remark}
\label{zerosp}
Clearly, if $K \subseteq\R\setminus \mathcal{N}_0$ is compact, we can find $N \ge 1$ such that $N(\beta)\le N \ \forall \beta \in K$. Indeed, first of all, by \eqref{Equa:limi_F}, we find $I \Subset (U^*,1)$ such that $-F(U) \notin K \ \forall U \notin I$, which implies
\begin{equation*}
\big\{ U\in(U^*,1): F(U)+\beta=0 \big\} = \big\{ U\in I: F(U)+\beta=0 \big\} \ \forall \beta \in K.
\end{equation*}
By contradiction, assume that for all $n \in \N$ there exists $\beta_n \in K$ such that
$$
\sharp \{U\in I: F(U)+\beta_n=0\} \geq n.
$$
By compactness, up to subsequences, we may assume $\beta_n \to \beta \in K$. 
By induction, construct a sequence $\{U_n\}_{n\in\N}\subseteq I$ such that
$$
F(U_n)+\beta_n=0 \ \forall n\in \N \quad \text{and} \quad U_n \neq U_k \ \forall n\neq k.
$$
With the same argument as the one employed in the proof of Lemma \ref{zeros}, we deduce that
$$
U_n \to U \in I \quad \text{such that} \quad F(U)+\beta=0 \ \text{and} \ F'(U)=0,
$$
which contradicts the hypothesis that $-\beta$ is not singular. The same proof shows that $N(\beta)$ is constant in any connected component of $\R \setminus \mathcal N_0$.
\end{remark}

In the case of polynomial $f(U)$, we can further improve the above proposition by showing that for all $\beta$ the number of critical points is finite and they are local minimizers or maximizers outside a codimension $1$ set in the space of polynomial with fixed degree $k$.

\begin{proposition}
\label{Prop:F_transv_poly}
Assume that
\begin{equation}
\label{Equa:ply_formu}
f(U) = \sum_{i=1}^k a_i U^i
\end{equation}
is a polynomial function satisfying condition \eqref{H1} (hence in particular $a_0 = 0$) and define, as in \eqref{F},
\begin{equation*}
F(U) = \frac{3f(U)+Uf'(U)}{2\sqrt{Uf(U)}}, \quad U \in (U^*,1).
\end{equation*}
Then the following holds.
\begin{enumerate}
\item 
The set
\begin{equation*}
\big\{ U \in (U^*,1): \ F'(U)=0 \big\}  
\end{equation*}
is finite.

\item \label{Point2:trans_poly} There exists a closed set \gls{NTFk} of Hausdorff dimension at most $k-1$ in the space of parameters $(a_1,\dots,a_k) \in \R^k$ defining polynomials with degree $k$ through \eqref{Equa:ply_formu} and satisfying Assumption \eqref{H1}, such that outside this set the second derivative $F''$ is not $0$ in the critical points $\{F' = 0\}$, i.e. they are maxima or minima.

\item \label{Point3:trans_poly} For all fixed $U^* \in (0,1)$, the set of non-transversality
\begin{equation*}
\begin{split}
&\bigg\{ (U,a_1,\dots,a_k), f(U) = \sum_{i=1}^k a_i U^i \ \text{satisfies (\ref{H1}), $f(U^*) = 0$ and} \ F'(U) = F''(U) = 0 \bigg\} \\
& \qquad \qquad \subset \bigg\{ (U,a_1,\dots,a_k), f(U) = \sum_{i=1}^k a_i U^i \ \text{satisfies (\ref{H1}), $f(U^*) = 0$} \bigg\} 
\end{split}
\end{equation*}
has Hausdorff dimensions at most $k-1$.
\end{enumerate}
\end{proposition}

Note that since $f(U)$ has $3$ zeros, clearly for $k \leq 2$ the sets are empty. Moreover, the non tranversality set has zero Lebesgue measure. Point \eqref{Point3:trans_poly} will be useful when studying the optimal profile problem keeping the point $U^*$	fixed.


\begin{proof}
{\it Point (1).} Clearly by \eqref{Equa:limi_F} we have that $F'(U)$ cannot be identically $0$ 
: then, since
\begin{equation*}
F'(U)=\frac{2Uf(U) f'(U) + U^2 f(U) f''(U) - \frac{3}{2}f^2(U) - \frac{1}{2}U^2(f'(U))^2}{2(Uf(U))^\frac{3}{2}},
\end{equation*}
zeros of $F'$ are also zeros for the non-identically null polynomial
\begin{equation*}
U \mapsto Q(U) = 2U f(U) f'(U)+U^2 f(U) f''(U)-\frac{3}{2}f^2(U)-\frac{1}{2}U^2(f'(U))^2.
\end{equation*}
We deduce that they are finite.
%

\medskip

\noindent{\it Point (2).} Consider the function
\begin{equation*}
(U,a_1,\dots,a_k) \mapsto Q(U,a_1,\dots,a_k) = 2 U f(U) f'(U) + U^2 f(U) f''(U)-\frac{3}{2}f^2(U)-\frac{1}{2}U^2(f'(U))^2,
\end{equation*}
where by Assunmption \eqref{H1}
\begin{equation*}
f(U) = \sum_{i=1}^k a_i U^i, \quad f(U^*) = \sum_{i=1}^k a_i (U^*)^i = 0, \quad f(1) = \sum_{i=1}^k a_i = 0.
\end{equation*}
In this part the second condition is not used, since we allow $U^*$ to vary.

The function $Q$ can be written explicitly as the quadratic form
\begin{equation*}
\begin{split}
Q(U,a_1,\dots,a_k) &= \frac{1}{2} \sum_{ij} a_i a_j U^{i+j} \big( i^2 + j^2 + i + j - ij - 3 \big) \\
&= \frac{1}{2} (a_i U^i)^T [\gamma_{ij}] (a_i U^i), \quad \gamma_{ij} = i^2 + j^2 + i + j - ij - 3.
\end{split}
\end{equation*}
If we define the $k \times k$-matrices
\begin{equation*}
L_1 = \left[ \begin{array}{ccccc}
1 & 1 & 1 & \dots & 1 \\
0 & 1 & 1 & \dots & 1 \\
0 & 0 & 1 & \dots & 1 \\
\vdots & \vdots & \vdots & \ddots & \vdots \\
0 & 0 & \dots & 0 & 1
\end{array} \right], \quad L^{-1}_1 = \left[ \begin{array}{ccccc}
1 & - 1 & 0 & \dots & 0 \\
0 & 1 & -1 & \dots & 0 \\
0 & 0 & 1 & \dots & 0 \\
\vdots & \vdots & \vdots & \ddots & \vdots \\
0 & 0 & \dots & 0 & 1
\end{array} \right],
\end{equation*}
\begin{equation*}
L_2 = \left[ \begin{array}{cccccc}
1 & 0 & 0 & 0 & \dots & 0 \\
0 & 1 & 1 & 1 & \dots & 1 \\
0 & 0 & 1 & 1 & \dots & 1 \\
0 & 0 & 0 & 1 & \dots & 1 \\
\vdots & \vdots & \vdots & \vdots &\ddots & \vdots \\
0 & 0 & 0 & 0 & \dots & 1
\end{array} \right], \quad L^{-1}_2 = \left[ \begin{array}{cccccc}
1 & 0 & 0 & 0 & \dots & 0 \\
0 & 1 & -1 & 0 & \dots & 0 \\
0 & 0 & 1 & -1 & \dots & 0 \\
0 & 0 & 0 & 1 & \dots & 0 \\
\vdots & \vdots & \vdots  & \vdots & \ddots & \vdots \\
0 & 0 & 0 & 0 & \dots & 1
\end{array} \right],
\end{equation*}
\begin{equation*}
L_3 = \left[ \begin{array}{ccccccc}
1 & 0 & 0 & 0 & 0 & \dots & 0 \\
0 & 1 & 0 & 0 & 0 &\dots & 0 \\
0 & 0 & 1 & 1 & 1 & \dots & 1 \\
0 & 0 & 0 & 1 & 1 & \dots & 1 \\
0 & 0 & 0 & 0 & 1 & \dots & 1 \\
\vdots & \vdots & \vdots & \vdots & \vdots & \ddots & \vdots \\
0 & 0 & 0 & 0 & 0 & \dots & 1
\end{array} \right], \quad L^{-1}_3 = \left[ \begin{array}{ccccccc}
1 & 0 & 0 & 0 & 0 & \dots & 0 \\
0 & 1 & 0 & 0 & 0 & \dots & 0 \\
0 & 0 & 1 & -1 & 0 & \dots & 0 \\
0 & 0 & 0 & 1 & -1 &\dots & 0 \\
0 & 0 & 0 & 0 & 1 &\dots & 0 \\
\vdots & \vdots & \vdots & \vdots & \vdots & \ddots & \vdots \\
0 & 0 & 0 & 0 & 0 & \dots & 1
\end{array} \right],
\end{equation*}
and
\begin{equation}
\label{Equa:L_matrix}
L = L_3 L_2 L_1 = \left[
\begin{array}{cccccc}
1 & 1 & 1 & 1 & \dots & 1 \\
0 & 1 & 2 & 3 & \dots & k-1 \\
0 & 0 & 1 & 3 & \dots & \frac{(k-1)(k-2)}{2} \\
0 & 0 & 0 & 1 & \dots & \frac{(k-2)(k-3)}{2} \\
\vdots & \vdots & \vdots & \vdots & \ddots & \vdots \\
0 & 0 & 0 & 0 & \dots & 1
\end{array} \right],
\end{equation}
then
\begin{equation}
\label{Equa:diag_gamma}
L^{-T} [\gamma_{ij}] L^{-1} = \left[ \begin{array}{cccccc}
0 & 3 & 2 & 0 & \dots & 0 \\
3 & -1 & 0 & 0 & \dots & 0 \\
2 & 0 & 0 & 0 & \dots & 0 \\
0 & 0 & 0 & 0 & \dots & 0 \\
\vdots & \vdots & \vdots & \vdots & \ddots & \vdots \\
0 & 0 & 0 & 0 & \dots & 0
\end{array} \right].
\end{equation}

Consider the map
\begin{equation}
\label{Equa:map_tildea}
(U,a_2,\dots,a_k) \mapsto \tilde a(U,a_2,\dots,a_k) = L (a_i U^i), \quad a_1 = - \sum_i a_i.
\end{equation}
This map is a diffeomorphism because $U \in (U^*,1)$ and the Jacobian matrix
\begin{equation*}
\left[ \begin{array}{ccccc}
- \sum_{i=2}^k a_i & - U & -U & \dots & - U \\
2 a_2 U & U^2 & 0 & \dots & 0 \\
3 a_3 U^2 & 0 & U^3 & \dots & 0 \\
\vdots & \vdots & \vdots & \ddots & \vdots \\
k a_k U^{k-1} & 0 & 0 & \dots & U^k
\end{array} \right],
\end{equation*}
has determinant
\begin{equation*}
U^{\frac{k(k+1)}{2}-1} \sum_{i=2}^k (i-1) a_i = U^{\frac{k(k+1)}{2}-1} f'(1) \not= 0.  
\end{equation*}

Using \eqref{Equa:map_tildea} we obtain
\begin{equation*}
Q(U,a_1,\dots,a_k) = \frac{1}{2} \tilde a^T L^{-T} [\gamma_{ij}] L^{-1} \tilde a.
\end{equation*}
Performing the same computations for the derivative of $Q(,a_1,\dots,a_k)$ we obtain
\begin{equation*}
\begin{split}
\partial_U Q(U,a_1,\dots,a_k) &= \frac{1}{2} \sum_{ij} a_i a_j U^{i+j-1} (i+j) \gamma_{ij} \\
&= \frac{1}{2U} \tilde a^T \Big( L^{-T} \diag(1,\dots,k) L^T L^{-T} [\gamma_{ij}] L^{-1} + L^{-T} [\gamma_{ij}] L^{-1} L \diag(1,\dots,k) L^{-1} \Big) \tilde a.
\end{split}
\end{equation*}

We will show that the conditions 
\begin{equation}
\label{Equa:QpartialQ_0}
Q(U,a_1,\dots,a_k) = 0, \quad \partial_U Q(U,a_1,\dots,a_k) = 0
\end{equation}
define a set of codimention at least $2$ in $(U,\tilde a) \in \R^{k+1}$.
%
%

For this purpose, it is enough to consider the first $3$ components of $\tilde a$: indeed because of \eqref{Equa:diag_gamma} we have
\begin{equation*}
L^{-T} \diag(1,\dots,k) L^T L^{-T} [\gamma_{ij}] L^{-1} + L^{-T} [\gamma_{ij}] L^{-1} L \diag(1,\dots,k) L^{-1} = \left[ \begin{array}{cc} \left[ \begin{array}{ccc}
0 & 9 & 14 \\
9 & 2 & 0 \\
14 & 0 & 0
\end{array} \right] & A_{12} \\
A_{12}^T & 0 \end{array} \right],
\end{equation*}
with $A_{12} \in \R^{3 \times (k-3)}$. The second condition in \eqref{Equa:QpartialQ_0} gives then
\begin{equation}
\label{Equa:der_tilde_a}
\left( \begin{array}{ccc}
\tilde a_1 & \tilde a_2 & \tilde a_3
\end{array} \right) \left[ \begin{array}{ccc}
0 & 9 & 14 \\
9 & 2 & 0 \\
14 & 0 & 0
\end{array} \right] \left( \begin{array}{c}
\tilde a_1 \\ \tilde a_2 \\ \tilde a_3
\end{array} \right) + 2 \left( \begin{array}{ccc}
\tilde a_4 & \dots & \tilde a_k
\end{array} \right) A_{12}^T\left( \begin{array}{c}
\tilde a_1 \\ \tilde a_2 \\ \tilde a_3
\end{array} \right) = 0.
\end{equation}
%

Assume that the linear term in $(\tilde a_4,\dots,\tilde a_k)$ of \eqref{Equa:der_tilde_a} is $0$: then by \eqref{Equa:QpartialQ_0} the first $3$ components of $\tilde a$ are required to satisfy 
\begin{equation*}
\begin{cases}
6 \tilde a_1 \tilde a_2 - \tilde a_2^2 + 4 \tilde a_1 \tilde a_3 = 0, \\
18 \tilde a_1 \tilde a_2 + 2 \tilde a_2^2 + 28 \tilde a_1 \tilde a_3 = 0,
\end{cases}
\end{equation*}
a system whose solutions are the three lines
\begin{equation}
\label{Equa:thrre_tildea}
\big\{ \tilde a_1 = \tilde a_2 = 0 \big\}, \quad \big\{ \tilde a_2 = \tilde a_3 = 0 \big\}, \quad \bigg\{ \tilde a_2 = \frac{8}{3} \tilde a_1, \tilde a_3 = - \frac{20}{9} \tilde a_1 \bigg\}.
\end{equation}
%
The first condition reads as
\begin{equation*}
\sum_i a_i U^i = f(U) = 0, \quad \sum_i a_i (i-1) U^i = U \frac{d}{dU} f(U) - f(U) = 0,
\end{equation*}
which is not possible under Assumption \eqref{H1}. The second condition reads as
\begin{equation*}
U \frac{df}{dU} - f = 0, \quad \sum_i a_i \frac{(i-1)(i-2)}{2} U^i = \frac{U^2}{2} \frac{d^2 f}{dU^2} - \bigg( U \frac{df}{dU} - f \bigg) = \frac{U^2}{2} \frac{d^2 f}{dU^2} = 0,
\end{equation*}
and substituting these conditions into \eqref{Equa:QpartialQ_0} we obtain
\begin{equation*}
\partial_U Q(U,a_1,\dots,a_k) = U^2 f \frac{d^3f}{dU^3} = 0.
\end{equation*}
Thus this case corresponds to
\begin{equation*}
U \frac{df}{dU} = f, \quad \frac{d^2 f}{dU^2} = 0, \quad \frac{d^3 f}{dU^3} = 0,
\end{equation*}
which has higher codimension (namely 3).
%

If the second term of \eqref{Equa:der_tilde_a} is not $0$, then it defines a linear function over a codimension $1$ algebraic surface given by the first condition of \eqref{Equa:QpartialQ_0}, i.e. a codimention $2$ algebraic surface. This set corresponds to an algebraic surface in $\R^{k+1}$ by the change of variable \eqref{Equa:map_tildea}: its projection on $(a_1,\dots,a_k)$ is then the set \gls{NTFk} of Hausdorff dimension $k-1$.

\medskip

\noindent{\it Point (3).} The final part is whether the codimension of the set
\begin{equation}
\label{Equa:critical1}
\begin{split}
&\bigg\{ (U,a_1,\dots,a_k), f(U) = \sum_{i=1}^k a_i U^i \ \text{satisfies (\ref{H1}), $f(U^*) = 0$ and} \ F'(U) = F''(U) = 0 \bigg\} \\
\end{split}
\end{equation}
is equal to $1$ in the set
\begin{equation}
\label{Equa:critical2}
\begin{split}
\bigg\{ (U,a_1,\dots,a_k), f(U) = \sum_{i=1}^k a_i U^i \ \text{satisfies (\ref{H1}), $f(U^*) = 0$} \bigg\} 
\end{split}
\end{equation}
with $U^*$ fixed: this condition is equivalent to
\begin{equation*}
a_1 + a_2 U^* = - \sum_{i=3}^k a_i (U^*)^{i-1}, \quad a_1 + a_2 = - \sum_{i=3}^k a_i.
\end{equation*}
The space \eqref{Equa:critical1} is thus the section of the codimension 2 surface given by \eqref{Equa:QpartialQ_0} into the affine codimension $1$ sub-space \eqref{Equa:critical2}: its Hausdorff dimension is thus $k-1$. 
%
\end{proof}

\begin{remark}
\label{Rem:cubic_case}
For the cubic source
\begin{equation*}
f(U,U^*) = U (U - U^*) (1 - U)
\end{equation*}
one gets using \eqref{Equa:L_matrix}
\begin{equation*}
\tilde a_1 = f(U,U^*) = - U^*U + (1 + U^*) U^2 - U^3, \quad \tilde a_2 = U f'(U,U^*) - f(U,U^*) = (1 + U^*) U^2 - 2 U^3, \quad \tilde a_3 = - U^3,
\end{equation*}
so that the three solutions \eqref{Equa:thrre_tildea} yield
\begin{equation*}
\tilde a_1 = 0, \ \tilde a_2 = 0 \quad \Rightarrow \quad f(U,U^*) = 0, \ U \frac{d}{dU} f(U,U^*) = 0,
\end{equation*}
\begin{equation*}
\tilde a_2 = 0, \ \tilde a_3 = 0 \quad \Rightarrow \quad U \frac{d}{dU} f(U,U^*) = f(U,U^*), \ U^3 = 0,
\end{equation*}
\begin{equation*}
\tilde a_2 = \frac{8}{3} \tilde a_1, \ \tilde a_3 = - \frac{20}{9} \tilde a_1 \quad \Rightarrow \quad U \frac{d}{dU} f(U) - f(U) = \frac{8}{3} f(U), \ U^3 = \frac{20}{9} f(U).
\end{equation*}
The only case with is not trivial of $f(U,U^*)$ is the third one, whose unique solution is
\begin{equation*}
U = \frac{2}{5}, \quad U^* = \frac{7}{25}.
\end{equation*}
See Fig. \ref{Fig:Fcubic_der} for a plot of $Q(U,U^*)$ and $\partial_U F(U,U^*)$.

%
%

\begin{figure}
\begin{subfigure}{.475\textwidth}
\resizebox{\textwidth}{!}{\includegraphics{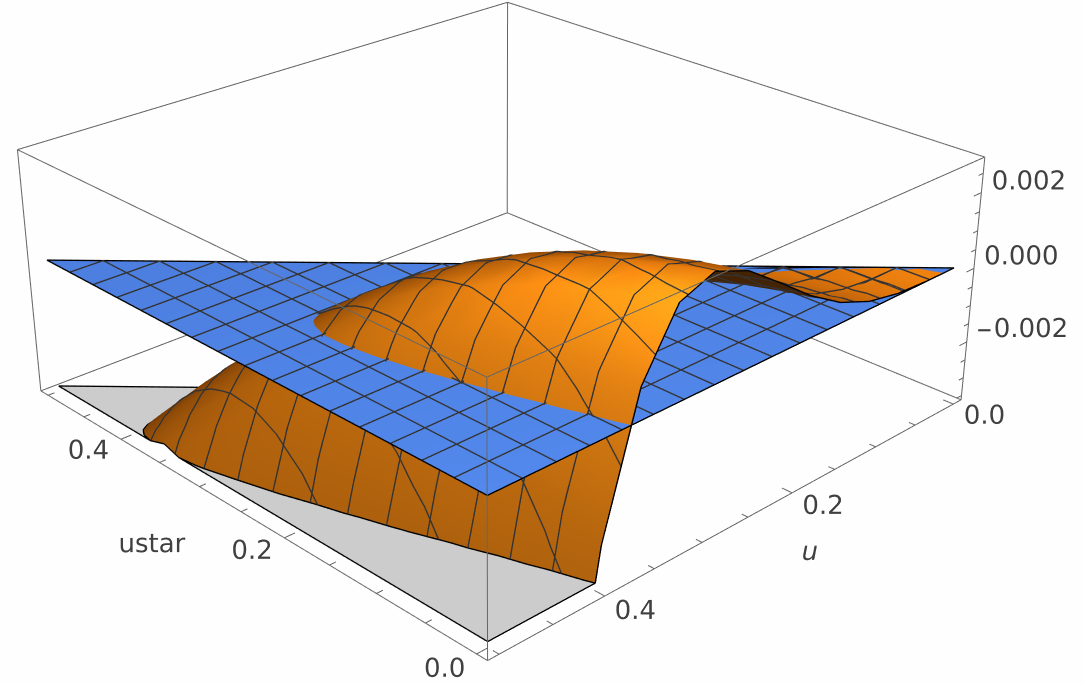}}
\end{subfigure} \hfill
\begin{subfigure}{.475\textwidth}
\resizebox{\textwidth}{!}{\includegraphics{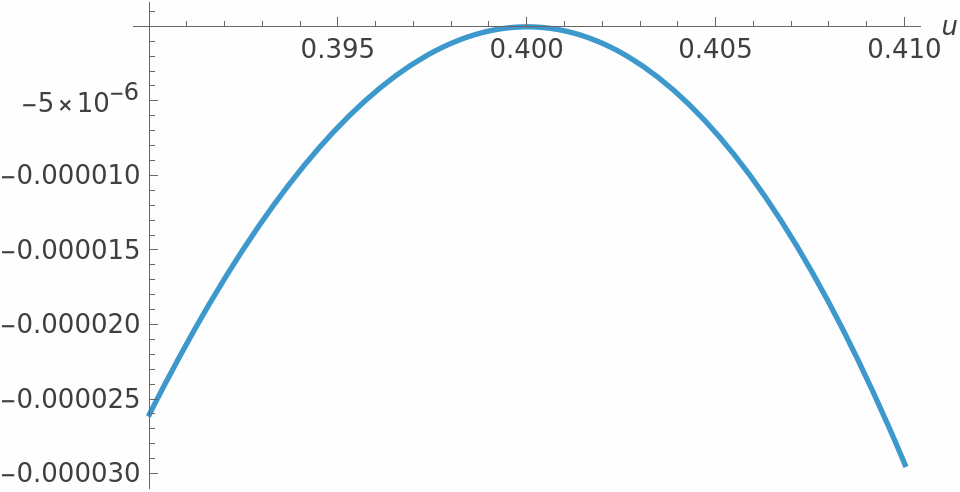}}
\end{subfigure}
\caption{The numerator $Q(U)$ of the function $\partial_U F(U,U^*)$ for the cubic case of Remark \ref{Rem:cubic_case} and its graph when $U^* = \frac{7}{25}$.}
\label{Fig:Fcubic_der}
\end{figure}
\end{remark}

\subsection{Transversality of $\Gamma_u(U,\beta)$, $\Gamma_s(U,\beta)$ and $P^*(U)$}
\label{Ss:transvers_Gammaus}

We turn now the attention to the transversality properties of $\Gamma_u,\Gamma_s$ and $P^*$.

\begin{proposition}
\label{Prop:transv_Gamma_gen}
There exists a compact set \newglossaryentry{Ncal1}{name=\ensuremath{\mathcal N_1},description={set of non-transversality of $\Gamma_s,\Gamma_u$ and the curve $P^*$}} $\gls{Ncal1} \subseteq \R$ of negligible Lebesgue measure such that for all $\beta \in [\beta^*,+\infty)\setminus \mathcal{N}_1$ the unstable and stable manifolds $\Gamma_u(U,\beta),\Gamma_s(U,\beta)$ intersect transversely the critical curve $P^*$ only at a finite number of points: in other words, the sets
\begin{equation*}
\gls{Cs} := \big\{ U\in (U^*, 1): \Gamma_s(U,\beta)=P^*(U) \big\}, \quad \gls{Cu} := \big\{ U\in (U^*, 1): \Gamma_u(U,\beta)=P^*(U) \big\}
\end{equation*}
are at most finite and 
\begin{equation}
\label{tang_10}
F(U) + \beta \neq 0 \quad \forall U \in C_s \cup C_u,
\end{equation}
if $\beta \notin \mathcal N_1$.
\end{proposition}

\begin{proof}{\it Study of $\Gamma_s$ and $P^*$.}
First of all, we wish to show that, for $\beta$ sufficiently large, the stable manifold associated to Equation \eqref{Equa:Gammas_eq0} and the critical curve $P^*$ cannot cross tangentially, i.e. the set \newglossaryentry{Ncal1s}{name=\ensuremath{\mathcal N_{1,s}},description={set of non-transversality of $\Gamma_s$ and the curve $P^*$}}
\begin{equation*}
\gls{Ncal1s} := \Big\{ \beta \in [\beta^*, +\infty): \Gamma_s(U,\beta)=P^*(U) \text{ and } F(U)+\beta=0 \text{ for some } U\in (U^*, 1) \Big\}
\end{equation*}
is bounded.

Note that, for $\beta>0$, the function $U \mapsto \beta(1-U)$ is a subsolution to \eqref{Equa:Gammas_eq0} in  $[U^*,1]$, which implies
\begin{equation}
\label{lbound}
\Gamma_s(U,\beta)\ge \beta(1 - U) \quad \forall U \in [U^*,1].
\end{equation}
We can thus estimate for $U^* \le U\le 1$
\begin{align*}
\Gamma_s(U,\beta) - \beta(1-U) =\int_U^1 \frac{f(U')}{\Gamma_s(\beta, U')}\, dU' \le \frac{1}{\beta}\int_U^1 \frac{f(U')}{1-U'} \, dU'\le -\frac{f'(1)}{\beta}(1+C_1)(1-U),
\end{align*}
where $C_1>0$ is a constant only depending on $\|f\|_{C^2}$. When $U$ is sufficiently close to $1$, there holds
$$
3f(U)+Uf'(U) \underset{U \to 1}{\sim} 4 f'(U) U < 0,
$$
and thus, assuming also
$$
\beta^2 \ge 1+f'(1)(1+C_1),
$$
we conclude that
\begin{align*}
-\frac{3f(U)+Uf'(U)}{2 \beta \Gamma_s(U,\beta)}\ge -\frac{3f(U)+Uf'(U)}{[\beta^2-f'(1)(1+C_1)](1-U)} \sim -\frac{f'(1)}{1-U} \to +\infty \quad \text{as} \ U \to 1.
\end{align*}
In particular, we can find $\overline U \in (U^*, 1)$ only depending on $f$ such that
\begin{equation}
\label{ge1}
- \frac{3f(U)+Uf'(U)}{2\beta\Gamma_s(U,\beta)} > 1 \quad \forall U \in [\overline U, 1] \text{ and } \forall \beta > \sqrt{1 + f'(1) (1 + C_1)}.
\end{equation}

Furthermore, note that, once again by \eqref{lbound}, the crossing condition
\begin{equation*}
\Gamma_s(U,\beta) = P^*(U)
\end{equation*}
implies, in particular,
\begin{equation*}
\beta (1-U) \le \Gamma_s(U,\beta)=\sqrt{Uf(U)}\le\sqrt{C_2(1-U)},
\end{equation*}
i.e.
\begin{equation}
\label{lb}
U \ge 1 - \frac{C_2}{\beta^2}, \quad \text{with } C_2 = C_2(\|f\|_{C^2}).
\end{equation}
We set
\begin{equation}
\label{betabar}
\overline {\beta}:=\max \bigg\{ \sqrt{1+f'(1)(1+C_1)}, \sqrt{\frac{C_2}{1-\overline{U}}} \bigg\}.
\end{equation}
By contradiction, let $\beta > \overline{\beta}$ and $U_T \in (U^*,1)$ such that
\begin{equation}
\label{tangu}
\Gamma_s(\beta, U_T)=\sqrt{U_Tf(U_T)} \quad \text{and} \quad F(U_T)+\beta=0.
\end{equation}
By \eqref{lb} and the choice of $\overline {\beta}$ in \eqref{betabar}, we have $U_T\ge \overline{U}$ and, thus, by \eqref{tangu} and \eqref{ge1},
\begin{align*}
1=-\frac{F(U_T)}{\beta}=-\frac{3f(U_T)+U_Tf'(U_T)}{2\beta\Gamma_s(\beta, U_T)}> 1,
\end{align*}
which clearly contradicts our assumption. We can thus conclude $\mathcal{N}_{1,s} \subseteq [\beta^*, \overline{\beta}]$.

By \eqref{Equa:limi_F}, there exists $I \subseteq (U^*,1)$ compact such that $F(U)+\beta\neq 0 \ \forall \beta \in [\beta^*, \overline {\beta}]$ and $\forall U \notin I$, which implies
\begin{equation*}
\mathcal{N}_{1,s} = \Big\{ \beta \in[\beta^*, \overline {\beta}]: \Gamma_s(U, \beta)=P^*(U) \text{ and } F(U)+\beta=0 \text{ for some } U\in I \Big\},
\end{equation*}
i.e. $\mathcal{N}_{1,s}$ is compact. By Proposition \ref{zeros}, it clearly holds
\begin{equation*}
\mathcal{L}^1(\mathcal{N}_{1,s}) = \mathcal{L}^1 \big( \mathcal{N}_{1,s} \setminus \mathcal{N}_0 \big).
\end{equation*}
Since the set $\R \setminus\mathcal{N}_0$ is open, it can be expressed as the countable union of compact intervals: by Remark \ref{zerosp} and the Implicit Function Theorem, if $I$ is one of these intervals, we can exhibit $N \ge 1$ and smooth functions \newglossaryentry{Uibeta}{name=\ensuremath{U_i(\beta)},description={roots of $F(U) + \beta = 0$}} $U_i = \gls{Uibeta}$, $i=1,...,N$, such that $\forall \beta \in I$ we have
\begin{equation}
\label{Equa:intervals_Ui}
\big\{ U \in (U^*,1): \ F(U)+\beta=0 \big\} = \big\{ U_1(\beta) < ... < U_{N}(\beta) \big\}. 
\end{equation}

Fix $\beta_0 < \beta^*$ and define the function \newglossaryentry{Gs}{name=\ensuremath{G_s},description={transversality function for $\Gamma_s$ and $P^*$}}
\begin{equation*}
\begin{array}{ccccc}
\gls{Gs} &:& (U^*,1) \times (\beta_0, + \infty) &\to& \R^2, \\
&& (U,\beta) &\mapsto&  (P^*(U)-\Gamma_s(U,\beta), F(U)+\beta).
\end{array}
\end{equation*}
If $I$ is one of the intervals where \eqref{Equa:intervals_Ui} holds, in order to prove our thesis it will be sufficient to show that
\begin{equation*}
\mathcal{L}^1 \Big( \Big\{ \beta \in I: \ G_s(U, \beta)= 0 \text{ for some } U\in I \Big\} \Big) = 0.
\end{equation*}
This means that
\begin{equation*}
\Big\{ \beta \in K : \ G_s(U, \beta)=0 \text{ for some } U \in (U^*, 1) \Big\} = \bigcup_{i=1}^N \big\{ \beta \in I: \ G_s(U_i(\beta), \beta) = 0 \big\} =: \bigcup_{i=1}^N \mathcal N_{1,s,i}
\end{equation*}
is negligible, i.e. every $\mathcal N_{1,s,i}$ is $\mathscr L^1$-negligible. Computing the Jacobian matrix of $G$ we get
\begin{equation*}
DG_s(U,\beta) = \left[ \begin{array}{cc}
F(U)+\beta  & -\partial_{\beta} \Gamma_s(U,\beta) \\
F'(U) & 1
\end{array} \right].
\end{equation*}
Using the ODE \eqref{Equa:Gammas_eq0}, one obtains by differentiating w.r.t. $\beta$ (see also \eqref{Equa:Gammas_beta}) that
\begin{equation*}
\partial_\beta \Gamma_s(U,\beta) > 0. 
\end{equation*}
and then 
$$
\det(DG_s(U_i(\beta), \beta)) = \partial_{\beta} \Gamma_s(U_i(\beta),\beta)F'(U_i(\beta))\neq 0 \quad \forall \beta \in I, i=1,...,N.
$$
We deduce that, given $\beta_0\in I$ such that $G_s(U_i(\beta_0), \beta_0)=0$, necessarily
\begin{equation*}
G_s(U_i(\beta), \beta) = (P^*(U_i(\beta)) - \Gamma_s(U_i(\beta),\beta),0) \neq 0 \quad \forall \beta \in (\beta_0-\delta_0, \beta_0+\delta_0) \setminus \beta_0,
\end{equation*}
i.e. points in $N_i$ are isolated and, thus, at most locally finite in $I$.

It remains to show that the cardinality of $C_s$ is finite. We exploit once again the argument in the proof of Proposition \ref{zeros} and Remark \ref{zerosp}: by contradiction consider a sequence $U_n \subseteq C_s$ consisting of distinct points and $U\in [U^*,1]$ such that $U_n \to U$. Clearly there holds $\Gamma_s(U, \beta)=P^*(U)$ and
\begin{equation*}
F(U)+\beta=(P^*-\Gamma_s(\cdot, \beta))'(U)=\lim_{n\to\infty}  \frac{(P^*(U_n)-\Gamma_s(U_n, \beta))-(P^*(U)-\Gamma_s(U, \beta))}{U_n-U}=0,
\end{equation*}
which is a contradiction both if $U\in \{U^*,1\}$ (by \eqref{Equa:limi_F}) and if $U\in (U^*,1)$ (by \eqref{tang_10}).

The proof for $C_u$ will follow the same line. Here, $\Gamma_u(U,\beta)=0 $ for all $\beta \ge \beta^{**}$, in which case the unstable manifold and $P^*$ never cross in the open interval $(U^*,1)$: in particular
\begin{align*}
\mathcal{N}_{1,u} :&= \Big\{ \beta \in [\beta^*, +\infty): \Gamma_u(U,\beta)=P^*(U) \text{ and } F(U)+\beta=0 \text{ for some } U \in (U^*, 1) \Big\} \\
&= \Big\{ \beta \in [\beta^*, \beta^{**}]: G_u(U,\beta)=(0,0)  \text{ for some } U \in (U^*, 1) \Big\},
\end{align*}
with \newglossaryentry{Gu}{name=\ensuremath{G_u},description={transversality function for $\Gamma_u$ and $P^*$}}
\begin{equation*}
\begin{array}{ccccc}
\gls{Gu} &:& (U^*,1) \times (\beta_0, +\infty) &\to& \R^2, \\
&& (U,\beta) &\mapsto& (P^*(U)-\Gamma_u(U,\beta), F(U)+\beta).
\end{array}
\end{equation*}

Finally, the thesis follows setting $\mathcal{N}_1 := \mathcal{N}_0\cup \mathcal{N}_{1,s} \cup \mathcal{N}_{1,u}$.
\end{proof}

\begin{remark}
\label{Rem:controll_inter}
The same computation as in Remark \ref{zerosp} gives that the number of intersections of $\Gamma_u(U),\Gamma_s(U)$ and $P^*(U)$ is constant in every connected component of $\R \setminus \mathcal N_1$.
\end{remark}

As in the case of $F(U) + \beta = 0$, by considering polynomial source $f(U)$ we can improve Proposition \ref{Prop:transv_Gamma_gen} obtaining that \gls{Ncal1} is finite.

\begin{proposition}
\label{Prop:Gamma_trans_poly}
Assume that the source $f(U)$ is a polynomial satisfying Assumption \eqref{H1} with $U^*$ fixed. Then
\begin{enumerate}
\item for all $\beta$ the cardinality of $C_s(\beta)$, $C_u(\beta)$ is uniformly bounded;
\item the set \gls{Ncal1}, defined in Proposition \ref{Prop:transv_Gamma_gen} as the set of speeds such that $\Gamma_u(U,\beta)$, $\Gamma_s(U,\beta)$ do not cross $P^*(U)$ transversely, is finite.
\end{enumerate}
\end{proposition}

\begin{proof}
It is sufficient to prove the proposition for $\Gamma_u$, being the proof for $\Gamma_s$ entirely similar.

\smallskip

\noindent{\it Finiteness of \gls{Cs}, \gls{Cu}.} For every $\beta$, the set where $F(U) + \beta = 0$ is a subset of the zeros of the polynomial
\begin{equation*}
\big( 3 f(U) + U f'(U) \big)^2 = 4 \beta^2 U f(U),
\end{equation*}
and then the number of zeros of $F(U) + \beta$ is uniformly bounded w.r.t. $\beta$.
Since there can be only one crossing point where of $P^*(U) - \Gamma_u(U) = 0$ in every connected component of $\{F(U) + \beta \not= 0\}$, it follows that
\begin{equation*}
\sharp \big\{ U : P^*(U) - \Gamma_u(U,\beta) = 0 \big\} \leq N < \infty
\end{equation*}
for some $N$. Hence the cardinality of $C_s(\beta),C_u(\beta)$ is uniformly bounded for all $\beta$.

\smallskip

\noindent{\it Finiteness of $\mathcal N_{1}$.} Consider the function
\begin{equation*}
U \mapsto H(U) = \Gamma_u(U,-F(U)) - P^*(U).
\end{equation*}
This function is $0$ exactly when
\begin{equation*}
\beta = - F(U) \quad \text{and} \quad \Gamma_u(U,\beta) = P^*(U).
\end{equation*}
Hence it is enough to prove that it has finitely many zeros $U_i$, which in turn it implies that $\mathcal N_{1,u} = - F(\{U_i\}_i)$ is finite.

For polynomial sources, $H(U)$ is real analytic in the interval because $\Gamma_u(U,\beta)$, $F(U)$ and $P^*$ are (the result \cite[Theorem 2.2]{Broer_Tangerman_1986} can be applied directly here since the contraction mapping generating $\Gamma_u$ depends analytically on $\beta$). Hence its zero set is locally finite.

Observe that the interval of interest is the set $F^{-1}([-\beta^{**},-\beta^*])$, which is compactly contained in $(U^*,1)$ because of $|F(U)| \to \infty$ as $U \to U^*,1$, and the zeros of $H(U)$ cannot converge to $U_{uf}(\beta)$.

Computing the derivative w.r.t. $U$ of $H(U)$ in one of its roots we obtain using the ODE
\begin{equation*}
\begin{split}
\frac{dH}{dU} &= \frac{d\Gamma_u}{dU} - \partial_\beta \Gamma_u \frac{dF}{dU} - \frac{dP^*}{dU} = - \frac{d\Gamma_u}{d\beta} \frac{dF}{dU}
\end{split}
\end{equation*}
Hence $H$ cannot be identically $0$ because $\partial_\beta \Gamma_u < 0$ and then $\frac{dF}{dU} = 0$ identically, which is impossible since then $F$ would constant: so the roots of $H$ are finite and isolated.

A similar proof can be done for $\Gamma_s$: the set $\mathcal N_{1,s}$ of speeds $\beta$ such $\Gamma_s$ does not cross $P^*(U)$ transversely is finite. In this case the compactness of the set $U$ is obtained by considering the set \gls{Ncal1} of Proposition \ref{Prop:transv_Gamma_gen}.
%
%

Taking $\gls{Ncal1} = \mathcal N_{1,u} \cup \mathcal N_{1,s}$ we obtain the second part of the statement. 
\end{proof}

\subsection{Transversal set $\mathcal T$ of speeds $\beta$}
\label{Ss:transc_calT}

We now collect the results of this section recalling the set of \emph{transversal speeds}, defined in Definition \eqref{Def:trans_calT_intro} as follows.


\begin{definition}
\label{Def:trans_calT}
We say that $\beta \ge \beta^*$ is \emph{transversal}, and write $\beta \in \gls{Tcal} = \mathcal T_f$, if
\begin{enumerate}
\item there exist finitely many $U_0=U^*<U_1<...<U_{2N}=1$ such that
\begin{equation*}
F(U)+\beta \begin{cases} =0 & \text{for }  U=U_i, \ i=1,...,2N-1, \\
>0  & \text{for } U\in P_i:=(U_{2i},U_{2i+1}), \ i=0,...,N-1, \\
<0  & \text{for } U\in N_i:=(U_{2i+1},U_{2i+2}), \ i=0,...,N-1;
\end{cases}
\end{equation*}
\item the tangency points are not degenerate:
\begin{equation*}
(-1)^i F'(U_i) > 0 \quad \forall i=1,..., 2N-1;
\end{equation*}
\item the sets
\begin{equation*}
\gls{Cs} := \big\{ U \in (U^*, 1): \Gamma_s(U,\beta) = P^*(U) \big\} \quad \text{and} \quad \gls{Cu} := \big\{ U \in (U^*, 1): \Gamma_u(U,\beta) = P^*(U) \big\}
\end{equation*}
are finite;
\item the points in $C_s \cup C_u$ correspond to transverse crossings:
\begin{equation*}
F(U) + \beta \neq 0 \quad \forall U \in C_s \cup C_u.
\end{equation*}
\end{enumerate}
\end{definition}

Since the properties above are valid in an open set $(\beta,f) \in \R \times C^3([0,1])$ with the product topology because the functions involved are continuous, we can collect the results of this section to obtain a proof of Theorem \ref{Theo:transv_generic_intro}:

\begin{theorem}
\label{Theo:transv_generic}
There is a dense open set \gls{Tfrak} in the space of admissible sources $f$ w.r.t. the norm $C^3([0,1])$, such that:
\begin{enumerate}
\item it contains a dense set of polynomials;
\item  \label{Point3:trandv_generic} the function \gls{Fdef} is such that whenever $F'(U) = 0$ then $F''(U) \not= 0$;
\item for all $f \in \mathfrak T$ there exists a finite set $\mathcal N_0 \cup \mathcal N_1$ given by Propositions \ref{zeros} and Proposition \ref{Prop:transv_Gamma_gen} such that the transversality set \gls{Tcalf} of $f$ is
$$
\gls{Tcalf} = [\beta^*,\infty) \setminus (\mathcal N_0 \cup \mathcal N_1);
$$
\item the set
\begin{equation*}
\bigcup_{f \in \mathfrak T} \mathcal T_f \times \{f\} \subset \gls{Dcal} = \Big\{ (\beta,f) \in \R \times C^3([0,1]), \beta > \beta^*(f), f \ \text{satisfying Assumption \eqref{H1}} \Big\}
\end{equation*}
is open and dense in the product topology, and of full measure in the subset
\begin{equation*}
\gls{Dcalk} = \bigg\{ (\beta,a_1,\dots,a_k) \in \R^{k+1}, f = \sum_{i=1}^K a_i U^i,\ \beta \geq \beta^*(f), \ f \ \text{satisfies Assumption \eqref{H1}} \bigg\}.
\end{equation*}
\end{enumerate}
\end{theorem}

\begin{proof}
Being the functions
\begin{equation*}
(U,\beta) \mapsto F(U) + \beta, \ F'(U),\  F''(U), \ \Gamma_u(U,\beta), \ \Gamma_s(U,\beta)
\end{equation*}
continuous in $C^3((U^*,1))$, the conditions listed in Definition \ref{Def:trans_calT} together with Point \eqref{Point3:trandv_generic} define an open set $\mathfrak T$ in $C^3$. Indeed for speeds $\beta = - F(U)$ with $F'(U) = 0$ we have that $F''(U) \not= 0$ and then at most two zeros of $F(U) + \beta = 0$ are generated or canceled. Moreover, the function $H(U) = \Gamma_u(U,-F(U)) - P^*(U)$ characterizing the tangential intersection of $\Gamma_u$ with $P^*$ satisfies at a zero of $\frac{dH}{dU}$
\begin{equation*}
\frac{dF}{dU} = 0, \quad \frac{d^2H}{dU^2} = - \partial_\beta \Gamma_u \frac{d^2 F}{dU^2}.
\end{equation*}
Thus in the case $\frac{dF}{dU} = 0$ it holds $\frac{d^2H}{dU^2} \not= 0$ because of Point \ref{Point3:trandv_generic}, and hence again at most 2 roots are generated or canceled. The same argument applies to $\Gamma_s$ since $\partial_\beta \Gamma_s > 0$.

By Propositions \ref{Prop:F_transv_poly} and \ref{Prop:Gamma_trans_poly} this set $\mathfrak T$ contains a set of polynomials which is dense in the space of polynomial sources, and it is of full measure being its complement the union of a closed set of Hausdorff dimension $k-1$ (Proposition \ref{Prop:F_transv_poly}) and the graphs of finitely many functions $\beta(f)$ (Proposition \ref{Prop:Gamma_trans_poly}).

By Weierstrass Theorem \cite[Section 12.3]{royden2010real} we conclude that $\mathfrak T$ must be dense in $C^3(U^*,1)$.
\end{proof}

A generic situation is presented in Fig. \ref{Fig:transv_generic}.

\begin{figure}
\resizebox{\textwidth}{!}{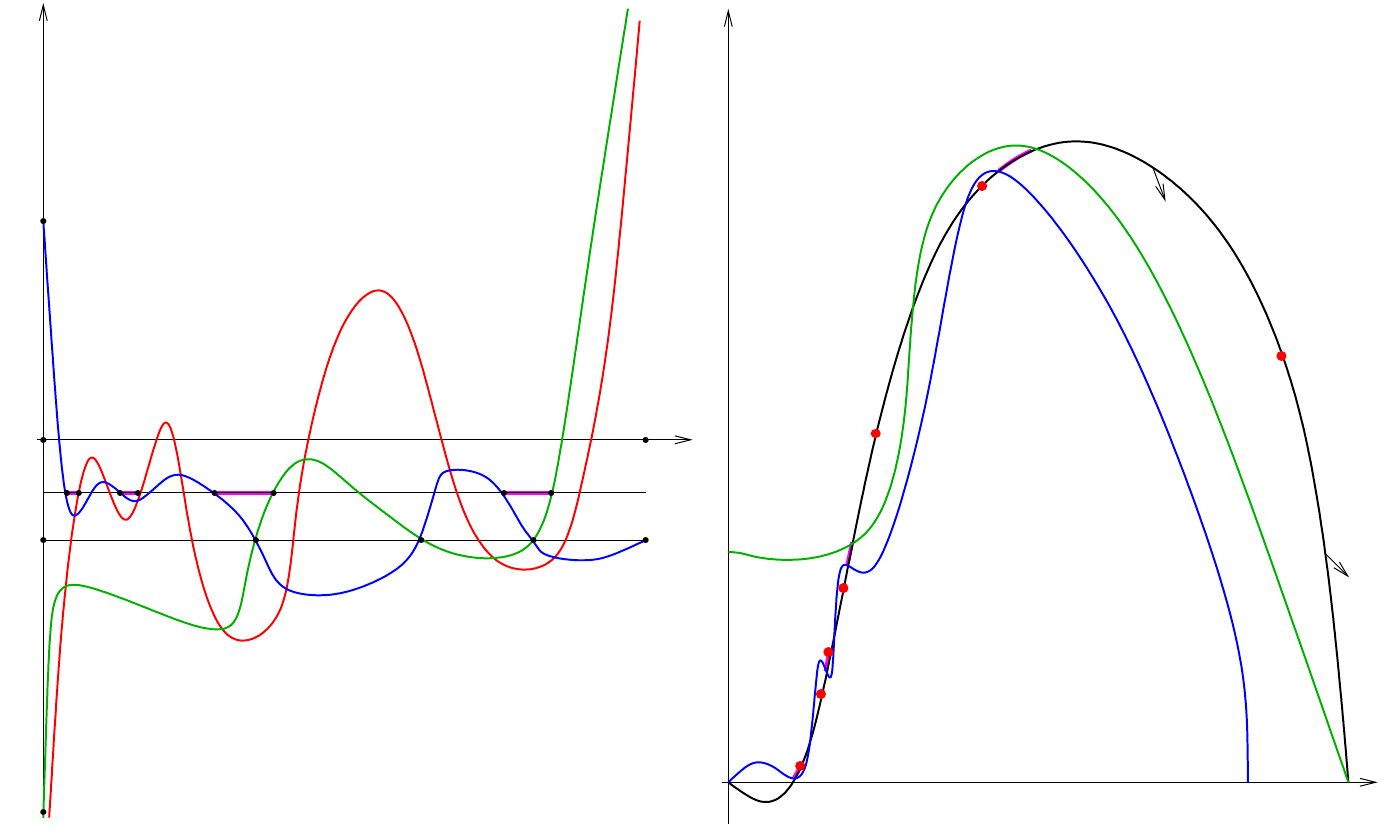}
\caption{A generic situation where Theorem \ref{Theo:transv_generic} holds. On the left the sets $\{F(U) + \beta = 0\}$ (red), $\{\Gamma_u(U,\beta) = P^*(U)\}$ (blue), $\{\Gamma_s(U,\beta) = P^*(U)\}$ (red) are depicted. Fixing $\beta$, the region where $- F(U) > \beta$ is the region where the vector field is exiting the region $\{P \leq P^*\}$, while $- F(U) < \beta$ the vector field is entering, so that the control can be applied here. In order to remain inside $\Gamma_u,\Gamma_s$, the only place where we can apply the control are the magenta segments. On the right the curve $P^*$ and the trajectories $\Gamma_u,\Gamma_s$ are represented in the phase space $(U,P$). The red dots are the zeros of $F(U) + \beta$.}
\label{Fig:transv_generic}
\end{figure}

\begin{remark}
\label{Rem:existence_multiple_inter}
We now show that it is possible to exhibit a source $f$ satisfying condition $(\ref{H1})$ and such that $F(U) + \beta$ has infinitely many oscillation. This proves that the statement of Theorem \ref{Theo:transv_generic} is not empty, i.e. $\mathfrak T$ does not contain all admissible sources $f$.

We start by considering a function $f \in C^{2,1}([0,1])$ defined as
\begin{align*}
f(U) := \begin{cases}
\text{smooth connection} & 0 \leq U < U^*, \\
\frac{a}{U^3}-\frac{a(U_1-U)^3}{(U^*)^3(U_1-U^*)^3}, & U^* \le U < U_1, \\
\frac{a}{U^3} & U_1 \le U < U_2, \\
\frac{a}{U^3}-\frac{a(U-U_2)^3}{(1-U_2)^3} & U_2 \le U \le 1,
\end{cases}
\end{align*}
where $U_1$, $U_2$ are points we will choose later. The main property of this function is that $F(U) = 0$ on the segment $[U_1,U_2]$: we will perturb it in order to achieve countably many oscillations of $F$.

In order to have
\begin{equation*}
f'(U^*) = \frac{3a(2U^* - U_1)}{(U^*)^4(U_1 - U^*)} > 0
\end{equation*}
we must assume $U_1 < 2 U^*$. Moreover, the same condition implies that 
\begin{equation*}
f(U) = \frac{a}{U^3} \bigg( 1 - \frac{U^3 (U_1 - U)^3}{(U^*)^3 (U_1 - U^*)^3} \bigg) > 0 \quad \forall U^* \le U \le U_1.
\end{equation*}
We also note that, for $U^* \le U \le U_1$,
\begin{align*}
F(U) &= \frac{3f(U)+Uf'(U)}{2\sqrt{Uf(U)}} \\
&= \frac{3}{2} \sqrt{\frac{a}{2(U^*)^3(U_1-U^*)^3}} \frac{U(U-U_1)^2(2U-U_1)}{\sqrt{(U^*)^3(U_1-U^*)^3 -U^3(U_1-U)^3}} > 0.
\end{align*}

By the previous choice of $U_1$, the map
\begin{equation*}
[U^*, U_1] \ni U \mapsto \sqrt{(U^*)^3(U_1 - U^*)^3 -U^3(U_1-U)^3}
\end{equation*}
is increasing. Viceversa, if we also assume $U_1<\frac{16}{7+\sqrt{17}} U^*$,
\begin{align*}
\frac{d}{dU} \bigg( U (U - U_1)^2 (2U - U_1) \bigg) = (U - U_1)(8U^2 -7U_1 U + U_1^2) \le 0 \quad \forall U \in [U^*, U_1].
\end{align*}
Both together imply that $F$ is decreasing in $[U^*, U_1]$.

Computing $F(U)$ for $U \in [U_1,U_2]$ we obtain
\begin{equation*}
3 f(U) + U f'(U) = \frac{3 a}{U^3} - \frac{3 a}{U^3} = 0,
\end{equation*}
so that $F(U) = 0$ for all $U \in [U_1, U_2]$.

Finally, in the interval $[U_2,1]$,
\begin{equation*}
F(U) := - \frac{3}{2}\sqrt{\frac{a}{(1-U_2)^3}} \frac{U (U-U_2)^2 (2U - U_2)}{\sqrt{(1 - U_2)^3 - U^3(U - U_2)^3}} \le 0,
\end{equation*}
and it is decreasing. 

We thus conclude that for the unperturbed $f(U)$ the set
\begin{equation*}
J^+ = \big\{ U \in (U^*, 1): \ F(U) > 0 \big\}
\end{equation*}
is an interval, and thus the optimal profile is
\begin{equation*}
P_\beta(U) := \begin{cases}
\Gamma_u(\beta, U) & \Gamma_u(\beta, U)\ge P^*(U), \\
P^*(U) & \Gamma_u(\beta,U) < P^*(U) < \Gamma_s (\beta, U), \\
\Gamma_s(\beta, U) & \Gamma_s(\beta, U)\le P^*(U), \\
\end{cases}
\end{equation*}
by Theorem $3.2$ in \cite{BressanChiri23} or Corollary 4.3 of \cite{bianchini2026travelingprofilescontrolcost}.

\begin{figure}
\resizebox{.65\textwidth}{!}{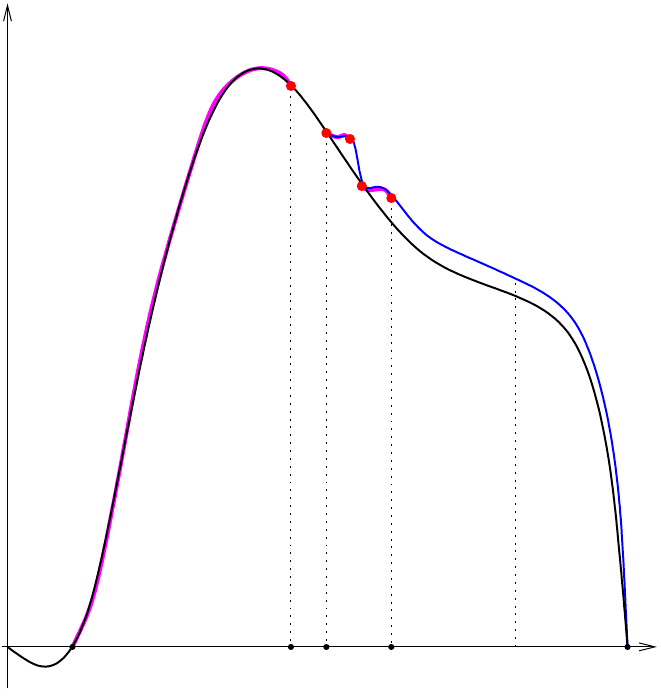}
\caption{The function $f(U)$ and its perturbation $f_\kappa(U)$, Remark \ref{Rem:existence_multiple_inter}. The magenta regions are where $F(U) + \beta > 0$.}
\label{Fig:cubic_hyperbol_perp_1}
\end{figure}

We will now perturb $f(U)$ into countably many disjoint intervals contained in $[U_1,U_2]$ so that $F(U)$ changes sign in each of these intervals. It is clear that we have only to give the perturbation in a single interval $I = (U^-_I,U^+_I) \subset [U_1,U_2]$. Set then
\begin{equation*}
f_\kappa(U) := \begin{cases}
\text{smooth connection} & 0 \leq U < U^*, \\
\frac{a}{U^3} - \frac{a(U_1-U)^3}{(U^*)^3(U_1-U^*)^3}, & \ U^* \le U < U_1, \\
\frac{a}{U^3} & \ U_1 \le U < U^-_I, \\
\frac{a}{U^3} + \kappa \sin \big( \frac{3\pi}{2|I|} (U - U^-) \big)^2 & \ U_I^- \le U < U^+_I, \\
\frac{a + \kappa (U^+_I)^3}{U^3} & \ U^-_I \le U < U^+_I, \\
\frac{a + \kappa (U^+_I)^3}{U^3}-\frac{(a + \kappa (U^+_I)^3)(U-U_2)^3}{(1-U_2)^3} & \ U_I^+ \le U \le 1,
\end{cases}
\end{equation*}
for some $\kappa > 0$ and compute in $I$
\begin{equation*}
\begin{split}
3 f_\kappa(U) + U \frac{d}{dU} f_\kappa(U) = 3 \kappa \sin \bigg( \frac{3\pi}{2|I|} (U - U^-) \bigg)^2 + \frac{3 \pi \kappa}{2|I|} U \sin \bigg( \frac{3\pi}{2I} (U - U^-) \bigg) \cos\bigg( \frac{3\pi}{2|I|} (U - U^-) \bigg).
\end{split}
\end{equation*}
It follows directly that this function is positive near $U^-_I,U^+_I$, while
\begin{equation*}
\begin{split}
3 f_\kappa \bigg( \frac{U^-_I + U^+_I}{2} \bigg) + \frac{U^-_I + U^+_I}{2} \frac{d}{dU} f_\kappa \bigg( \frac{U^- + U^+}{2} \bigg) &= 3 \kappa \sin \bigg( \frac{3\pi}{4} \bigg)^2 - \frac{3 \pi \kappa}{2 |I|} \frac{U^- + U^+}{2} \sin \bigg( \frac{3\pi}{4} \bigg) \cos \bigg( \frac{3\pi}{4} \bigg) \\
&= \frac{3\kappa}{2} \bigg( 1 - \frac{\pi}{2} \frac{U^+ + U^-}{2(U^+ - U^-)} \bigg) < 0.
\end{split}
\end{equation*}
Hence the function has an oscillation.

We observe that the perturbation is not $C^3$, but a slight variation will not alter the sign of $F(U)$. 
%
\end{remark}

\section{Dependence of the manifolds $\Gamma_u$, $\Gamma_s$ with respect to parameters}
\label{Ss:stab_sti_Gamma_us}

In this section, we study the dependence of the unstable manifold $\Gamma_u$ and the stable manifold $\Gamma_s$ with respect to $\beta$ and the source $f$. It is well known that the stable and unstable manifolds depend smoothly on the vector field; our aim is to provide direct proofs of some explicit estimates which will be important in the analysis of the differentiability properties of the Effort function \gls{Ebeta}. A critical case is the heteroclinic connection at $\beta^{**}$, where the dependence of \gls{Gammau} may fail to be smooth.

It is elementary from the ODEs \eqref{Equa:Gammau_eq0}, \eqref{Equa:Gammas_eq0} to see that the functions
$$
\beta,f \mapsto \Gamma_u(U,\beta,f), \ \gls{Uufbeta} = U_{uf}(\beta,f)
$$
are monotonically decreasing, as well as
$$
\beta,f \mapsto \Gamma_s(U,\beta,f)
$$
is monotonically increasing and always defined in $[0,1]$ if $\beta \geq \beta^*(f)$. Recall moreover that \gls{Uubeta} is the first intersection point of \gls{Gammau} and \gls{Pstar} and, similarly, \gls{Usbeta} is the last intersection point of \gls{Gammas} and \gls{Pstar}, see page \pageref{Equa:ODE_contr_PU_intro}.  

For the solutions $P(U)$ we are considering (see Section \ref{Ss:optim_Pbeta}), the intervals of interests of $\Gamma_u,\Gamma_s$ are $[0,\gls{Uubeta}]$, $[\gls{Usbeta},1]$, where $U_u(\beta)$ ($U_s(\beta)$) is the first (last) intersection points of $\Gamma_u(\beta)$ ($\Gamma_s(\beta)$) with the critical curve
\begin{equation*}
\gls{Pstar} = \sqrt{Uf(U)}, \quad U \in [U^*,1].
\end{equation*}
We recall that for $\beta \geq \beta^{**}$ we have $U_u(\beta) = U^*$; observe also that $U_u(\beta) > U^*$ and $\Gamma_u(U_u) > 0$ for $\beta < \beta^{**}$. Moreover, since the derivative of $P^*$ in $U=1$ is $-\infty$, the point $U_s(\beta)$ is well defined. When we need to emphasize the dependence on $f$ we will write $U_u(\beta,f)$, $U_s(\beta,f)$, or for shortness $U_u(f)$, $U_s(f)$ when $\beta$ is kept constant.


We need first the following lemma: it allows to compare the trajectory $\Gamma_u$ with the function $f(U)$ in some range of speeds $\beta$. In particular, one can deduce that $\gls{betastarstar} < \infty$.

\begin{lemma}
\label{Lem:large_beta_Gammau}
Define the constants
\begin{equation}
\label{Equa:def_M0_MU*}
M_0^+ = \max_{[0,U^*]} \bigg\{ \frac{f'(U)}{f'(0)} \bigg\} \geq 1, \quad M_{0}^- = \max_{[0,U^*]} \bigg\{ \frac{- f'(U)}{f'(0)} \bigg\} > 0.
\end{equation}
Then it holds for $U \in [0,U^*]$
\begin{equation}
\label{Equa:lower_bdd_a}
\Gamma_u(U,\beta) \geq \bigg( - \frac{\beta}{2M_0^+} + \sqrt{ \frac{\beta^2}{4 (M_0^+)^2} - \frac{f'(0)}{M_0^+}} \bigg) \frac{f(U)}{f'(0)}, \quad \beta \geq \beta^*.
\end{equation}
Moreover, if
\begin{equation*}
\beta \geq 2 \sqrt{- M_{0}^- f'(0)} = 2 \sqrt{\max_{[0,U^*]} \{f'(U)\}}.
\end{equation*}
then
\begin{equation*}
\Gamma_u(U,\beta) \leq \bigg( \frac{\beta}{2M_0^-} - \sqrt{\frac{\beta^2}{4 (M_0^-)^2} + \frac{f'(0)}{M_0^-}} \bigg) \frac{f(U)}{f'(0)}.
\end{equation*}
\end{lemma}

\begin{proof}
Consider the function
\begin{equation*}
P_a(U) = a \frac{f(U)}{f'(0)}, \quad a > 0, \ U \in [0,U^*].
\end{equation*}
If
\begin{equation}
\label{Equa:a_subsol}
a \frac{f'(U)}{f'(0)} < - \frac{f'(0)}{a} - \beta,
\end{equation}
then $P_a$ is a subsolution to \eqref{Equa:ODE_nocontr_PU}. Recalling $M_0^+$ given by \eqref{Equa:def_M0_MU*}, the inequality \eqref{Equa:a_subsol} requires that
\begin{equation*}
M_0^+ a^2 + \beta a + f'(0) < 0, \quad 0 < a < - \frac{\beta}{2M_0^+} + \sqrt{ \frac{\beta^2}{4 (M_0^+)^2} - \frac{f'(0)}{M_0^+}}.
\end{equation*}
We need to check that the inequality holds near $U = 0$. Since $M_0^+ \geq 1$, then
\begin{equation*}
- \frac{\beta}{2M_0^+} + \sqrt{ \frac{\beta^2}{4 (M_0^+)^2} - \frac{f'(0)}{M_0^+}} = \frac{\lambda_+(\frac{\beta}{\sqrt{M_0^+}},0)}{\sqrt{M_0^+}} \leq - \frac{\beta}{2} + \sqrt{ \frac{\beta^2}{4} - f'(0)} = \lambda_+(0,\beta)
\end{equation*}
because
\begin{equation*}
\partial_{M_0^+} \bigg( - \frac{\beta}{2M_0^+} + \sqrt{\frac{\beta^2}{4 (M_0^+)^2} - \frac{f'(0)}{M_0^+}} \bigg) < 0, 
\end{equation*}
and we conclude that $\Gamma_u(U) \geq a \frac{f(U)}{f'(0)}$ for all $a < \frac{\lambda_+(\frac{\beta}{\sqrt{M_0^+}},0)}{\sqrt{M_0^+}}$. Letting $a \nearrow \frac{\lambda_+(\frac{\beta}{\sqrt{M_0^+}},0)}{\sqrt{M_0^+}}$ we obtain \eqref{Equa:lower_bdd_a}.

In a similar way, $P(U) = a \frac{f(U)}{f'(0)}$ is a supersolution if
\begin{equation*}
a \frac{f'(U)}{f'(0)} > - \frac{f'(0)}{a} - \beta,
\end{equation*}
which requires
\begin{equation*}
- a^2 M_0^- + \beta a + f'(0) > 0, \quad M_0^- = \max_{[0,U^*]} \bigg\{ \frac{-f'(u)}{f'(0)} \bigg\}.
\end{equation*}
Hence we obtain
\begin{equation*}
\frac{\beta}{2M^-_0} - \sqrt{ \frac{\beta^2}{4 (M^-_0)^2} + \frac{f'(0)}{M^-_0}} < a < \frac{\beta}{2M^-_0} + \sqrt{ \frac{\beta^2}{4 (M^-_0)^2} + \frac{f'(0)}{M^-_0}}.
\end{equation*}
For $\beta > 2 \sqrt{- M^-_0 f'(0)}$ we have
\begin{equation*}
\begin{split}
\lambda_+(0,\beta) &\leq \frac{\beta}{2M^-_0} - \sqrt{ \frac{\beta^2}{4 (M^-_0)^2} + \frac{f'(0)}{M^-_0}}
\end{split}
\end{equation*}
because
\begin{equation*}
- \lambda_+(0,\beta)^2 M_0^- + \lambda_+(0,\beta) \beta + f'(0) = - \lambda_+(0,\beta)^2 (M_0^- + 1) < 0.
\end{equation*}
In particular for $\beta > 2 \sqrt{-M^-_0 f'(0)}$ we have
\begin{equation*}
- \frac{\beta}{2M^+_0} + \sqrt{ \frac{\beta^2}{4 (M^+_0)^2} - \frac{f'(0)}{M^+_0}} \leq \lambda_+(0,\beta) \leq \frac{\beta}{2M^-_0} - \sqrt{ \frac{\beta^2}{4 (M^-_0)^2} + \frac{f'(0)}{M^-_0}}. \qedhere
\end{equation*}
\end{proof}

We can now prove the main theorem of this section, which concern the first order derivative of $\Gamma_u,\Gamma_s$ as stated in Theorem \ref{manifoldbound_intro}.

\begin{theorem}
\label{manifoldbound}
Let $\beta \ge \beta^*$ and $f \in C^2([0,1])$ satisfying hypothesis $(\ref{H1})$.

\begin{enumerate}
\item Derivative w.r.t. $\beta$: it holds
\begin{equation}
\label{Equa:Gammau_beta}
- 1 - \bigg[ \ln \bigg(\frac{U}{U^*} \bigg) \bigg]^+ < \frac{\partial_\beta \Gamma_u(U,\beta)}{U} < 0, \quad U \in (0,U_u(\beta)),
\end{equation}
\begin{equation}
\label{Equa:Gammas_beta}
0 < \frac{\partial_\beta \Gamma_s(U,\beta)}{1 - U} < 1, \quad U \in [U^*,1).
\end{equation}

\item 
Derivative w.r.t. $f$: there are two constants 
\begin{equation*}
\gls{rbar} = \underline{r}(\|f\|_{C^2}), \quad \gls{betaunder} = \underline{\beta}(\|f\|_{C^2}) < \beta^{**}(f)
\end{equation*}
and a constant $\gls{Cfbeta} > 0$ depending only on
\begin{equation*}
\begin{split}
&\|f\|_{C_2}, \quad \max \Big\{ \Gamma_u(U)^{-1}, \beta \in [\beta^*,\underline{\beta}], U \in [\bar r,U_u] \Big\}, \\
&\max \bigg\{ \Gamma_u(U)^{-1}, \beta \in \Big[ \beta^{*},2 \sqrt{\max_{[0,U^*]} f'(U)} \Big], U \in [\underline{r},U_u - \underline{r}] \bigg\},
\end{split}
\end{equation*}
such that if $\Gamma_u(U,f + \epsilon \delta f)$, $\Gamma_s(U,f+\epsilon \delta f)$ are the unstable, stable manifold for the source $f + \epsilon \delta f$ satisfying \eqref{H1} with $\delta f \geq 0$, then 
\begin{equation}
\label{Equa:estimat_f_uGG}
0 \le - \lim_{\epsilon \to 0} \frac{\Gamma_u(U,f+\epsilon \delta f) - \Gamma_u(U,f)}{U} \le C(f) \|\delta f\|_{C^1},
\end{equation}
for all $U \leq U_u(f)$, and
\begin{equation}
\label{Equa:estimat_f_sGG}
0 \le \lim_{\epsilon \to 0} \frac{\Gamma_s(U,f+\epsilon \delta f) - \Gamma_s(U,f)}{1 - U} \le C(f) \|\delta f\|_{C^1},
\end{equation}
for all $U \in (U^*,1)$.
\end{enumerate}
\end{theorem}


\begin{proof}
We divide the proof in several steps. Steps 1.1 and 1.2 concern the dependence of $\Gamma_u$ and $\Gamma_s$ w.r.t. $\beta$, respectively. Steps 2.1 to 2.6 analyze the dependence of $\Gamma_u$ w.r.t. $f$: the most difficult part is for $\beta$ close to the critical value $\beta^{**}$, where the asymptotic behavior of $\Gamma_u(U)$ changes sharply. Finally the analysis of $\Gamma_s$ in Step 3 concludes the proof.

\medskip

\noindent {\it Step 1: dependence w.r.t. $\beta$.} We first study the dependence w.r.t. $\beta$ for $\Gamma_u(\beta)$, $\Gamma_s(\beta)$. 

\smallskip

{\it Step 1.1: analysis of $\Gamma_u(\beta)$.} Recall that the ODE for $\Gamma_u$ is
\begin{equation*}
\dot \Gamma_u = - \frac{f(U)}{\Gamma_u} - \beta, \quad \lim_{U \to 0} \frac{\Gamma_u}{U} = \lambda_+(0,\beta) = - \frac{\beta}{2} + \sqrt{ \frac{\beta^2}{4} - f'(0)}.
\end{equation*}
%
Thus the perturbation $\delta \Gamma_u = \partial_\beta\Gamma_u(U,\beta)$ solves
\begin{equation*}
\dot{\delta \Gamma_u} = \frac{f(U)}{\Gamma_u^2} \delta \Gamma_u - 1, \quad \lim_{U \to 0} \frac{\delta \Gamma_u}{U} = \partial_U \lambda_+(0,\beta) = \frac{\beta}{4 \sqrt{ \frac{\beta^2}{4} - f'(0)}} - \frac{1}{2}.
\end{equation*}
The solution is given by the formula
\begin{equation}
\label{Equa:formula_Gammabeta}
\delta \Gamma_u(U,\beta) = \delta \Gamma_u(\beta,\epsilon) e^{\int_\epsilon^U \frac{f(W)}{\Gamma_u(W,\beta)^2} dW} - \int_\epsilon^U e^{\int_V^U \frac{f(W)}{\Gamma_u(W,\beta)^2} dW} dV,
\end{equation}
and since 
\begin{equation*}
\frac{f(W)}{\Gamma_u(W,\beta)^2} \sim\frac{f'(0)}{\lambda_+(0,\beta)^2 W} \quad \text{as $W \to 0$},
\end{equation*}
we deduce that
\begin{equation}
\label{Equa:sol_der_beta_u}
\delta \Gamma_u(U,\beta) = - \int_0^U e^{\int_V^U \frac{f(W)}{\Gamma_u(W,\beta)^2} dW} dV = - \int_0^U \frac{\Gamma_u(V)}{\Gamma_u(U)} e^{- \int_V^U \frac{\beta}{\Gamma_u(W)} dW} dV.
\end{equation}
Since $f(U) < 0$ in $(0,U^*)$, for $U \in (0,U^*)$ we obtain the estimate
\begin{equation}
\label{Equa:delta_Gamma_u}
0 > \delta \Gamma_u(U,\beta) > - \int_0^{U} 1\, dV = - U. 
\end{equation}
Using
\begin{equation*}
\frac{d}{dU} \frac{\delta \Gamma_u}{U} = \frac{U f(U) - \Gamma_u^2}{U \Gamma_u^2} \frac{\delta \Gamma_u}{U} - \frac{1}{U},
\end{equation*}
we obtain also for $U \in [U^*,U_u(\beta,f)]$
\begin{equation*}
\begin{split}
\frac{\delta \Gamma_u(U)}{U} &= \frac{\delta \Gamma_u(U^*)}{U^*} e^{\int_{U^*}^U \frac{W f(W) - \Gamma_u(W)^2}{W \Gamma_u(W)^2} dW} - \int_{U^*}^U \frac{1}{V} e^{\int_V^U \frac{W f(W) - \Gamma_u(W)^2}{W \Gamma_u(W)^2} dW} dV \\
&\geq - 1 - \ln \bigg( \frac{U}{U^*} \bigg). 
\end{split}
\end{equation*}
being $\Gamma_u(U) \geq P^*(U) = \sqrt{Uf(U)}$ in the interval under consideration. Hence \eqref{Equa:Gammau_beta} is proved.

\smallskip

{\it Step 1.2: analysis of $\Gamma_s(\beta)$.} The analysis for $\Gamma_s(U,\beta)$ is entirely similar. The ODEs are the same, with limits
\begin{equation*}
\lim_{U \to 1} \frac{\Gamma_s}{1 - U} = - \lambda_-(1,\beta), \quad \lim_{U \to 1} \frac{\delta \Gamma_s}{1 - U} = - \partial_\beta \lambda_-(1,\beta).
\end{equation*}
The formula for the solution is
\begin{equation*}
\delta \Gamma_s(U,\beta) = - \int_1^U e^{\int_V^U \frac{f(W)}{\Gamma_s(W,\beta)^2} dW} dV = \int_U^1 e^{- \int_U^V \frac{f(W)}{\Gamma_s(W,\beta)^2} dW} dV,
\end{equation*}
and since $f(U) > 0$ in $(U^*,1)$ we obtain
\begin{equation*}
0 < \delta \Gamma_s(U,\beta) \leq 1 - U, 
\end{equation*}
which is \eqref{Equa:Gammas_beta}.

\medskip

\noindent{\it Step 2: $\Gamma_u$ dependence w.r.t. $f$.} 
The equation for the perturbation $\delta \Gamma_u(U,\beta)$ of $\Gamma_u(U,\beta)$ when $f \to f + \delta f$ satisfies
\begin{equation*}
\frac{d}{dU} \delta \Gamma_u = \frac{f(U)}{\Gamma_u^2} \delta \Gamma_u - \frac{\delta f(U)}{\Gamma_u}, \quad \lim_{U \to 0} \frac{\delta \Gamma_u(U,\beta)}{U} = \partial_{f'(0)} \lambda_+(0,\beta).
\end{equation*}
As in the previous step, noticing that
\begin{equation*}
\frac{f(U)}{\Gamma_u^2} \underset{U \to 0}{\sim} \frac{f'(0)}{\lambda_+(0,\beta)^2 U},
\end{equation*}
we obtain the representation
\begin{equation}
\label{Equa:deltaGamma_uu}
\delta \Gamma_u(U) = - \int_0^U \frac{\delta f(V)}{\Gamma_u(V)} e^{\int_V^U \frac{f(W)}{\Gamma_u(W)^2} dW} dV.
\end{equation}
Using the ODE \eqref{Equa:Gammau_eq0} for $\Gamma_u(U,\beta)$ we can write
\begin{equation*}
\begin{split}
\int_V^U \frac{f(W)}{\Gamma_u(W)^2} dW &= \int_V^U \bigg( - \frac{1}{\Gamma_u(W)} \frac{d\Gamma_u}{dW} - \frac{\beta}{\Gamma_u(W)} \bigg) dW \\
&= \ln \bigg( \frac{\Gamma_u(V)}{\Gamma_u(U)} \bigg) - \int_V^U \frac{\beta}{\Gamma_u(W)} dW,
\end{split}
\end{equation*}
so that we also obtain
\begin{equation}
\label{Equa:deltaGamma_u}
\frac{\delta \Gamma_u(U)}{U} = - \int_0^U \frac{\delta f(V)}{U \Gamma_u(U)} e^{- \int_V^U \frac{\beta}{\Gamma_u(W)} dW} dV.
\end{equation}

In both cases $\delta \Gamma_u \leq 0$, so that the sign in \eqref{Equa:estimat_f_uGG} is correct.

\smallskip

{\it Step 2.1: estimates for $\beta > 2\sqrt{\max_{[0,U^*]} f'(U)}$.} 
Using Lemma \ref{Lem:large_beta_Gammau}, for $\beta \geq 2 \sqrt{\max_{[0,U^*]} f'(U)}$ we estimate \eqref{Equa:deltaGamma_uu} as 
\begin{equation*}
\begin{split}
\bigg| \frac{\delta \Gamma_u(U)}{U} \bigg| &\leq \frac{\|\delta f\|_{C^0([0,U])}}{U} \int_0^U \frac{f'(0)}{\Big( - \frac{\beta}{2M_0^+} + \sqrt{\frac{\beta^2}{4 (M_0^+)^2} - \frac{f'(0)}{M_0^+}} \Big) f(V)} e^{\int_V^U \frac{f'(0)^2}{\Big( \frac{\beta}{2M_0^-} - \sqrt{\frac{\beta^2}{4 (M_0^-)^2} + \frac{f'(0)}{M_0^-}} \Big)^2 f(W)} dW} dV \\
&= \frac{\|\delta f\|_{C^0([0,U])}}{U} \frac{\Big( \frac{\beta}{2M_0^-} - \sqrt{\frac{\beta^2}{4 (M_0^-)^2} + \frac{f'(0)}{M_0^-}} \Big)^2}{\Big( - \frac{\beta}{2M_0^+} + \sqrt{\frac{\beta^2}{4 (M_0^+)^2} - \frac{f'(0)}{M_0^+}} \Big) f'(0)} \\
& \quad \quad \cdot \int_0^U \frac{f'(0)^2}{\Big( \frac{\beta}{2M_0^-} - \sqrt{\frac{\beta^2}{4 (M_0^-)^2} + \frac{f'(0)}{M_0^-}} \Big)^2 f(V)} e^{\int_V^U \frac{f'(0)^2}{\Big( \frac{\beta}{2M_{U^*}} - \sqrt{\frac{\beta^2}{4 M_{U^*}^2} - \frac{f'(U^*)}{M_{U^*}}} \Big)^2 f(W)} dW} dV \\
&= \frac{\|\delta f\|_{C^0([0,U])}}{U} \frac{\Big( \frac{\beta}{2M_0^-} - \sqrt{\frac{\beta^2}{4 (M_0^-)^2} + \frac{f'(0)}{M_0^-}} \Big)^2}{\Big( - \frac{\beta}{2M_0^+} + \sqrt{\frac{\beta^2}{4 (M_0^+)^2} - \frac{f'(0)}{M_0^+}} \Big) f'(0)} \cdot \left. - e^{\int_V^U \frac{f'(U^*)^2}{\Big( \frac{\beta}{2M_0^-} - \sqrt{\frac{\beta^2}{4 (M_0^-)^2} + \frac{f'(0)}{M_0^-}} \Big)^2 f(W)} dW} \right|_0^U \\
&\leq C \frac{\|\delta f\|_{C^1}}{\beta}, 
\end{split}
\end{equation*}
where the constant $C$ depends only on $f'(0),M_0^-,M_0^+$, and we observed that as $\beta \to \infty$
\begin{equation*}
- \frac{\beta}{2M_0^+} + \sqrt{\frac{\beta^2}{4 (M_0^+)^2} - \frac{f'(0)}{M_0^+}}, \frac{\beta}{2M_0^-} - \sqrt{\frac{\beta^2}{4 (M_0^-)^2} + \frac{f'(0)}{M_0^-}} \sim \frac{1}{\beta}.
\end{equation*}

In subsequent points we need only to consider the case for $\beta$ bounded, i.e. $\beta^* \leq \beta \leq 2 \sqrt{\displaystyle{\max_{[0,U^*]} f'(U)}}$. 

\smallskip

{\it Step 2.2: estimates for $U$ close to $0$.} By the stability of the unstable manifold \cite[Theorem 6.2.3]{katokhass}, there is a neighborhood $U = 0$ such that
\begin{equation}
\label{Equa:expans_U0}
\big| \Gamma_u(U,\beta) - \lambda_+(0,\beta) U \big| \leq \mathcal O(\|f\|_{C^2}) U^2,
\end{equation}
and the radius $\gls{rbar} = r(\|f\|_{C^2})$ depends only on the $C^2$-norm of $f$ for $\beta \in [\beta^*,2 \sqrt{\displaystyle{\max_{[0,U^*]} f'(U)}}]$. 
By \eqref{Equa:expans_U0}
\begin{equation*}
\Gamma_u(U,\beta) \ge \lambda_+(0,\beta)U-\mathcal{O}(\|f\|_{C^2}) U^2, \quad \forall U\in [0, \underline{r}],
\end{equation*}
so that for all $0\le V\le U \le \underline{r}$
\begin{align*}
\int_{V}^U \frac{1}{\Gamma_u(W)} \, dW &\le \int_{V}^U \frac{1}{\lambda_+(0,\beta) W - \mathcal O(\|f\|_{C^2}) W^2} \\
&= \frac{1}{\lambda_+(0,\beta)} \ln \bigg( \frac{U}{V} \bigg) + \frac{1}{\lambda_+(0,\beta)} \ln \bigg( \frac{1 - \mathcal O(\|f\|_{C^2}) V}{1 - \mathcal O(\|f\|_{C^2}) U} \bigg),
\end{align*}
which implies, for $\beta <0$,
\begin{align*}
e^{-\int_V^U \frac{\beta}{\Gamma_u(W)} \, dW} \le C(\|f\|_{C^2}) \bigg( \frac{V}{U} \bigg)^{\frac{\beta}{\lambda_+(0,\beta)}}.
\end{align*}
Substituting into \eqref{Equa:deltaGamma_u} and observing that
\begin{equation}
\label{Equa:lambda_+0_beta}
- \lambda_+(0,\beta) < \beta,
\end{equation}
we obtain (the constant $C(\|f\|_2)$ may vary line by line)
\begin{equation*}
\begin{split}
\bigg| \frac{\delta \Gamma_u(U)}{U} \bigg| &\le \frac{C(\|f\|_{C^2}) \|\delta f\|_{C^0([0,U])}}{\Gamma_u(U) U} \int_0^U \bigg( \frac{V}{U} \bigg)^{\frac{\beta}{\lambda_+(0,\beta)}} \, dV \\
&= \frac{C(\|f\|_{C^2}) \|\delta f\|_{C^0([0,U])}}{\beta + \lambda_+(0,\beta)} \frac{\lambda_+(0,\beta)}{\Gamma_u(U)} \\
&\leq \frac{C(\|f\|_{C^2})}{\frac{\beta}{2} + \sqrt{\frac{\beta^2}{4} - f'(0)}} \frac{\|\delta f\|_{C^0([0,U])}}{U} \leq \frac{C(\|f\|_{C^2})}{\frac{\beta^*}{2} + \sqrt{\frac{\beta^*}{4} - f'(0)}} \|\delta f\|_{C^{1}}. 
\end{split}
\end{equation*}

The case $\beta\geq0$ is simpler as there holds
\begin{align*}
e^{-\int_V^U \frac{\beta}{\Gamma_u(W)} \, dW} \le 1 \quad \forall 0\leq V\leq U \leq \underline{r},
\end{align*}
and, thus, 
\begin{equation*}
\begin{split}
\bigg| \frac{\delta \Gamma_u(U)}{U} \bigg| &\le \frac{\|\delta f\|_{C^0([0,U])}}{\Gamma_u(U)}\\
&\le \frac{1}{\lambda_+(0,\beta)-\mathcal{O}(\|f\|_{C^2} )}\frac{\|\delta f\|_{C^0([0,U])}}{U}\\
&\le C(\|f\|_{C^2}) \|\delta f\|_{C^{1}}.
\end{split}
\end{equation*}

In both cases
\begin{equation*}
\bigg| \frac{\delta \Gamma_u(U)}{U} \bigg| \le C(\|f\|_{C^2}) \|\delta f\|_{C^{1}}.
\end{equation*}

%

\smallskip

{\it Step 2.3: estimates for $\beta^* \leq \beta < \beta^{**}$ and $U \geq \underline{r}$.} In this case, by \eqref{Equa:expans_U0} we can choose $a > 0$ such that
\begin{equation*}
\Gamma_u(U,\beta) \geq 
\begin{cases}
(\lambda_+(0,\beta) - a) U & 0 \leq U < \underline{r}, \\
\displaystyle{\min_{[\underline{r},U_u(\beta)]}} \{\Gamma_u\} & \underline{r} \leq U \leq U_u(\beta),
\end{cases} \quad \text{with} \ \frac{-\beta}{\lambda_+(0,\beta) - a} < 1.
\end{equation*}
Indeed, for $U < \underline{r}$ we can use the previous point and \eqref{Equa:lambda_+0_beta}, while $\Gamma_u$ is uniformly positive for $U \geq \underline{r}$ and $\beta$ bounded. 
%

Hence, for $\beta < 0$ we can estimate the integral in the exponential in \eqref{Equa:deltaGamma_uu} as
\begin{equation*}
\begin{split}
\int_V^U \frac{-\beta}{\Gamma_u(W)} dW \leq \frac{-\beta}{\displaystyle{\min_{[\underline{r},U_u(\beta)]} \{\Gamma_u\}}} + \frac{-\beta}{\lambda_+(0,\beta) - a} \ln \bigg( \frac{\underline{r}}{V} \bigg) \ind_{[0,\underline{r})}(V).
\end{split}
\end{equation*}
Using the above estimate we obtain that
\begin{equation*}
\begin{split}
\bigg| \frac{\delta \Gamma_u(U)}{U} \bigg| &\leq \frac{\|\delta f\|_{C^0([0,U])}}{U \Gamma_u(U)} \int_0^U e^{- \int_V^U \frac{\beta}{\Gamma_u(W)} dW} dV \\
&\leq \frac{\|\delta f\|_{C^0([0,U])}}{U \displaystyle{\min_{[\underline{r},U_u(\beta)]} \{\Gamma_u\}}} e^{\frac{-\beta}{\displaystyle{\min_{[\underline{r},U_u(\beta)]} \{\Gamma_u\}}}} \bigg[ (U - \underline{r}) +  \int_0^{\underline{r}} \bigg( \frac{\underline{r}}{V} \bigg)^{\frac{-\beta}{\lambda_+(0,\beta) - a}} dV \bigg] \\
&= \frac{\|\delta f\|_{C^0([0,U])}}{U \displaystyle{\min_{[\underline{r},U_u(\beta)]} \{\Gamma_u\}}} e^{\frac{-\beta}{\displaystyle{\min_{[\underline{r},U_u(\beta)]} \{\Gamma_u\}}}} \bigg[ (U - \underline{r}) + \frac{\underline{r}}{1 + \frac{\beta}{\lambda_+(0,\beta) - a}} \bigg].
\end{split}
\end{equation*}

For $0 \leq \beta < \beta^{**}$ we can simply observe that
\begin{equation*}
e^{\int_V^U \frac{-\beta}{\Gamma_u(W)} dW} \leq 1,
\end{equation*}
and then
\begin{equation*}
\begin{split}
\bigg| \frac{\delta \Gamma_u(U)}{U} \bigg| &\leq \frac{\|\delta f\|_{C^0([0,U])}}{U \Gamma_u(U)} \int_0^U e^{- \int_V^U \frac{\beta}{\Gamma_u(W)} dW} dV \leq \frac{\|\delta f\|_{C^0([0,U])}}{\displaystyle{\min_{[\underline{r},U_u(\beta)]} \{\Gamma_u\}}}. \\ 
\end{split}
\end{equation*}

Hence we conclude that 
\begin{equation*}
\bigg| \frac{\delta \Gamma_u(U)}{U} \bigg| \leq C \Big( \max_{[\underline{r},U_u(\beta)]} \{\Gamma_u^{-1}\} \Big) \|\delta f\|_{C_0}.
\end{equation*}
The quantity $\displaystyle{\min_{[r(\|f\|_{C^2}),U_u(\beta)]} \Gamma_u}$ is uniformly bounded when we will fix the maximal speed $\underline{\beta} < \beta^{**}$, see Step 2.6 below.

\smallskip

{\it Step 2.4: estimates for $\beta^{**}  < \beta < 2 \sqrt{\max_{[0,U^*]} f'(U)}$ and $U$ close to $U^*$.} 
%
%
For $\beta > \beta^{**}$ we use the fact that at $\beta^{**}$ we are on the strongly stable manifold corresponding to $\lambda_-(U^*,\beta)$ and for $\beta > \beta^{**}$ we approach along the eigenvector corresponding to $\lambda_+(U^*,\beta)$, so that by \cite[Theorem 6.2.8]{katokhass} 
\begin{equation*}
- \lambda_+(U^*,\beta) (U^* - U) - \mathcal O(\|f\|_{C^2}) (U^* - U)^2 \leq \Gamma_u(U) \leq - \lambda_-(U^*,\beta) (U^* - U) + \mathcal O(\|f\|_{C^2}) (U^* - U)^2
\end{equation*}
for $U^* - \gls{rbar}(\|f\|_{C^2}) \leq U \leq U^*$, i.e. $\Gamma_u$ lies in between the invariant manifold corresponding to $\lambda_-(U^*,\beta)$ and the one corresponding to $\lambda_+(U^*,\beta)$. Then estimate the exponential in \eqref{Equa:deltaGamma_u} as
\begin{equation*}
\begin{split}
- \int_V^U \frac{\beta}{\Gamma_u(W)} dW &\leq - \int_{\max\{V,U^* - \underline{r}\}}^{U} \frac{\beta}{\Gamma_u(W)} dV \\
&\leq \mathcal O(\|f\|_{C^2}) + \frac{\beta}{\lambda_-(U^*,\beta)} \ln \bigg( \frac{\min\{U^* - V,\underline{r}\}}{U^* - U} \bigg).
\end{split}
\end{equation*}
Hence
\begin{equation*}
\begin{split}
e^{- \int_V^U \frac{\beta}{\Gamma_u(W)} dW} 
= C(\|f\|_{C^2}) \bigg( \frac{\min\{U^*-V,\underline{r}\}}{U^*-U} \bigg)^\frac{\beta}{\lambda_-(U^*,\beta)}.
\end{split}
\end{equation*}
Substituting into \eqref{Equa:deltaGamma_u} we get for $0 < U^* - U < \underline r(\|f\|_{C^2})$
\begin{equation*}
\begin{split}
\bigg| \frac{\delta \Gamma_u}{U} \bigg| &= \frac{C(\|f\|_{C^2}) \|\delta f\|_{C^0([0,U^*]}}{U \Gamma_u(U)} \int_0^U \bigg( \frac{\min\{U^* - V,\underline{r}\}}{U^* - U} \bigg)^\frac{\beta}{\lambda_-(U^*,\beta)} dV \\
&= \frac{C(\|f\|_{C^2}) \|\delta f\|_{C^0([0,U^*]}}{U (U^* - U)} \int_0^{U^* - \underline{r}} \bigg( \frac{U^*-U}{\underline{r}} \bigg)^\frac{\beta}{-\lambda_-(U^*,\beta)} dV \\
&\quad + \frac{C(\|f\|_{C^2}) \|\delta f\|_{C^0([0,U^*]}}{U (U^* - U)} \int_{U^* - \underline{r}}^U \bigg( \frac{U^*-V}{U^*-U} \bigg)^\frac{\beta}{\lambda_-(U^*,\beta)} dV \\
&\leq \frac{C(\|f\|_{C^2}) \|\delta f\|_{C^0([0,U^*]}}{U} \frac{U^* - \underline{r}}{\underline{r}^\frac{\beta}{- \lambda_-(U^*,\beta)}} (U^* - U)^{- 1 + \frac{\beta}{- \lambda_-(U^*,\beta)}} \\
& \quad + \frac{C(\|f\|_{C^2}) \|\delta f\|_{C^0([0,U^*]}}{U} \frac{(U^* - U)^{- 1 + \frac{\beta}{-\lambda_-(U^*,\beta)}}}{- 1 + \frac{\beta}{-\lambda_-(U^*,\beta)}} \Big( (U^* - U)^{1 + \frac{\beta}{\lambda_-(U^*,\beta)}} - \underline{r}^{1 + \frac{\beta}{\lambda_-(U^*,\beta)}} \Big).
\end{split}
\end{equation*}
We have used the fact that $\beta > - \lambda_-(U^*,\beta)$. Again by
$$
-1 - \frac{\beta}{\lambda_-(U^*,\beta)} > 0,
$$
it follows that the above term is uniformly bounded for $0 < U^* - U < \underline{r}$.

\smallskip

{\it Step 2.5: estimates for $0 \leq \beta \leq 2 \sqrt{\max_{[0,U^*]} f'(U)}$ and $\underline{r} \leq U \leq U^* - \underline{r}$.} In this case we have simply
\begin{equation*}
\bigg| \frac{\delta \Gamma_u}{U} \bigg| \leq \frac{\|\delta f\|_{C^0([0,U^*]}}{\Gamma_u(U,\beta)} \leq C \Big( \max_{[\underline{r},U^* - \underline{r}]} \{\Gamma_u^{-1}\} \Big) \|\delta f\|_{C^0([0,U^*])}.
\end{equation*}
\smallskip

{\it Step 2.6: estimates for $\beta$ close to $\beta^{**}$ and $U > U^* - \underline{r}(\|f\|_{C^2})$.} Choose $0 < \underline{\beta} < \beta^{**}$ sufficiently close to $\beta^{**}$ such that for $\underline{\beta} < \beta < \beta^{**}$ the point $U_u(\beta)$ where $\Gamma_u$ intersects $P^*$ satisfies
$$
\big| (U_u(\beta)-U^*,\Gamma_u(U^*)) \big| \leq \gls{rbar}. 
$$
Using \eqref{Equa:delta_Gamma_u} for $U  = U^* - \underline{r}$, it is sufficient to require that 
\begin{equation*} 
(- \lambda_-(\beta^{**},U^*) + C(\|f\|_{C^2}) \underline{r}) \underline{r} + \beta^{**} - \underline{\beta} < \underline{r}, \quad \beta^{**} - \underline{\beta} = \mathcal O(\underline{r}) = \mathcal O(\|f\|_{C^2}).
\end{equation*}
Taking a smaller $\underline{r}$ if necessary, from the ODE \eqref{Equa:Gammau_eq0} it follows that $\Gamma_u$ is decreasing because $\Gamma_u$ is above the manifold corresponding to the eigenvalues $\lambda_-(U^*,\beta)$ (see \eqref{Equa:eigenv_equi}),
\begin{equation*}
\Gamma_u(U) \geq \big( - \lambda_-(U^*,\beta) - C(\|f\|_{C^2})(U^* - U) \big) (U^* - U),
\end{equation*}
and for $U - \underline{r} \leq U < U^*$
\begin{equation*}
\begin{split}
- \frac{f(U)}{\Gamma_u} - \beta &\leq \frac{-f(U)}{\big( - \lambda_-(U^*,\beta) - C(\|f\|_{C^2})(U^* - U) \big) (U^* - U)} - \beta \\
&\leq - \frac{f'(U^*)}{\lambda_-(U^*,\beta)} - \beta + C(\|f\|_{C^2}) \underline{r} \\
&= \lambda_-(U^*,\beta) + C(\|f\|_{C^2}) \underline{r} < 0.
\end{split}
\end{equation*}
Moreover
\begin{equation}
\label{Equa:local_Lip_U**}
\frac{d\Gamma_u}{dU} \geq - \frac{f(U_u)}{\Gamma_u(U_u)} - \beta = - \sqrt{\frac{f(U_u)}{U_u}} - \beta > - \beta - \mathcal O(\underline{r}^{\frac{1}{2}}),
\end{equation}
which means that it is uniformly Lipschitz in the region $U < U_u(\beta)$. 
We also observe that by the bound \eqref{Equa:local_Lip_U**} 
the trajectories are crossing the curve $- (\beta + \frac{1}{C(\|f\|_{C^2})}) (U - U^*)$ in a point \newglossaryentry{Uunder}{name=\ensuremath{\underline{U}},description={intersection of $\Gamma_u$ with the line $(\beta + \frac{1}{C(\|f\|_{C^2})}) (U^* - U)$}} $(\gls{Uunder},\Gamma_u(\gls{Uunder}))$ such that
\begin{equation}
\label{Equa:rel_barU_U*}
\lim_{U_u \to U^*} \frac{U_u - \underline{U}}{\Gamma_u(U_u)} = C(\|f\|_{C^2}). 
\end{equation}
%
We can thus estimate $\Gamma_u(V)$ from above as
\begin{equation*}
\Gamma_u(V) \leq 
\Gamma_u(U) - (\beta + \mathcal O(\underline{r}^{\frac{1}{2}})) (V-U), \quad V \leq U. 
\end{equation*}

Using \eqref{Equa:deltaGamma_uu}, the quantity we have to estimate is
\begin{equation*}
\begin{split}
\bigg| \int_{0}^{U} \frac{\delta f(V)}{U \Gamma_u(V)} e^{\int_V^U \frac{f(W)}{\Gamma_u(W)^2} dW} dV \bigg| &\leq \bigg| \int_{0}^{U^* - \underline{r}} \frac{\delta f(V)}{U \Gamma_u(V)} e^{\int_V^U \frac{f(W)}{\Gamma_u(W)^2} dW} dV \bigg| + \frac{\|\delta f\|_{C_0}}{U} \int_{U^* - \underline{r}}^U \frac{e^{\int_V^U \frac{f(W)}{\Gamma_u(W)^2} dW}}{\Gamma_u(V)} dV \\
&= \bigg| \frac{\delta \Gamma_u(U^* - \underline{r})}{U} \bigg| e^{\int_{U^*-\underline{r}}^U \frac{f(W)}{\Gamma_u(W)^2} dW} + \frac{\|\delta f\|_{C_0}}{U} \int_{U^* - \underline{r}}^U \frac{e^{\int_V^U \frac{f(W)}{\Gamma_u(W)^2} dW}}{\Gamma_u(V)} dV.
\end{split}
\end{equation*}
For $W\le U\le U_u(\beta)$
\begin{equation}
\label{Equa:min_ration_22}
\Gamma_u(\beta, W)^2 \ge f(W)W, 
\end{equation}
so that
\begin{align*}
\bigg| \frac{\delta \Gamma_u(U^* - \underline{r})}{U} \bigg| e^{\int_{U^*-\underline{r}}^U \frac{f(W)}{\Gamma_u(W)^2} dW} &\leq \bigg| \frac{\delta \Gamma_u(U^* - \underline{r})}{U} \bigg| e^{\int_{U^*}^U \frac{f(W)}{\Gamma_u(W)^2} dW} \\
&\leq \bigg| \frac{\delta \Gamma_u(U^* - \underline{r})}{U} \bigg| e^{\int_{U^*}^U \frac{1}{W} dW} \leq \bigg| \frac{\delta \Gamma_u(U^* - \underline{r})}{U^* - \underline{r}} \bigg|.
\end{align*}
The quantity $\frac{\delta \Gamma_u(U^* - \underline{r})}{U^*-\underline{r}}$ is already estimated in Step 2.5.

We are left with
\begin{equation}
\label{Equa:estim_last_stab}
\int_{U^* - \underline{r}}^U \frac{e^{\int_V^U \frac{f(W)}{\Gamma_u(W)^2} dW}}{\Gamma_u(V)} dV.
\end{equation}
Assume first that $U^* - \underline{r} < U \leq \underline{U}$: then, also $U^* - \underline{r} \leq W \leq \underline{U}$ and 
\begin{equation*}
\begin{split}
\frac{- f(W)}{\Gamma_u(W)} &\geq \frac{(f'(U^*) - C(\|f\|_{C^2}) \underline{r}) (U^* - W)}{\Gamma_u(\underline{U}) + (\beta + \frac{\underline{r}}{C(\|f\|_{C^2})}) (\underline{U} - W)} \\
&= \frac{(f'(U^*) - C(\|f\|_{C^2}) \underline{r}) (U^* - W)}{(\beta + \frac{\underline{r}}{C(\|f\|_{C^2})}) (U^* - \underline{U}) + (\beta + \frac{\underline{r}}{C(\|f\|_{C^2})}) (\underline{U} - W)} \\
&= \frac{(f'(U^*) - C(\|f\|_{C^2}) \underline{r}) (U^* - W)}{(\beta + \frac{\underline{r}}{C(\|f\|_{C^2})})(U^*-W)} \\
&\geq \frac{f'(U^*)}{\beta} - C(\|f\|_{C^2}) \underline{r} \geq \frac{1}{C(\|f\|_{C^2})}.
\end{split}
\end{equation*}
Then, observing that $f < 0$ in this region, \eqref{Equa:estim_last_stab} becomes for $U^* - \underline{r} \leq U \leq \underline{U}$
\begin{equation*}
\begin{split}
\int_{U^* - \underline{r}}^{U} \frac{e^{\int_V^U \frac{f(W)}{\Gamma_u(W)^2} dW}}{\Gamma_u(V)} dV &\leq \int_{U^* - \underline{r}}^U \frac{e^{- \int_V^U \frac{1}{C(\|f\|_{C^2})) \Gamma_u(W)} dW}}{\Gamma_u(V)} dV \leq C(\|f\|_{C^2}). 
\end{split}
\end{equation*}

If $\underline{U} \leq U \leq U_u$, then using \eqref{Equa:min_ration_22}
we obtain
\begin{equation*}
\begin{split}
\int_{\underline{U}}^U \frac{1}{\Gamma_u(V)} e^{\int_V^U \frac{f(W)}{\Gamma_u(W)^2} dW} dV \leq \frac{\underline{U} - U_u}{\Gamma_u(U_u)} \frac{U}{U^*} \underset{(\ref{Equa:rel_barU_U*})}{\leq} C(\|f\|_{C^2}).
\end{split}
\end{equation*}

We thus conclude that there exists $\underline{\beta} = \underline{\beta}(\|f\|_{C^2})$ such that for $\underline{\beta} < \beta < \beta^{**}$ and $U^* - \underline{r} < U < U_u$ we have
\begin{equation*}
\bigg| \frac{\delta \Gamma}{U} \bigg| \leq C \Big( \|f\|_{C^2},\max_{[\underline{r},U^* - \underline{r}], \beta \leq \beta^{**}} \{\Gamma_u^{-1}\} \Big) \|\delta f\|_{C^1([0,U^*])}.
\end{equation*}

\smallskip

{\it Final estimate for $\Gamma_u$.} There exists $\underline{r} = \underline{r}(\|f\|_{C^2})$, $\underline{\beta}(\|f\|_{C^2})$ such that for $U \in (0,U_u)$
\begin{equation*}
\bigg| \frac{\delta \Gamma_u}{U} \bigg| \leq C(f) \|\delta f\|_{C^{1}}, 
\end{equation*}
with the constant $C(f)$ depending on
\begin{equation*}
\|f\|_{C_2}, \ \max \Big\{ \Gamma_u(U)^{-1}, \beta \in [\beta^*,\underline{\beta}], U \in [\underline{r},U_u] \Big\}, \ \max \Big\{ \Gamma_u(U)^{-1}, \beta \in \Big[ \beta^{*},2 \sqrt{\max_{[0,U^*]} f'(U)} \Big], U \in [\underline{r},U_u - \underline{r}] \Big\}.
\end{equation*}
Thus, recalling that $\delta \Gamma_u \leq 0$, we have proved \eqref{Equa:estimat_f_uGG}.

\smallskip

\noindent{\it Step 3: study of the perturbation of $\Gamma_s$.} The analysis of the dependence of $\Gamma_s$ w.r.t. the source $f$ follows the same lines, with the simplification that $\Gamma_s(U)$ is positive away from $1$ and decreasing as $\lambda_-(1,\beta) (U-1)$ for $U \nearrow 1$.

The linearized ODE is
\begin{equation*}
\frac{d \delta \Gamma_s}{dU} = \frac{f(U)}{\Gamma_s^2} \delta \Gamma_s - \frac{\delta f(U)}{\Gamma_s},
\end{equation*}
so that, noticing that
$$
- \frac{f(U)}{\Gamma_s(U)^2} \sim - \frac{f'(1)}{\lambda_-(1,\beta)^2} \frac{1}{1 - U} \quad \text{as $U \to 1$}
$$
and following the analysis leading to \eqref{Equa:deltaGamma_uu}, the solution can be written as
\begin{equation*}
\delta \Gamma(U) = \int_U^1 \frac{\delta f(V)}{\Gamma(V)} e^{- \int_U^V \frac{f(W)}{\Gamma_s(W)^2} dW} dV = \int_U^1 \frac{\delta f(V)}{\Gamma_s(U)} e^{\int_U^V \frac{\beta}{\Gamma_s(W)} dW} dV.
\end{equation*}
In particular $\delta \Gamma_s \geq 0$ being $\delta f \geq 0$.

Observing that
$$
\Gamma_s(U) \geq \frac{(1- U)}{C(\|f\|_{C^2})}
$$
for some constant $C = C(\|f\|_{C^2})$ and $f \geq 0$ in $[U^*,1]$, we can estimate
\begin{equation*}
\begin{split}
|\delta \Gamma(U)| &= \bigg| \int_U^1 \frac{\delta f(V)}{1-V} \frac{1-V}{\Gamma_s(V)} e^{- \int_U^V \frac{f(W)}{\Gamma(W)^2} dW} dV \bigg| \\
&\leq \|\delta f\|_{C^1([U,1])} \int_U^1 \frac{1 - V}{\Gamma_s(V)} dV \leq C(\|f\|_{C^2}) \|\delta f\|_{C^1([U,1])} (1-U).
\end{split}
\end{equation*}
This concludes the proof of \eqref{Equa:estimat_f_sGG}.
\end{proof}

\begin{remark}
\label{Rem:precise_partial_Gamma}
The analysis of  $\Gamma_s$ is much easier since $\Gamma_s$ is uniformly positive and we know the asymptotic behavior near $U = 1$.

The main difficulty in studying $\Gamma_u$ arises when $\beta \sim \beta^{**}$: in this case the trajectory $\Gamma_u$ switches from the strongly stable manifold corresponding to $\lambda_-(U^*,\beta)$ to the stable invariant manifold corresponding to $\lambda_+(U^*,\beta)$. We will see that this has consequences for the analysis of $\partial^2_\beta \Gamma_u$ and the $C^1$-regularity of $f \mapsto \Gamma_u(U,\beta,f)$.
\end{remark}

Using the regularity of the unstable/stable manifold w.r.t. the parameters of the ODE, we can prove Fr\'echet differentiability of the maps
\begin{equation*}
(U,\beta,f) \mapsto \Gamma_u(U,\beta,f), \Gamma_s(U,\beta,f)
\end{equation*}
w.r.t. the product topology of $\R \times C^2([0,1]) \times \R$.

\begin{corollary}
\label{Cor:Frechet_Gamma}
Let $f$ satisfy Assumption \eqref{H1} and $U \in [0,U_{uf}(U,\beta))$: then the map
\begin{equation*}
(U,\beta,f) \mapsto \Gamma_u(U,\beta,f)
\end{equation*}
is Fr\'echet differentiable w.r.t. the product topology of
\begin{equation}
\label{Equa:space_funct_1}
\R \times \big\{ f \in C^2([0,1]), f(0) = f(1) = 0 \big\} \times \R,
\end{equation}
with derivative
\begin{equation*}
\begin{split}
\langle \delta \Gamma_u(U,\beta,f),(\delta \beta,\delta f,\delta U) \rangle &= - \delta \beta \int_0^U e^{\int_V^U \frac{f(W)}{\Gamma_u(W,\beta,f)^2} dW} dV \\
& \quad - \int_0^U \frac{\delta f(V)}{\Gamma_u(V,\beta,f)} e^{\int_V^U \frac{f(W)}{\Gamma_u(W,\beta,f)^2} dW} dV - \delta U \bigg( \frac{f(U)}{\Gamma_u(U,\beta,f)} + \beta \bigg).
\end{split}
\end{equation*}

Similarly, for $f$ satisfying Assumption \eqref{H1} and $U \in (U^*,1]$, the map
\begin{equation*}
(U,\beta,f) \mapsto \Gamma_s(U,\beta,f)
\end{equation*}
is Fr\'echet differentiable w.r.t. the product topology of the space \eqref{Equa:space_funct_1}, with derivative
\begin{equation*}
\begin{split}
\langle \delta \Gamma_s(U,\beta,f),(\delta \beta,\delta f,\delta U) \rangle &= \delta \beta \int_U^1 e^{-\int_U^V \frac{f(W)}{\Gamma_s(W,\beta,f)^2} dW} dV \\
& \quad + \int_U^1 \frac{\delta f(V)}{\Gamma_s(V,\beta,f)} e^{-\int_U^V \frac{f(W)}{\Gamma_s(W,\beta,f)^2} dW} dV - \delta U \bigg( \frac{f(U)}{\Gamma_s(U,\beta,f)} + \beta \bigg).
\end{split}
\end{equation*}
\end{corollary}

Note that we are not saying anything at the point $(U^*,0)$, in particular when $\beta = \beta^{**}$.

\begin{proof}
We present the proof only for $\Gamma_u$, and only for the derivative w.r.t. $\beta,f$, being the differentiability w.r.t. $U$ obtained directly from the ODE.

Recall that:
\begin{itemize}
\item near $(U,P) = (0,0)$ there is a neighborhood whose size $(0,\underline{r})$ is depending only on the $C^2$-norm of $f$, such that $\Gamma_u$ is the unique fixed point a a contraction map, depending smoothly on $\beta,f$ (\cite[Theorem 8.2]{coudeneerne_ergodic}),
\item the ODE depends smoothly on $\beta,f$.
\end{itemize}
The formula for the derivative is obtained by using \eqref{Equa:sol_der_beta_u} and \eqref{Equa:deltaGamma_uu}.
\end{proof}

In the next statement we will prove the second derivatives
\begin{equation*}
\partial_\beta^2 \Gamma_u(U,\beta) = \lim_{\delta \beta \searrow 0} \frac{\partial_\beta \Gamma_u(\beta + \delta \beta,U) - \partial_\beta \Gamma_u(U,\beta)}{\delta \beta}, \quad U \in [0,U_u(\beta)]
\end{equation*}
\begin{equation*}
\partial_\beta^2 \Gamma_s(U,\beta) = \lim_{\delta \beta \searrow 0} \frac{\partial_\beta \Gamma_s(\beta + \delta \beta,U) - \partial_\beta \Gamma_s(U,\beta)}{\delta \beta}, \quad U \in [0,1],
\end{equation*}
exist and are bounded when $\beta$ is away from the critical speed $\beta^{**}$. This completes the proof of Theorem \ref{manifoldbound_intro}. 
%

\begin{theorem}
\label{Theo:second_der_Gammaus}
The following holds.
\begin{enumerate}
\item For all $\beta \leq \underline{\beta} < \beta^{**}$ there is a constant $C(f,\underline{\beta})$ depending only on
\begin{equation*}
\begin{split}
&\|f\|_{C_2}, \quad \max \Big\{ \Gamma_u(U)^{-1}, \beta \in [\beta^*,\underline{\beta}], U \in [\bar r,U_u(\beta)] \Big\} 
\end{split}
\end{equation*}
such that
\begin{equation}
\label{Equa:beta2_Gamma_u}
\big| \partial^2_\beta \Gamma_u(U,\beta) \big| \leq C(f,\underline{\beta}), \quad U \in [0,U_u(\beta)].
\end{equation}

\item There is a constant depending only on $\|f\|_{C^2}$ such that
\begin{equation}
\label{Equa:beta2_Gamma_s}
0 \leq \partial^2_\beta \Gamma_s(U) \leq C(\|f\|_{C^2}), \quad U \in [U^*,1].
\end{equation}
\end{enumerate}
\end{theorem}

The fact that $\partial_\beta^2 \Gamma_u$ may be positive will have implications in the analysis of the convexity of the cost functional $E(\beta)$ also in the regions where $\beta \not= \beta^{**}$, see the examples in \cite{bianchini2026travelingprofilescontrolcost}.

\begin{proof}
We study separately $\Gamma_u$ and $\Gamma_s$, and we will do the estimate only for the endpoints $U_u(\beta), U_s(\beta)$.

\smallskip

\noindent{\it Step 1: analysis of $\Gamma_u$.} The equation for $\partial^2_\beta \Gamma_u$ is
\begin{equation*}
\frac{d}{dU} \partial_\beta^2 \Gamma_u = \frac{f(U)}{\Gamma_u^2} \partial_\beta^2 \Gamma_u - \frac{2 f(U)}{\Gamma_u^3} (\partial_\beta \Gamma_u)^2, \quad \lim_{U \to 0} \frac{\partial_\beta^2 \Gamma_u}{U} = \partial^2_\beta \lambda_+(0,\beta).
\end{equation*}
As in the case of $\partial_\beta \Gamma_u$ (Step 1.1 of the proof of Theorem \ref{manifoldbound}), we obtain the representation
\begin{equation}
\label{Equa:formula_beta2u}
\partial_\beta^2 \Gamma_u(U) = - \int_0^{U} \frac{2 f(V)}{\Gamma_u(V)^3} \big( \partial_\beta \Gamma_u(V) \big)^2 e^{\int_V^U \frac{f(W)}{\Gamma_u(W)^2} dW} dV.
\end{equation}

Using \eqref{Equa:Gammau_beta} we have
\begin{equation*}
\frac{f(V) \big( \partial_\beta \Gamma_u(V) \big)^2 }{\Gamma_u(V)^3} \leq C(f),
\end{equation*}
where $C(f)$ depends only on
$$
\|f\|_{C^2} \quad  \text{and} \quad \max \Big\{ \Gamma_u(U)^{-1}, \beta \in [\beta^*,\underline{\beta}], U \in [\bar r,U_u(\beta)] \Big\}.
$$
Since $f \leq 0$ for $U \leq U^*$, we then obtain
\begin{equation*}
\begin{split}
|\partial^2_\beta \Gamma_u(U_u(\beta)| &\leq C(f) \int_0^{U_u(\beta)} e^{\int_{U^*}^{U_u(\beta)} \frac{f(W)}{\Gamma_u(W)^2} dW} dV \leq C(f). 
\end{split}
\end{equation*}
which gives \eqref{Equa:beta2_Gamma_u} because of \eqref{Equa:rel_barU_U*}.

\smallskip

\noindent{\it Step 2: analysis of $\Gamma_s$.} The computations are similar, with the simplification that from \eqref{Equa:Gammas_beta} and $\Gamma_s(U) \geq \frac{1 - U}{C(\|f\|_{C^2})}$ we have
\begin{equation*}
0 \leq \frac{\partial_\beta \Gamma_s(U)}{\Gamma_s(U)}, \frac{f(U)}{\Gamma_s(U)} \leq C(\|f\|_{C^2}), \quad U \in (U^*,1).
\end{equation*}
The solution formula for $\partial_\beta^2 \Gamma_s$ is
\begin{equation*}
\partial_\beta^2 \Gamma_s(U_s) = \int_{U_s}^1 \frac{2 f(V)}{\Gamma_s(V)^3} \big( \partial_\beta \Gamma_s(V) \big)^2 e^{- \int_{U_s}^V \frac{f(W)}{\Gamma_s(W)^2} dW} dV,
\end{equation*}
and using the previous bounds one immediately deduces \eqref{Equa:beta2_Gamma_s}.
\end{proof}

Let $\gls{Ubarubeta} = \max \, \gls{Cu}$ be the last intersection of $\Gamma_u$ with $P^*$, and \newglossaryentry{Ubarsbeta}{name=\ensuremath{\overline{U}_s(\beta)},description={first point of intersection of $\gamma_u(\beta)$ with $P^*$}} $\gls{Ubarsbeta} = \min \, \gls{Cs}$ be the first intersection of $\Gamma_s$ with $P^*$. Then we have also the following

\begin{corollary}
\label{Cor:bound_hatu}
The estimate \eqref{Equa:beta2_Gamma_u} holds in the interval $[0,\overline {U}_u(\beta)]$ and $\beta^* \leq \beta \leq \underline{\beta} < \beta^{**}$.
\end{corollary}

\begin{proof}
It suffices to prove that if $\overline U_u(\beta) > U_u(\beta)$, then $\Gamma_u(U,\beta)$ is uniformly positive in $[U_u(\beta),\overline U_u(\beta)]$. We observe first that since $\Gamma_u$ is decreasing for $U > U^*$ and $P^*(U)$ increasing in $(U^*,U^* + \underline{r})$, then if $U_u(\beta) < U^* + \underline{r}$ then there are no further intersections. Hence if $\overline U_u(\beta) > U_u(\beta)$ we must have
\begin{equation*}
\Gamma_u(U) > \min_{[U^* + \bar r,1-\bar r]} f(U) > 0. \qedhere
\end{equation*}
\end{proof}

\subsection{Sharp bounds on regularity}
\label{Ss:sharp_regul}

In this last section we prove two results, which show that the regularity results obtained above are optimal. Together they provide a proof of Theorem \ref{Theo:disco_cont_intro}.

The first result below is the key estimate for the loss of regularity of $E(\beta)$ at $\beta^{**}$, see \cite{bianchini2026travelingprofilescontrolcost}. 

\begin{proposition}
\label{Prop:blowupsecond}
If
\begin{equation*}
\frac{\lambda_+(\beta^{**},U^*)}{\lambda_-(\beta^{**},U^*)} < \frac{1}{2},
\end{equation*}
then
\begin{equation*}
\lim_{\beta \nearrow \beta^{**}} \partial^2_\beta \Gamma_u(\beta,U_u(\beta)) = +\infty.
\end{equation*}
\end{proposition}

\begin{figure}
\resizebox{.75\textwidth}{!}{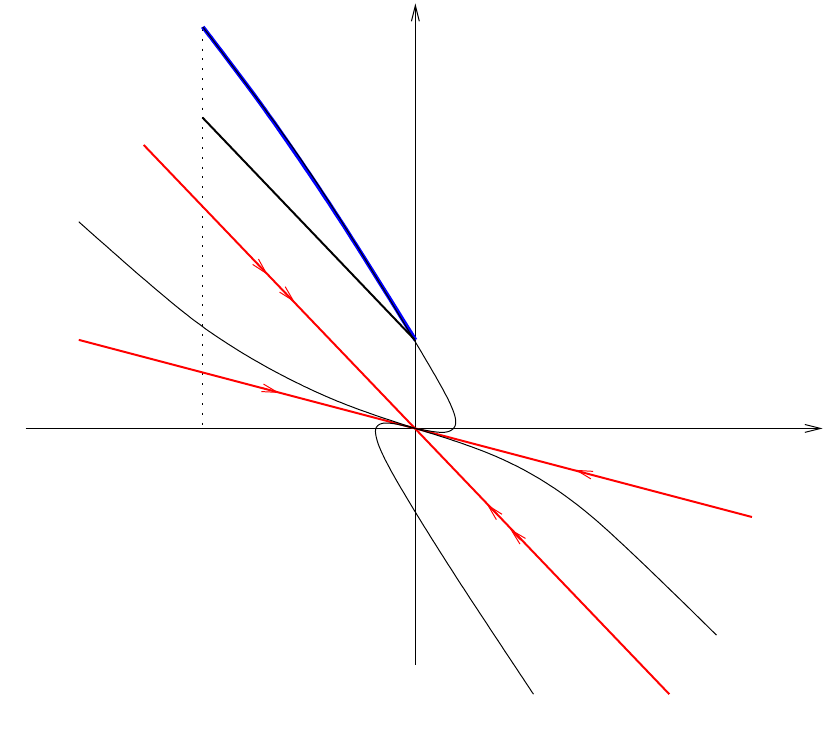}
\caption{The phase-plane picture near $(U^*,0)$ for $\beta \sim \beta^{**}$.}
\label{fig:nearUstar}
\end{figure}

\begin{proof}
If
$$
M^*_-(U,\beta) = \lambda_-(U^*,\beta) (U - U^*) + \mathcal O(\|f\|_{C^2}) (U - U^*)^2
$$
is the strongly stable invariant manifold corresponding to $\lambda_-(U^*,\beta)$ (see \cite[Theorem 6.2.8]{katokhass}), then writing
$$
P(U) = M^*_-(U,\beta) + g(U), \quad g(U) > 0,
$$
we obtain
\begin{equation*}
\frac{dg}{dU} = - \frac{f(U)}{M^*_-(U,\beta) + g(U)} - \beta + \frac{d}{dU} M^*_-(U,\beta) = \frac{f(U) g(U)}{M^*_-(U,\beta) (M^*_-(U,\beta) + g(U))} < 0
\end{equation*}
for $U \in (U^* - \underline{r},U^*)$, which gives
\begin{equation}
\label{Equa:Gammau_bound}
\Gamma_u(U^*) + M^*_-(U,\beta) \leq \Gamma_u(U) \leq \Gamma_u(U^* - \underline{r}) + \big( M^*_-(U,\beta) - M^*_-(\beta,U^* - \underline{r}) \big).
\end{equation}
A picture of trajectories near the point $(U^*,0)$ for $\beta$ close to $\beta^{**}$ is given in Fig. \ref{fig:nearUstar}.

Using the assumption that $f(U) < 0$ in $(0,U^*)$, we can estimate the following integral for $U^* - \underline{r} \leq V' \leq V \leq U^*$ as
\begin{equation*}
\begin{split}
\int_{V'}^{V} \frac{f(W)}{\Gamma_u(W,\beta)^2} dW &\geq \int_{U^*-\underline{r}}^{V} \frac{f(W)}{\Gamma_u(W,\beta)^2} dW \\
&\geq \int_{U^* - \underline{r}}^{V} \frac{f'(U^*) (W - U^*)}{(\Gamma_u(U^*) + (\lambda_-(U^*,\beta) - \mathcal O(\|f\|_{C^2}) \underline{r}) (W - U^*))^2} dW \\
&\geq - \frac{f'(U^*)}{(\lambda_-(U^*,\beta) - \mathcal O(\|f\|_{C^2}) \underline{r})^2} \ln \bigg( 1 + \frac{- \lambda_-(U^*,\beta) + \mathcal O(\|f\|_{C^2}) \underline{r}}{\Gamma_u(U^*)} (V - U^* + \underline{r}) \bigg).
\end{split}
\end{equation*}
Hence by \eqref{Equa:formula_Gammabeta}
\begin{equation}
\label{Equa:more_perc_der_beta}
\begin{split}
\partial_\beta \Gamma_u(V) &= \partial_\beta \Gamma_u(U^* - \underline{r}) e^{\int_{U^* - \underline{r}}^{V} \frac{f(W)}{\Gamma_u(W,\beta)^2} dW} - \int_{U^* - \underline{r}}^{V} e^{\int_{V'}^{V} \frac{f(W)}{\Gamma_u(W,\beta)^2} dW} dV' \\
&\leq - \mathcal O(1) \bigg( 1 + \frac{- \lambda_-(U^*,\beta) + \mathcal O(\|f\|_{C^2}) \underline{r}}{\Gamma_u(U^*)} (V - U^* + \underline{r}) \bigg)^{- \frac{f'(U^*)}{(\lambda_-(U^*,\beta) - \mathcal O(\|f\|_{C^2}) \underline{r})^2}} \\
&\leq - C_0 \Gamma_u(U^*)^{\frac{f'(U^*)}{(\lambda_-(U^*,\beta) - \mathcal O(\|f\|_{C^2}) \underline{r})^2}},
\end{split}
\end{equation}
where the constant depends only on $f$ (and $\underline r$, which in turn depends on $f$) for $V$ close to $U^*$. Formula \eqref{Equa:formula_beta2u} gives (being the second derivative positive)
\begin{equation*}
\begin{split}
\partial^2_\beta \Gamma_u(U^*) &= - \int_0^{U^*} \frac{2 f(V)}{\Gamma_u(V)^3} \big( \partial_\beta \Gamma_u(V) \big)^2 e^{\int_V^{U^*} \frac{f(W)}{\Gamma_u(W)^2} dW} dV \\
&= \partial^2_\beta \Gamma_u(U^* - \underline{r}) e^{\int_{U^* - \underline{r}}^{U^*} \frac{f(W)}{\Gamma_u(W)^2} dW} - \int_{U^* - \underline{r}}^{U^*} \frac{2 f(V)}{\Gamma_u(V)^3} \big( \partial_\beta \Gamma_u(V) \big)^2 e^{\int_V^{U^*} \frac{f(W)}{\Gamma_u(W)^2} dW} dV \\
&\geq \partial^2_\beta \Gamma_u(U^* - \underline{r}) \Gamma_u(U^*)^{\frac{f'(U^*)}{(\lambda_-(U^*,\beta) + \mathcal O(\|f\|_{C^2}) \underline{r})^2}} \\
& \quad + C_0^2 \Gamma_u(U^*)^{\frac{2 f'(U^*)}{(\lambda_-(U^*,\beta) + \mathcal O(\|f\|_{C^2}) \underline{r})^2}} \int_{U^* - \underline{r}}^{U^*} \frac{- 2 f(V)}{\Gamma_u(V)^3} e^{\int_V^{U^*} \frac{f(W)}{\Gamma_u(W)^2} dW} dV.
\end{split}
\end{equation*}
The latter integral can be estimated by integrating by parts
\begin{equation*}
\begin{split}
\int_{U^* - \underline{r}}^{U^*} \frac{- 2 f(V)}{\Gamma_u(V)^3} e^{\int_V^{U^*} \frac{f(W)}{\Gamma_u(W)^2} dW} dV &= \int_{U^* - \underline{r}}^{U^*} \frac{- f(V)}{\Gamma_u(V)^3} e^{\int_V^{U^*} \frac{f(W)}{\Gamma_u(W)^2} dW} dV \\
& \quad + \frac{e^{\int_V^{U^*} \frac{f(W)}{\Gamma_u(W)^2} dW}}{\Gamma_u(V)} \bigg|_{U^*-\underline{r}}^{U^*} - \int_{U^* - \underline{r}}^{U^*} \frac{\frac{d\Gamma(V)}{dV}}{\Gamma_u(V)^2} e^{\int_V^{U^*} \frac{f(W)}{\Gamma_u(W)^2} dW} dV \\
&= \frac{1}{\Gamma_u(U^*)} - \frac{e^{\int_{U^* - \underline{r}}^{U^*} \frac{f(W)}{\Gamma_u(W)^2} dW}}{\Gamma_u(U^* - \underline{r})} + \int_{U^* - \underline{r}}^{U^*} \frac{\beta}{\Gamma_u(V)^2} e^{\int_V^{U^*} \frac{f(W)}{\Gamma_u(W)^2} dW} dV \\
&= \frac{1 - e^{- \int_{U^* - \underline{r}}^{U^*} \frac{\beta}{\Gamma_u(W)} dW}}{\Gamma_u(U^*)} + \frac{1}{\Gamma_u(U^*)} \int_{U^* - \underline{r}}^{U^*} \frac{\beta}{\Gamma_u(V)} e^{- \int_V^{U^*} \frac{\beta}{\Gamma_u(W)} dW} dV \\
&= 2 \frac{1 - e^{- \int_{U^* - \underline{r}}^{U^*} \frac{\beta}{\Gamma_u(W)} dW}}{\Gamma_u(U^*)}. 
\end{split}
\end{equation*}
Using that the term
\begin{equation*}
\int_{U^*}^{U_u} \frac{f(V)}{\Gamma_u(V)^3} (\partial_\beta \Gamma_u(V))^2 e^{\int_{V}^{U_u} \frac{f(W)}{\Gamma_u(W)^2} dW} dV \leq C(f) \frac{(U_u - U^*)^2}{\Gamma_u(U_u)^3} \leq C_1 \sqrt{U_u - U^*},
\end{equation*}
we obtain that as $\beta \to \beta^{**}$ 
\begin{equation*}
\partial_\beta^2 \Gamma_u(U_u) \geq c \Gamma_u(U^*)^{\frac{2 f'(U^*)}{\lambda_-(U^*,\beta)^2} - 1} - C.
\end{equation*}
In particular, if
\begin{equation*}
\frac{f'(U^*)}{\lambda_-(\beta^{**},U^*)^2} = \frac{\lambda_+(\beta^{**},U^*)}{\lambda_-(\beta^{**},U^*)} < \frac{1}{2}
\end{equation*}
then the second derivative approaches $+\infty$ as $\beta \nearrow \beta^{**}$.
\end{proof}

This second result addresses the dependence of the derivatives w.r.t. $\beta$ and the dependence of $f \mapsto \Gamma_u(f)$ when $\beta \sim \beta^*$.

\begin{proposition}
\label{Prop:Lipechit_U_u}
The map
\begin{equation*}
f \mapsto \partial_\beta \Gamma_u(U_u(\beta,f),\beta,f)
\end{equation*}
is continuous in the $C^2$-topology of $f$. The map
\begin{equation*}
f \mapsto \Gamma_u(U_u(\beta,f),\beta,f)
\end{equation*}
is locally Lipschitz but in general not $C^1$.
\end{proposition}

\begin{proof}
For $\beta < \beta^{**}(f)$, we can apply the previous Corollary \ref{Cor:Frechet_Gamma} in both cases, observing that the derivatives are bounded and continuous w.r.t. $\beta,f$; for $\beta > \beta^{**}(f)$, we have $\Gamma_u(U_u = U^*) = 0$. It therefore remains only to consider the limit $\beta \nearrow \beta^{**}$.

The derivative w.r.t. $\beta$ is estimated similarly to \eqref{Equa:more_perc_der_beta}: from \eqref{Equa:Gammau_bound} we obtain
\begin{equation*}
\begin{split}
\int_{V}^{U^*} \frac{f(W)}{\Gamma_u(W)^2} dW &\gtrsim \int_{V}^{U^*} \frac{f'(U^*) (W - U^*)}{(\Gamma_u(\beta,U^* - \underline{r}) - M^*_-(\beta,U^*- \underline{r})) + \lambda_-(U^*,\beta) (W - U^*))^2} dW \\
&\sim - \frac{f'(U^*)}{\lambda_-(U^*,\beta)^2} \ln \bigg( 1 + \frac{-\lambda_-(U^*,\beta) (U^* - V)}{\Gamma_u(\beta,U^* - \underline{r}) - M^*_i(\beta,U^*- \underline{r})} \bigg),
\end{split}
\end{equation*}
and then for $\underline{\beta} < \beta < \beta^{**}$
\begin{equation*}
\begin{split}
\partial_\beta \Gamma_u(U_u(\beta,f),\beta,f) &= \partial_\beta \Gamma_u(U^*-\underline{r},\beta,f) e^{\int_{U^*-r}^{U^*} \frac{f(W)}{\Gamma_u(W,\beta)^2} dW} \\
& \qquad \qquad - \int_{U^*-\underline{r}}^{U^*} e^{\int_{V}^{U^*} \frac{f(W)}{\Gamma_u(W,\beta)^2} dW} dV - \int_{U^*}^{U_u(\beta)} \bigg( \frac{f(V)}{\Gamma_u(V)} + \beta \bigg) dV \\
&\sim \big( \Gamma_u(\beta,U^* - \underline{r}) - M^*_-(\beta,U^*- \underline{r}) \big)^{- \frac{f'(U^*)}{\lambda-(U^*,\beta)^2}} + \mathcal O(U_u - U^*).
\end{split}
\end{equation*}
Since $\Gamma_u(\beta,U^* - \underline{r}) \to M^*_-(\beta,U^* - \underline{r})$ as $\beta \nearrow \beta^{**}$, then $\partial_\beta \Gamma_u(U_u(\beta,f),\beta,f) \to 0$ as $\beta \to \beta^{**}(f)$. 

The analysis of $f \mapsto \Gamma_u(U^*(\beta,f),\beta,f)$ follows the same line: write with the notation of Step 2 of Theorem \ref{manifoldbound}
\begin{equation*}
\begin{split}
\delta \Gamma_u(U_u) &= - \int_0^{U_u} \frac{\delta f(V)}{\Gamma_u(V)} e^{\int_V^{U_u} \frac{f(W)}{\Gamma_u(W)^2} dW} dV \\
&= \delta \Gamma(U^* - \underline{r}) e^{\int_{U^* - \underline{r}}^{U_u} \frac{f(W)}{\Gamma_u(W)^2} dW} - \int_{U^* - \underline{r}}^{U^*} \frac{\delta f(V)}{\Gamma_u(V)} e^{\int_V^{U_u} \frac{f(W)}{\Gamma_u(W)^2} dW} dV - \int_{U^*}^{U_u} \frac{\delta f(V)}{\Gamma_u(V)}e^{\int_V^{U_u} \frac{f(W)}{\Gamma_u(W)^2} dW} dV \\
&\sim \delta \Gamma(U^* - \underline{r}) e^{\int_{U^* - \underline{r}}^{U_u} \frac{f(W)}{\Gamma_u(W)^2} dW} - \int_{U^* - \underline{r}}^{U^*} \frac{\delta f(V)}{\Gamma_u(V)} e^{\int_V^{U_u} \frac{f(W)}{\Gamma_u(W)^2} dW} dV - \mathcal O(\sqrt{U_u - U^*}) \|\delta f\|_{C^0}.
\end{split}
\end{equation*}
The only contribution that does not converge to zero as $\beta \nearrow \beta^{**}$ is the second integral: using that
\begin{equation*}
\Gamma_u(U,\beta) \sim \lambda_-(\beta,U^*(f)) (U - U^*(f)) 
\end{equation*}
by the linear analysis of Proposition \ref{Prop:blowupsecond}, 
\begin{equation*}
\begin{split}
\int_{U^*- \underline{r}}^{U^*} \frac{e^{\int_V^U \frac{f(W)}{\Gamma_u(W)^2} dW}}{\Gamma_u(V)} dV &\geq \int_{U^* - \underline{r}}^{U^*} e^{\frac{f'(U^*)}{\lambda_-(U^*)} \int_V^{U^*} \frac{dW}{\Gamma(U(W))}} = \frac{- \lambda_-(U^*)}{f'(U^*)} = \frac{1}{- \lambda_+(U^*,\beta)}. 
\end{split}
\end{equation*}
Hence the derivative w.r.t. the perturbation $\delta f(V)$ is not continuous, unless $\delta f(U^*) = 0$, i.e. $U^*$ remains constant.
\end{proof}

\begin{remark}
\label{Rem:const_U*}
If $U^*(f)$ is kept fixed, the previous analysis can be extended to show that the dependence on $f$ is in fact smooth. This is analogous to the smooth dependence of $\Gamma_u$ near $U=0$. We do not pursue this question here.
\end{remark}

\newpage

\appendix

\printglossaries

\newpage

\printbibliography
\end{document}

%% file: KOIdef1_intro.pdf_t
\begin{picture}(0,0)%
\includegraphics{KOIdef1_intro.pdf}%
\end{picture}%
\setlength{\unitlength}{4144sp}%
\begin{picture}(5019,5514)(5794,-6013)
\put(6256,-5731){\makebox(0,0)[lb]{\smash{\fontsize{12}{14.4}\usefont{T1}{ptm}{m}{n}{\color[rgb]{0,0,0}$U^*$}%
}}}
\put(6391,-1276){\makebox(0,0)[lb]{\smash{\fontsize{12}{14.4}\usefont{T1}{ptm}{m}{n}{\color[rgb]{1,0,0}$\Gamma_s(\beta)$}%
}}}
\put(6616,-5326){\makebox(0,0)[lb]{\smash{\fontsize{12}{14.4}\usefont{T1}{ptm}{m}{n}{\color[rgb]{1,0,0}$\Gamma_u(\beta)$}%
}}}
\put(9847,-5347){\makebox(0,0)[lb]{\smash{\fontsize{12}{14.4}\usefont{T1}{ptm}{m}{n}{\color[rgb]{0,0,0}$U$}%
}}}
\put(10633,-5657){\makebox(0,0)[lb]{\smash{\fontsize{12}{14.4}\usefont{T1}{ptm}{m}{n}{\color[rgb]{0,0,0}$1$}%
}}}
\put(6572,-2970){\makebox(0,0)[lb]{\smash{\fontsize{12}{14.4}\usefont{T1}{ptm}{m}{n}{\color[rgb]{0,0,1}$P(U)$}%
}}}
\put(8074,-1542){\makebox(0,0)[lb]{\smash{\fontsize{12}{14.4}\usefont{T1}{ptm}{m}{n}{\color[rgb]{0,0,0}$U_s(\beta)$}%
}}}
\put(9811,-2851){\makebox(0,0)[lb]{\smash{\fontsize{12}{14.4}\usefont{T1}{ptm}{m}{n}{\color[rgb]{0,0,0}$f(U)$}%
}}}
\put(8911,-1006){\makebox(0,0)[lb]{\smash{\fontsize{12}{14.4}\usefont{T1}{ptm}{m}{n}{\color[rgb]{.63,.25,0}$P^*(U) = \sqrt{U f(U)}$}%
}}}
\put(5986,-5056){\makebox(0,0)[lb]{\smash{\fontsize{12}{14.4}\usefont{T1}{ptm}{m}{n}{\color[rgb]{0,0,0}$U_u(\beta)$}%
}}}
\put(8101,-5821){\makebox(0,0)[lb]{\smash{\fontsize{12}{14.4}\usefont{T1}{ptm}{m}{n}{\color[rgb]{0,0,0}$\mathcal O(P,\beta)$}%
}}}
\put(5896,-826){\makebox(0,0)[lb]{\smash{\fontsize{12}{14.4}\usefont{T1}{ptm}{m}{n}{\color[rgb]{0,0,0}$P$}%
}}}
\end{picture}%

%% file: fU_Gammasu1.pdf_t
\begin{picture}(0,0)%
\includegraphics{fU_Gammasu1.pdf}%
\end{picture}%
\setlength{\unitlength}{4144sp}%
\begin{picture}(10509,4129)(529,-3593)
\put(1891,-2491){\makebox(0,0)[lb]{\smash{\fontsize{12}{14.4}\usefont{T1}{ptm}{m}{n}{\color[rgb]{0,0,0}$U^*$}%
}}}
\put(5131,-2491){\makebox(0,0)[lb]{\smash{\fontsize{12}{14.4}\usefont{T1}{ptm}{m}{n}{\color[rgb]{0,0,0}$1$}%
}}}
\put(10711,-3121){\makebox(0,0)[lb]{\smash{\fontsize{12}{14.4}\usefont{T1}{ptm}{m}{n}{\color[rgb]{0,0,0}$1$}%
}}}
\put(8641,-286){\makebox(0,0)[lb]{\smash{\fontsize{12}{14.4}\usefont{T1}{ptm}{m}{n}{\color[rgb]{0,0,0}$\Gamma_s(\beta) = \Gamma_s(\beta,U)$}%
}}}
\put(8371,-2446){\makebox(0,0)[lb]{\smash{\fontsize{12}{14.4}\usefont{T1}{ptm}{m}{n}{\color[rgb]{0,0,0}$\Gamma_u(\beta) = \Gamma_u(\beta,U)$}%
}}}
\put(7201,-1366){\makebox(0,0)[lb]{\smash{\fontsize{12}{14.4}\usefont{T1}{ptm}{m}{n}{\color[rgb]{0,0,1}$\Gamma_u(\beta^*) = \Gamma_s(\beta^*)$}%
}}}
\put(7201,-3076){\makebox(0,0)[lb]{\smash{\fontsize{12}{14.4}\usefont{T1}{ptm}{m}{n}{\color[rgb]{0,0,0}$U^*$}%
}}}
\put(6166,344){\makebox(0,0)[lb]{\smash{\fontsize{12}{14.4}\usefont{T1}{ptm}{m}{n}{\color[rgb]{0,0,0}$P$}%
}}}
\put(5356,-2446){\makebox(0,0)[lb]{\smash{\fontsize{12}{14.4}\usefont{T1}{ptm}{m}{n}{\color[rgb]{0,0,0}$U$}%
}}}
\put(10846,-3076){\makebox(0,0)[lb]{\smash{\fontsize{12}{14.4}\usefont{T1}{ptm}{m}{n}{\color[rgb]{0,0,0}$U$}%
}}}
\put(6976,-2581){\makebox(0,0)[lb]{\smash{\fontsize{12}{14.4}\usefont{T1}{ptm}{m}{n}{\color[rgb]{1,0,0}$\Gamma_u(\beta^{**})$}%
}}}
\put(8416,-3031){\makebox(0,0)[lb]{\smash{\fontsize{12}{14.4}\usefont{T1}{ptm}{m}{n}{\color[rgb]{0,0,0}$U_{uf}(\beta)$}%
}}}
\put(3421,-151){\makebox(0,0)[lb]{\smash{\fontsize{12}{14.4}\usefont{T1}{ptm}{m}{n}{\color[rgb]{0,0,0}$f(U)$}%
}}}
\end{picture}%

%% file: KOIdef1.pdf_t
\begin{picture}(0,0)%
\includegraphics{KOIdef1.pdf}%
\end{picture}%
\setlength{\unitlength}{4144sp}%
\begin{picture}(5019,6366)(5794,-6865)
\put(6256,-5731){\makebox(0,0)[lb]{\smash{\fontsize{12}{14.4}\usefont{T1}{ptm}{m}{n}{\color[rgb]{0,0,0}$U^*$}%
}}}
\put(6391,-1276){\makebox(0,0)[lb]{\smash{\fontsize{12}{14.4}\usefont{T1}{ptm}{m}{n}{\color[rgb]{1,0,0}$\Gamma_s(\beta)$}%
}}}
\put(6616,-5326){\makebox(0,0)[lb]{\smash{\fontsize{12}{14.4}\usefont{T1}{ptm}{m}{n}{\color[rgb]{1,0,0}$\Gamma_u(\beta)$}%
}}}
\put(8128,-5977){\makebox(0,0)[lb]{\smash{\fontsize{12}{14.4}\usefont{T1}{ptm}{m}{n}{\color[rgb]{0,0,0}$\mathcal K(P,\beta)$}%
}}}
\put(8108,-6378){\makebox(0,0)[lb]{\smash{\fontsize{12}{14.4}\usefont{T1}{ptm}{m}{n}{\color[rgb]{0,0,0}$\mathcal O(P,\beta)$}%
}}}
\put(8120,-6689){\makebox(0,0)[lb]{\smash{\fontsize{12}{14.4}\usefont{T1}{ptm}{m}{n}{\color[rgb]{0,0,0}$\mathcal I(P,\beta)$}%
}}}
\put(9847,-5347){\makebox(0,0)[lb]{\smash{\fontsize{12}{14.4}\usefont{T1}{ptm}{m}{n}{\color[rgb]{0,0,0}$U$}%
}}}
\put(10633,-5657){\makebox(0,0)[lb]{\smash{\fontsize{12}{14.4}\usefont{T1}{ptm}{m}{n}{\color[rgb]{0,0,0}$1$}%
}}}
\put(6572,-2970){\makebox(0,0)[lb]{\smash{\fontsize{12}{14.4}\usefont{T1}{ptm}{m}{n}{\color[rgb]{0,0,1}$P(U)$}%
}}}
\put(8074,-1542){\makebox(0,0)[lb]{\smash{\fontsize{12}{14.4}\usefont{T1}{ptm}{m}{n}{\color[rgb]{0,0,0}$U_s(\beta)$}%
}}}
\put(9811,-2851){\makebox(0,0)[lb]{\smash{\fontsize{12}{14.4}\usefont{T1}{ptm}{m}{n}{\color[rgb]{0,0,0}$f(U)$}%
}}}
\put(8911,-1006){\makebox(0,0)[lb]{\smash{\fontsize{12}{14.4}\usefont{T1}{ptm}{m}{n}{\color[rgb]{.63,.25,0}$P^*(U) = \sqrt{U f(U)}$}%
}}}
\put(5986,-5011){\makebox(0,0)[lb]{\smash{\fontsize{12}{14.4}\usefont{T1}{ptm}{m}{n}{\color[rgb]{0,0,0}$U_u(\beta)$}%
}}}
\put(5941,-871){\makebox(0,0)[lb]{\smash{\fontsize{12}{14.4}\usefont{T1}{ptm}{m}{n}{\color[rgb]{0,0,0}$P$}%
}}}
\end{picture}%

%% file: optproof_2.pdf_t
\begin{picture}(0,0)%
\includegraphics{optproof_2.pdf}%
\end{picture}%
\setlength{\unitlength}{4144sp}%
\begin{picture}(10284,5845)(1744,-6614)
\put(11611,-6541){\makebox(0,0)[lb]{\smash{\fontsize{12}{14.4}\usefont{T1}{ptm}{m}{n}{\color[rgb]{0,0,0}$U$}%
}}}
\put(10936,-6451){\makebox(0,0)[lb]{\smash{\fontsize{12}{14.4}\usefont{T1}{ptm}{m}{n}{\color[rgb]{0,0,0}$1$}%
}}}
\put(2071,-6451){\makebox(0,0)[lb]{\smash{\fontsize{12}{14.4}\usefont{T1}{ptm}{m}{n}{\color[rgb]{0,0,0}$0$}%
}}}
\put(2161,-1051){\makebox(0,0)[lb]{\smash{\fontsize{12}{14.4}\usefont{T1}{ptm}{m}{n}{\color[rgb]{0,0,0}$P$}%
}}}
\put(9361,-1501){\makebox(0,0)[lb]{\smash{\fontsize{12}{14.4}\usefont{T1}{ptm}{m}{n}{\color[rgb]{0,0,0}${\color{green} P^*(U)}$}%
}}}
\put(7201,-4786){\makebox(0,0)[lb]{\smash{\fontsize{12}{14.4}\usefont{T1}{ptm}{m}{n}{\color[rgb]{0,0,0}$P^2<Uf(U)$}%
}}}
\put(2701,-3526){\makebox(0,0)[lb]{\smash{\fontsize{12}{14.4}\usefont{T1}{ptm}{m}{n}{\color[rgb]{0,0,0}$P^2 > Uf(U)$}%
}}}
\put(5986,-6451){\makebox(0,0)[lb]{\smash{\fontsize{12}{14.4}\usefont{T1}{ptm}{m}{n}{\color[rgb]{0,0,0}$U_s(\beta)$}%
}}}
\put(3376,-6451){\makebox(0,0)[lb]{\smash{\fontsize{12}{14.4}\usefont{T1}{ptm}{m}{n}{\color[rgb]{0,0,0}$U^*$}%
}}}
\put(3871,-6451){\makebox(0,0)[lb]{\smash{\fontsize{12}{14.4}\usefont{T1}{ptm}{m}{n}{\color[rgb]{0,0,0}$U_u(\beta)$}%
}}}
\put(4546,-5776){\makebox(0,0)[lb]{\smash{\fontsize{12}{14.4}\usefont{T1}{ptm}{m}{n}{\color[rgb]{0,0,0}$\Gamma_u(\beta,U)$}%
}}}
\put(3871,-1411){\makebox(0,0)[lb]{\smash{\fontsize{12}{14.4}\usefont{T1}{ptm}{m}{n}{\color[rgb]{0,0,0}$\Gamma_s(\beta,U)$}%
}}}
\put(4051,-3571){\makebox(0,0)[lb]{\smash{\fontsize{12}{14.4}\usefont{T1}{ptm}{m}{n}{\color[rgb]{0,0,1}${\color{blue} \beta = \beta^*}$}%
}}}
\put(2566,-6091){\makebox(0,0)[lb]{\smash{\fontsize{12}{14.4}\usefont{T1}{ptm}{m}{n}{\color[rgb]{1,0,0}${\color{red} \beta \geq \beta^{**}}$}%
}}}
\put(3016,-5281){\makebox(0,0)[lb]{\smash{\fontsize{12}{14.4}\usefont{T1}{ptm}{m}{n}{\color[rgb]{0,0,0}$\beta^* < \beta < \beta^{**}$}%
}}}
\end{picture}%

%% file: diagram_Gamma_F.pdf_t
\begin{picture}(0,0)%
\includegraphics{diagram_Gamma_F.pdf}%
\end{picture}%
\setlength{\unitlength}{4144sp}%
\begin{picture}(10512,6340)(301,-5804)
\put(721,-1906){\makebox(0,0)[lb]{\smash{\fontsize{12}{14.4}\usefont{T1}{ptm}{m}{n}{\color[rgb]{0,0,1}$\Gamma_u = P^*$}%
}}}
\put(5086,-2716){\makebox(0,0)[lb]{\smash{\fontsize{12}{14.4}\usefont{T1}{ptm}{m}{n}{\color[rgb]{0,0,0}$1$}%
}}}
\put(5266,-3031){\makebox(0,0)[lb]{\smash{\fontsize{12}{14.4}\usefont{T1}{ptm}{m}{n}{\color[rgb]{0,0,0}$U$}%
}}}
\put(10486,-5686){\makebox(0,0)[lb]{\smash{\fontsize{12}{14.4}\usefont{T1}{ptm}{m}{n}{\color[rgb]{0,0,0}$U$}%
}}}
\put(10621,-5281){\makebox(0,0)[lb]{\smash{\fontsize{12}{14.4}\usefont{T1}{ptm}{m}{n}{\color[rgb]{0,0,0}$1$}%
}}}
\put(6256,-5731){\makebox(0,0)[lb]{\smash{\fontsize{12}{14.4}\usefont{T1}{ptm}{m}{n}{\color[rgb]{0,0,0}$U^*$}%
}}}
\put(9361,-781){\makebox(0,0)[lb]{\smash{\fontsize{12}{14.4}\usefont{T1}{ptm}{m}{n}{\color[rgb]{0,0,0}$f(U)$}%
}}}
\put(9406,-1996){\makebox(0,0)[lb]{\smash{\fontsize{12}{14.4}\usefont{T1}{ptm}{m}{n}{\color[rgb]{0,.69,0}$\Gamma_s$}%
}}}
\put(7111,-3976){\makebox(0,0)[lb]{\smash{\fontsize{12}{14.4}\usefont{T1}{ptm}{m}{n}{\color[rgb]{0,0,1}$\Gamma_u$}%
}}}
\put(6616,-4921){\makebox(0,0)[lb]{\smash{\fontsize{12}{14.4}\usefont{T1}{ptm}{m}{n}{\color[rgb]{1,0,0}$F(U)+\beta=0$}%
}}}
\put(316,-3616){\makebox(0,0)[lb]{\smash{\fontsize{12}{14.4}\usefont{T1}{ptm}{m}{n}{\color[rgb]{0,0,0}$\beta^*$}%
}}}
\put(316,-2986){\makebox(0,0)[lb]{\smash{\fontsize{12}{14.4}\usefont{T1}{ptm}{m}{n}{\color[rgb]{0,0,0}$U^*$}%
}}}
\put(316,-1231){\makebox(0,0)[lb]{\smash{\fontsize{12}{14.4}\usefont{T1}{ptm}{m}{n}{\color[rgb]{0,0,0}$\beta^{**}$}%
}}}
\put(406,209){\makebox(0,0)[lb]{\smash{\fontsize{12}{14.4}\usefont{T1}{ptm}{m}{n}{\color[rgb]{0,0,0}$\beta$}%
}}}
\put(1486,-5236){\makebox(0,0)[lb]{\smash{\fontsize{12}{14.4}\usefont{T1}{ptm}{m}{n}{\color[rgb]{0,0,0}The vector field is exiting $\{P \leq P^*\}$}%
}}}
\put(2971,-1591){\makebox(0,0)[lb]{\smash{\fontsize{12}{14.4}\usefont{T1}{ptm}{m}{n}{\color[rgb]{1,0,0}$-F(U)$}%
}}}
\put(4096,-1681){\makebox(0,0)[lb]{\smash{\fontsize{12}{14.4}\usefont{T1}{ptm}{m}{n}{\color[rgb]{0,.69,0}$\Gamma_s=P^*$}%
}}}
\put(1351,-331){\makebox(0,0)[lb]{\smash{\fontsize{12}{14.4}\usefont{T1}{ptm}{m}{n}{\color[rgb]{0,0,0}The vector field is entering $\{P \leq P^*\}$}%
}}}
\end{picture}%

%% file: cubic_hyperbol_perp_1.pdf_t
\begin{picture}(0,0)%
\includegraphics{cubic_hyperbol_perp_1.pdf}%
\end{picture}%
\setlength{\unitlength}{4144sp}%
\begin{picture}(5019,5305)(5794,-5804)
\put(10486,-5686){\makebox(0,0)[lb]{\smash{\fontsize{12}{14.4}\usefont{T1}{ptm}{m}{n}{\color[rgb]{0,0,0}$U$}%
}}}
\put(10621,-5281){\makebox(0,0)[lb]{\smash{\fontsize{12}{14.4}\usefont{T1}{ptm}{m}{n}{\color[rgb]{0,0,0}$1$}%
}}}
\put(6256,-5731){\makebox(0,0)[lb]{\smash{\fontsize{12}{14.4}\usefont{T1}{ptm}{m}{n}{\color[rgb]{0,0,0}$U^*$}%
}}}
\put(7066,-2896){\makebox(0,0)[lb]{\smash{\fontsize{12}{14.4}\usefont{T1}{ptm}{m}{n}{\color[rgb]{0,0,0}$f(U)$}%
}}}
\put(9046,-2131){\makebox(0,0)[lb]{\smash{\fontsize{12}{14.4}\usefont{T1}{ptm}{m}{n}{\color[rgb]{0,0,1}$f_\kappa(U)$}%
}}}
\put(7786,-5731){\makebox(0,0)[lb]{\smash{\fontsize{12}{14.4}\usefont{T1}{ptm}{m}{n}{\color[rgb]{0,0,0}$U_1$}%
}}}
\put(8191,-5731){\makebox(0,0)[lb]{\smash{\fontsize{12}{14.4}\usefont{T1}{ptm}{m}{n}{\color[rgb]{0,0,0}$U^-_I$}%
}}}
\put(8686,-5731){\makebox(0,0)[lb]{\smash{\fontsize{12}{14.4}\usefont{T1}{ptm}{m}{n}{\color[rgb]{0,0,0}$U^+_I$}%
}}}
\put(9631,-5731){\makebox(0,0)[lb]{\smash{\fontsize{12}{14.4}\usefont{T1}{ptm}{m}{n}{\color[rgb]{0,0,0}$U_2$}%
}}}
\end{picture}%

%% file: nearUstar.pdf_t
\begin{picture}(0,0)%
\includegraphics{nearUstar.pdf}%
\end{picture}%
\setlength{\unitlength}{4144sp}%
\begin{picture}(6282,5575)(3586,-7154)
\put(9516,-5110){\makebox(0,0)[lb]{\smash{\fontsize{12}{14.4}\usefont{T1}{ptm}{m}{n}{\color[rgb]{0,0,0}$U$}%
}}}
\put(6848,-1912){\makebox(0,0)[lb]{\smash{\fontsize{12}{14.4}\usefont{T1}{ptm}{m}{n}{\color[rgb]{0,0,0}$P$}%
}}}
\put(6794,-4764){\makebox(0,0)[lb]{\smash{\fontsize{12}{14.4}\usefont{T1}{ptm}{m}{n}{\color[rgb]{0,0,0}$U^*$}%
}}}
\put(5006,-5066){\makebox(0,0)[lb]{\smash{\fontsize{12}{14.4}\usefont{T1}{ptm}{m}{n}{\color[rgb]{0,0,0}$U^* - \underline{r}$}%
}}}
\put(6855,-4093){\makebox(0,0)[lb]{\smash{\fontsize{12}{14.4}\usefont{T1}{ptm}{m}{n}{\color[rgb]{0,0,0}$\Gamma_u(U^*)$}%
}}}
\put(5986,-2761){\makebox(0,0)[lb]{\smash{\fontsize{12}{14.4}\usefont{T1}{ptm}{m}{n}{\color[rgb]{0,0,1}$\Gamma_u(U)$}%
}}}
\put(8821,-5731){\makebox(0,0)[lb]{\smash{\fontsize{12}{14.4}\usefont{T1}{ptm}{m}{n}{\color[rgb]{0,0,0}$\sim \lambda_+(\beta,U^*)(U - U^*)$}%
}}}
\put(3601,-2356){\makebox(0,0)[lb]{\smash{\fontsize{12}{14.4}\usefont{T1}{ptm}{m}{n}{\color[rgb]{0,0,0}$\Gamma_u(U^*) + M^*_-(\beta,U^*)$}%
}}}
\put(8056,-7081){\makebox(0,0)[lb]{\smash{\fontsize{12}{14.4}\usefont{T1}{ptm}{m}{n}{\color[rgb]{0,0,0}$M^*_-(\beta,U) \sim \lambda_-(\beta,U^*) (U - U^*)$}%
}}}
\end{picture}%

%% file: sources.bib
@article {BressanChiri22,
    AUTHOR = {Bressan, A. and T. Chiri, M. and Salehi, N.},
     TITLE = {On the optimal control of propagation fronts},
   JOURNAL = {Math. Models Methods Appl. Sci},
  FJOURNAL = {Mathematical Models and Methods in Applied Sciences},
    VOLUME = {32},
      YEAR = {2022},
    NUMBER = {06},
     PAGES = {1109-1140},
}

@article{Broer_Tangerman_1986, title={From a differentiable to a real analytic perturbation theory, applications to the Kupka Smale theorems}, volume={6}, DOI={10.1017/S0143385700003540}, number={3}, journal={Ergodic Theory and Dynamical Systems}, author={Broer, H. W. and Tangerman, F. M.}, year={1986}, pages={345–362}}

@article {BressanChiri23,
    AUTHOR = {Bressan, A. and T. Chiri, M. and Salehi, N.},
     TITLE = {Optimal control of moving sets},
   JOURNAL = {J. Differential Equation},
  FJOURNAL = {Journal of Differential Equations},
    VOLUME = {361},
      YEAR = {2023},

     PAGES = {97-137},
}

@book {Fife,
    AUTHOR = {C. Fife, P},
     TITLE = {Mathematical Aspects of Reacting and Diffusing Systems},
    SERIES = {Springer Lecture Notes in Biomathematics},
 PUBLISHER = {Springer},
      YEAR = {1979}
}

@book{katokhass,
    author = {Katok, A. and Hasselblatt, B.},
    title = {Introduction to the Modern Theory of Dynamical Systems},
    publisher = {Cambridge University Press},
    year = {1997} 
}

@techreport{bressancetraro,
    author = {Bressan, A.},
    title = {Control Problems for Traveling Fronts and Moving Sets I - Optimal control of traveling profiles},
    institution = {Cetraro},
    year = {2024},
    url = {https://bpb-us-e1.wpmucdn.com/sites.psu.edu/dist/c/168431/files/2025/05/setcontrol-Cetraro-tk24-1.pdf}
}

@article{sard_origi,
    author = {Sard, A.},
    title = {The measure of the critical values of differentiable maps},
    journal = {Bull. Am. Math. Soc.},
    number = {48},
    year = {1942},
    pages = {883-890}
}

@book{coudeneerne_ergodic,
    author = {Yves Coud\`ene and Reinie Ern\'e},
    title = {Ergodic Theory and Dynamical Systems},
    publisher = {Springer},
    year = {2016} 
}

@book{royden2010real,
  author    = {Royden, H. L. and Fitzpatrick, Patrick M.},
  title     = {Real Analysis},
  edition   = {4th},
  publisher = {Pearson},
  year      = {2010},
  address   = {Boston},
}

@article{anitacapasso2019,
author = {Aniţa, Sebastian and Capasso, Vincenzo and Dimitriu, Gabriel},
title = {Regional control for a spatially structured malaria model},
journal = {Mathematical Methods in the Applied Sciences},
volume = {42},
number = {8},
pages = {2909-2933},
year = {2019}
}

@article{doi:10.1137/16M1061886,
author = {Coclite, G. M. and Garavello, M.},
title = {A Time-Dependent Optimal Harvesting Problem with Measure-Valued Solutions},
journal = {SIAM Journal on Control and Optimization},
volume = {55},
number = {2},
pages = {913-935},
year = {2017}
}

@article{Ania2019GlobalEF,
  title={Global eradication for spatially structured populations by regional control},
  author={Sebastian Aniţa and Vincenzo Capasso and Ana Maria Moşneagu},
  journal={Discrete \& Continuous Dynamical Systems - B},
  year={2019},
  url={https://api.semanticscholar.org/CorpusID:125375702}
}

@article{coclitegaraspino2018,
title = {Optimal strategies for a time-dependent harvesting problem},
journal = {Discrete and Continuous Dynamical Systems - S},
volume = {11},
number = {5},
pages = {865-900},
year = {2018},
%doi = {10.3934/dcdss.2018053},
url = {https://www.aimsciences.org/article/id/515fbce4-3fcb-4f96-ab8b-0f0af670ff98},
author = {Giuseppe Maria Coclite and Mauro Garavello and Laura V. Spinolo},
keywords = {Optimal control, differential games, measured-valued solutions, fish harvesting}
}

@book{coron2007,
    author = {J. M. Coron},
    title = {Control and Nonlinearity},
    publisher = {AMS},
    year = {2007}
}

@article{seirin2013,
    author = {L. Seirin and R. Baker and E. Gaffney and S. White},
    title = {Optimal barrier zones for stopping the invasion of Aedes aegypti mosquitoes via transgenic or sterile insect techniques},
    journal = {Theoretical Ecology},
    number = {6},
    year = {2013}
}

@article{faucris.242204001,
 author = {Ruiz-Balet, Domènec and Zuazua Iriondo, Enrique},
 journal = {Journal De Mathematiques Pures Et Appliquees},
 title = {{Control} under constraints for multi-dimensional reaction-diffusion monostable and bistable equations},
 year = {2020}
}

@article{Lenhart2001OPTIMALCO,
  title={Optimal Control of Harvesting in a Parabolic System Modelin Two Subpopulations},
  author={Suzanne M. Lenhart and J. A. Montero},
  journal={Mathematical Models and Methods in Applied Sciences},
  year={2001},
  volume={11},
  pages={1129-1141},
  url={https://api.semanticscholar.org/CorpusID:15902467}
}

@misc{bianchini2026travelingprofilescontrolcost,
      title={Traveling profiles and control cost for a PDE describing the evolution of invasive species}, 
      author={Stefano Bianchini and Chiara Trifone},
      year={2026},
      eprint={2606.23979},
      archivePrefix={arXiv},
      primaryClass={math.OC},
      url={https://arxiv.org/abs/2606.23979}, 
}
